\documentclass[reqno,10pt]{amsart}

\usepackage{amsmath}
\usepackage{amsfonts}
\usepackage{amssymb}
\usepackage{amsthm}
\usepackage{newlfont}
\usepackage{graphicx}
\newtheorem{thm}{Theorem}[section]

\newtheorem{lem}[thm]{Lemma}
\newtheorem{prop}[thm]{Proposition}

\theoremstyle{definition}

\theoremstyle{remark}
\newtheorem{rem}[thm]{Remark}
\numberwithin{equation}{section}

\newcommand{\R}{\mathbb R}

\newcommand{\op}[1]{\operatorname{#1}}
\newcommand{\ed}{\end {document}}

\newcommand{\spp}{\operatorname{supp}}

\title[Norm inflation in $C^m$]
{Strong illposedness of the incompressible Euler equation in integer $C^m$ spaces}
\author[J. Bourgain]{Jean Bourgain}
\address[J. Bourgain]{School of Mathematics, Institute for Advanced Study, Princeton, NJ 08544, USA}
\email{bourgain@math.ias.edu}

\author[D. Li]{Dong Li}
\address[D. Li]{Department of Mathematics, University of British Columbia, Vancouver BC Canada V6T 1Z2}%
\email{dli@math.ubc.ca}

\begin{document}
\begin{abstract}
We consider the $d$-dimensional incompressible Euler equations. We show strong illposedness of velocity in
any $C^m$ spaces whenever $m\ge 1$ is an \emph{integer}. More precisely, we show for a set of initial data dense in
the $C^m$ topology, the corresponding solutions lose $C^m$ regularity instantaneously in time.
In the $C^1$ case,
our proof is based on an anisotropic Lagrangian deformation and a short-time flow expansion.
In the $C^m$, $m\ge 2$ case, we introduce a flow decoupling method which allows to tame the nonlinear
flow almost as a passive transport. The proofs also cover illposedness in Lipschitz spaces $C^{m-1,1}$ whenever
$m\ge 1$ is an integer.
\end{abstract}

\maketitle

\section{Introduction}
In this work we consider the $d$-dimensional incompressible Euler equation posed on the whole space:
\begin{align} \label{V_usual}
 \begin{cases}
  \partial_t u + (u\cdot \nabla) u =- \nabla p, \quad (t,x) \in \mathbb R \times \mathbb R^d, \\
  \nabla \cdot u =0, \\
  u \bigr|_{t=0}=u_0,
 \end{cases}
\end{align}
where $u=u(t,x)=(u_1(t,x),\cdots,u_d(t,x)):\, \mathbb R \times \mathbb R^d \to \mathbb R^d$
denotes the velocity of the fluid and $p=p(t,x):\, \mathbb R \times \mathbb R^d \to \mathbb R$ is the pressure.   The first equation in \eqref{V_usual} is just the usual Newton's law:
the LHS  describes the acceleration of the fluid in Eulerian frame, whereas the RHS represents the force (we assume there is no external forcing here
for simplicity). The second equation in \eqref{V_usual} is the usual incompressibility (divergence-free)
condition.\footnote{In
this work, without explicit mentioning, we shall always assume the initial velocity $u_0$ is divergence-free.}
It can be also regarded
as a constraint through which the pressure term emerges as a Lagrangian multiplier. The system \eqref{V_usual} has $(d+1)$ unknowns
($u_1,\cdots,u_d$ and $p$), $(d+1)$ equations and thus is formally self-consistent. To reduce the complexity of the system,
a standard way is to eliminate the pressure term by projecting equation \eqref{V_usual} into
 the space of divergence-free
vector fields. Alternatively one may use vorticity formulation. For example in 2D, define $\omega= \nabla^{\perp}
\cdot u$ (here $\nabla^{\perp}=(-\partial_{x_2},\partial_{x_1})$), then equation \eqref{V_usual} becomes
\begin{align*}
\partial_t \omega + (u\cdot \nabla)\omega=0,
\end{align*}
where under some suitable assumptions $u$ is recovered from $\omega$ by the Biot-Savart law:
\begin{align*}
u& = \Delta^{-1} \nabla^{\perp} \omega \notag \\
 & = \frac 1 {2\pi} \int_{\mathbb R^2} \frac {y^{\perp}}{|y|^2} \omega(x-y)dy, \quad y^{\perp}=(-y_2,y_1).
 \end{align*}
Note in this vorticity form, it is evident that for smooth solutions
 $\|\omega\|_p$ is preserved in time for all $1\le p\le \infty$. The conservation of $\|\omega\|_{\infty}$ is the key
 reason for global wellposedness in 2D. Analogously in 3D, one can introduce
 $\omega= \nabla \times u$, and the vorticity equation takes the form:
\begin{align*}
\partial_t \omega + (u\cdot \nabla) \omega =(\omega \cdot \nabla) u,
\end{align*}
with (again under suitable regularity assumptions)
\begin{align*}
u = -\Delta^{-1} \nabla \times \omega.
\end{align*}
Compared with 2D, the vorticity stretching term $(\omega\cdot \nabla)u$ is the main obstruction to global wellposedness
in 3D. In general dimensions $d\ge 3$, one can introduce the vorticity matrix $\Omega=Du -(Du)^T$ (here and below
for any matrix $A$ we denote by $A^T$ its
 matrix transpose) or in component-wise form $\Omega_{ij} = \partial_j u_i -\partial_i u_j$. Then
the vorticity equation reads
\begin{align*}
\partial_t \Omega + (u\cdot \nabla) \Omega = - (Du)^T \Omega - \Omega Du.
\end{align*}
Analogous Biot-Savart laws holds between $\Omega$ and $Du$.
In this work we will be primarily concerned with the local wellposedness issues of \eqref{V_usual} in
physical dimensions $d=2,3$. The generalization to higher dimensions is straightforward with some modification
in numerology. Our main objective is to study the illposedness in critical/borderline function spaces.

The wellposedness theory for Euler equations has been extensively investigated in many different types of
function spaces, spatial-temporal domains and boundary conditions (especially so for local wellposedness).
We shall not attempt to give a complete survey here and refer the interested
readers to Majda-Bertozzi \cite{MB}, Chemin \cite{Chemin_book}, Bahouri-Chemin-Danchin \cite{BCD_book},
 Constantin \cite{C07} and the references therein
for a more comprehensive account.   The first group of results dates back (at least) to
Lichtenstein \cite{L21} and Gunther \cite{Gun27} who considered
  local wellposedness in H\"older spaces $C^{k,\alpha}$ ($k\ge 1$ is an integer and $0<\alpha<1$).
Global wellposedness of classical solutions (in H\"older spaces) for 2D Euler was obtained
by Wolibner \cite{Wolibner33}. By topologizing the space of differeomorphisms with
Sobolev $H^s$, $s>d/2+1$ norms, Ebin and Marsden \cite{EM70} obtained wellposedness of the Euler equation
on general compact manifolds possibly with $C^{\infty}$ boundary. The generalization of $H^s$ to $W^{s,p}$ with $s>d/p+1$
was obtained by Bourguignon and Brezis \cite{BB74}. In the Euclidean setting Kato  \cite{Kato72}
proved local wellposedness of $d$-dimensional Euler in the space $C_t^0 H_x^m$ for initial velocity
$u_0\in H^m(\mathbb R^d)$ with integer $m>d/2+1$. In a later work Kato and Ponce
\cite{KP88} removed the restriction that $m$ is an integer and proved wellposedness results in
the general Sobolev space $W^{s,p}(\mathbb R^d)$ with real $s>d/p+1$ and $1<p<\infty$. In order
to deal with non-integer $s$, Kato-Ponce \cite{KP88} proved the
following commutator estimate
 for the nonlocal differential operator $J^s=(1-\Delta)^{s/2}$, $s>0$:
\begin{align}
\| J^s(fg)-f J^s g\|_{p} \lesssim_{d,s,p} \|D f\|_{\infty}  \|
J^{s-1} g\|_{p} +\| J^s f\|_p \| g\|_{\infty}, \qquad 1<p<\infty.\,
 \label{commutator_1}
\end{align}
(For $p=\infty$ a version of  the $L^{\infty}$ end-point Kato-Ponce inequality (conjectured
in \cite{GMN}) and several new Kato-Ponce type inequalities are proved in recent \cite{BL13_KP}.)
The aforementioned Sobolev spaces $W^{s,p}(\mathbb R^d)$, $s>d/p+1$ are sometimes called subcritical/non-borderline spaces (for $d$-dimensional Euler)
and the space $W^{d/p+1,p}(\mathbb R^d)$ is called critical/borderline. The index $s_{d,p}=d/p+1$ is critical in the sense
that it is the minimal requirement for closing the energy estimate.
To see this consider for simplicity the space $H^s(\mathbb R^d)$. By performing
an energy estimate on \eqref{V_usual} (here we neglect the usual mollification/regularization arguments), one
obtains
\begin{align}
\frac d {dt}( \|u\|_{H^s}^2) \lesssim \| Du\|_{\infty} \| u\|_{H^s}^2. \label{energy_Hs_bound}
\end{align}
Note here for non-integer $s$, one has to use some version of the
Kato-Ponce commutator estimate to derive the estimate above. To
close the energy estimate, one must choose $s$ such that
$\|Du\|_{\infty} \lesssim \| u\|_{H^s}$. In view of Sobolev
embedding $H^{d/2+\epsilon}(\mathbb R^d) \hookrightarrow
L^{\infty}(\mathbb R^d)$, one quickly deduces the requirement
 $s>1+d/2$ for $p=2$. In a similar vein one can deduce $s>1+d/p$ for general $1<p<\infty$.

More refined results are available in Besov type spaces which can accommodate $L^{\infty}$ end-point embedding.
Vishik \cite{Vishik98} (see also \cite{Vishik99}) constructed global solutions to 2D Euler in
Besov space $B^{ 2/p+1}_{p,1} (\mathbb R^2) $ with $1<p<\infty$.
Note that in 2D the regularity index
$s=2/p+1$ is critical (albeit at the expense of $l^1$-summation in dyadic frequency blocks).
For general dimension $d\ge 2$, Chae \cite{Chae04} obtained local wellposedness of
 Euler in critical Besov space $B^{d/p+1}_{p,1}(\mathbb R^d)$
with $1<p<\infty$. For the case $p=\infty$ the unique local solvability in
$B^{1}_{\infty,1}(\mathbb R^d)$, $d\ge 2$ was proved by Pak and
Park in \cite{PP04} by using a compactness argument.
Very recently Pak and Park \cite{PP13} extended their analysis to cover the case $p=1$, i.e. the Besov
space $B^{d+1}_{1,1}(\mathbb R^d)$. It should be pointed out that an important inequality used in these works is
a composition estimate of the form (due to Vishik, see e.g. Theorem 4.2 in \cite{Vishik99} for the case
$B^0_{\infty,1}$; see  also Proposition 3.1 in \cite{PP13}): for $0\le s<1$, $1\le p\le \infty$, $f \in B^s_{p,1}$ and $g:\mathbb R^d \to \mathbb R^d$
a bi-Lipschitz volume-preserving map,
\begin{align*}
\| f \circ g^{-1} \|_{B^s_{p,1}} \lesssim_{s,p,d}  (1+ \log (\|g \|_{\op{Lip}} \| g^{-1} \|_{\op{Lip}} ) )
\| f\|_{B^s_{p,1}}.
\end{align*}
A more general wellposedness theorem which incorporates all the above (local) results can be found in the book
\cite{BCD_book} (see e.g. Theorem 7.1 on pp293 therein): for any $1\le p,q\le \infty$, $s\in \mathbb R$
such that  $B^s_{p,q}
\hookrightarrow C^{0,1}$ (i.e. $s>d/p+1$ or $s=d/p+1$ and $q=1$) and divergence-free initial data $u_0 \in B^s_{p,q}$,
 one can construct a local solution $u, \nabla p \in L^{\infty}([-T, T]; B^s_{p,q})$ with
 $T\ge \operatorname{const}/\|u_0\|_{B^s_{p,q}}$; furthermore
 \begin{align}
 u, \nabla p \in C_t^0 B^s_{p,q}, \qquad \text{for $q<\infty$}, \label{BCD_result_e1}
 \end{align}
 (
 for $q=\infty$ it is weakly continuous in time). It is worthwhile pointing out that, if one insists on
 having the critical regularity $s=d/p+1$, then one must appeal to the strong Besov space $B^{d/p+1}_{p,1}$.
 No other wellposedness results were known in $B^{d/p+1}_{p,q}$ for $1<q\le \infty$.
  Oversimplifying quite a bit, a common theme in the above mentioned results is the following:
 one finds a Banach space $X$ such that
\begin{itemize}
\item[a)] $X \hookrightarrow C^{0,1}$;
\item[b)] A version of the Kato-Ponce commutator estimate or a Lagrangian-commutator estimate holds in $X$.
\end{itemize}
Here in point b), by a Lagrangian-commutator we meant an estimate such as commuting singular integral operators
with a transport (composition) map. One can even try to recast these wellposedness results into
 an abstract formalism allowing as much generality as possible.
 This is done for example in recent \cite{Le12} using a scale of Banach spaces satisfying 9 hypotheses.

Despite the plethora of results in subcritical function spaces, the situation with borderline spaces such
as $H^{ d/2+1}(\mathbb R^d)$ for $d$-dimensional Euler had remained unclear until very recently.
In \cite{CW_log,DL_log}, wellposedness in critical $H^{d/2+1}$ spaces were proved for
some logarithmically regularized Euler equations. However the analysis therein relies heavily on the
logarithm regularization and has no bearing on the general case without logarithm.
 In \cite{Ta10}, Takada constructed\footnote{Takada also constructed in \cite{Ta10}
 counterexamples for the case $s<d/p+1$.}
 several counterexamples (involving divergence-free vector fields) of Kato-Ponce-type
commutator estimates in
Besov $B^{d/p+1}_{p,q}(\mathbb R^d)$ (resp. Triebel-Lizorkin spaces $F^{d/p+1}_{p,q}(\mathbb R^d)$) for
 $1\le p\le \infty$, $1< q\le \infty$ (resp. for Triebel-Lizorkin:
$1<p<\infty$, $1\le q\le \infty$ or $p=q=\infty$). These counterexamples suggest that one should perhaps
expect a negative solution to the wellposedness problem in borderline spaces. To put things into perspective, we
now mention a couple of earlier results in such flavor. In studying measure-valued solutions for 3D Euler,
DiPerna and Majda \cite{DM87} considered a shear flow of the explicit form:
\begin{align*}
u(t,x)=(f(x_2),0,g(x_1-tf(x_2))), \quad x=(x_1,x_2,x_3),
\end{align*}
where $f$ and $g$ are given single-variable functions. In the literature this flow is also called the two and a half-dimensional
flow and it solves \eqref{V_usual} with pressure $p=0$. DiPerna and Lions used the above shear flow example
(see e.g. p152 of \cite{Lions_book}) to show that for every $1\le p<\infty$, $T>0$, $M>0$, there exists a smooth
solution satisfying
$\| u(0)\|_{W^{1,p}(\mathbb T^3)}=1$ and $\| u(T)\|_{W^{1,p}(\mathbb T^3)}>M$.
Bardos and Titi \cite{BT10} recently used this example to construct a weak
solution which is in the space $C^{\alpha}$ initially but does not belong
to any $C^{\beta}$ for any $t>0$ and $1>\beta>\alpha^2$. By using a similar argument (see Remark 1 therein)
one can also prove illposedness in the spaces $F^{1}_{\infty,2}$ and $B^1_{\infty,\infty}$.
More recently in \cite{MY12}, Misio{\l}ek and Yoneda revisited this shear flow example
and showed illposedness of 3D Euler in
the logarithmic Lipschitz space $\operatorname{LL}_{\alpha}(\mathbb R^d)$, $0<\alpha\le 1$,
which consists of continuous functions
such that
\begin{align*}
\| f\|_{\operatorname{LL}_{\alpha}} = \| f \|_{\infty} + \sup_{0<|x-y|<\frac 12}
\frac{|f(x)-f(y)|}{|x-y| |\log |x-y||^{\alpha}}<\infty.
\end{align*}

In a recent paper \cite{CS10}, Cheskidov and Shvydkoy proved\footnote{For Navier-Stokes it was shown that the solution
operator is discontinuous at $t=0$ in the metric $B^{-1}_{\infty,\infty}$.}  the solution operator of
$d$-dimensional Euler cannot be continuous in Besov spaces
$B^s_{r,\infty}(\mathbb T^d)$ (here $\mathbb T^d$ is the periodic torus)
where $s>0$ if $2<r\le \infty$ and $s>d(\frac 2r-1)$ if $1\le r\le 2$. More precisely, for initial
$u_0$ of the form (below $\vec{e}_1=(1,0,\cdots)$, $\vec{e}_2=(0,1,\cdots)$ denote standard unit vectors in $\mathbb R^d$):
\begin{align*}
u_0(x_1,\cdots,x_d)= \vec{e}_1 \cos x_2 + \vec{e}_2 \sum_{q=0}^{\infty} 2^{-qs} \cos(2^q x_1),
\end{align*}
they proved that the corresponding weak solution must satisfy
\begin{align*}
\limsup_{t\to 0+} \| u(t) - u_0\|_{B^s_{r,\infty}(\mathbb T^d)} \ge \delta=\delta(d,r,s)>0.
\end{align*}
Note here this result confirms\footnote{Although the result in \cite{CS10} is stated for the periodic torus $\mathbb T^d$,
as remarked by Cheskidov-Shvydkoy, there is no essential difficulty in extending it
to the whole space case.} that the weak in time continuity (for $q=\infty$) mentioned in the paragraph
after formula \eqref{BCD_result_e1} is essentially optimal.

As was already mentioned, illposedness/wellposedness in borderline spaces such as $W^{d/p+1,p}(\mathbb R^d)$ had
remained unsolved until very recently. There are some exciting evidences that the illposedness in critical spaces
can now be resolved in full generality. In \cite{BL13,BL13b}, we introduced a new strategy and proved the following
\medskip

\noindent \textbf{Theorem \cite{BL13,BL13b}}. Let the dimension $d=2,3$. The Euler
equation \eqref{V_usual} is \emph{strongly illposed} in the Sobolev space
$W^{d/p+1,p}  $ for any $1\le p<\infty$ or the Besov space
$B^{d/p+1}_{p,q} $ for any $1\le p<\infty$, $1<q\le \infty$.
\medskip

Here the meaning of ``strong illposedness'' needs some clarification. To allow
some generality for later discussions, let us denote by $X$ a general Banach space with norm $\|\cdot \|_X$.
For example $X=W^{d/p+1,p} (\mathbb R^d)$ or $B^{d/p+1}_{p,q}(\mathbb R^d)$.
Roughly speaking, the type of strong illposedness results proved in \cite{BL13,BL13b}
is the following: for any given smooth initial data $u^{(g)}$ and any $\epsilon>0$, one can find a nearby data $ u_0$, such
that
\begin{align*}
 \| u_0-u^{(g)}\|_X <\epsilon,
\end{align*}
and the nonlinear solution $u$ corresponding to $ u_0$ has the property that
\begin{align*}
 \op{ess-sup}_{0<t<t_0} \| u(t)\|_X =+\infty,
\end{align*}
for any $t_0>0$. In yet other words, we prove that the inflation of $\|\cdot \|_X$ norm is dense,
and jumps to infinity instantaneously in time. More precise statements of the results can be found  in
\cite{BL13,BL13b}.

Broadly speaking, the scheme developed in \cite{BL13,BL13b} consists of three steps:
\begin{enumerate}
 \item[Step 1.] Creation of large Lagrangian deformation. That is to say we find initial data (with bounded critical norm) such
 that corresponding flow map has large distortion within some well-controlled time interval.
 \item[Step 2.] Local inflation of critical norm. In this step one performs a (modulated) high frequency perturbation around the initial
 data chosen in step 1 such that the corresponding critical norm is inflated through the large Lagrangian map. Here it is crucial that
 the flow map remains essentially unchanged under the high frequency perturbation.
 \item[Step 3a.] Patching in unbounded domains. In this step one repeats the local construction infinitely many times in patches which stay
 essentially disjoint from each other during the time of evolution. A key point used in this step is finite speed propagation and decay
 of the (Riesz-type) interaction kernel.
 \item[Step 3b.] Patching in compact domains. It is this step which requires a delicate analysis. To accommodate infinitely many patches
 in a compact domain, one have to analyze the fine interactions of these local patches. The whole procedure is arranged in such a way that each
 patch has a local (very short) patch time during which it interacts very weakly with other patches; moreover, the critical norm on this patch
 is inflated in the short patch time.
\end{enumerate}

While the analysis in \cite{BL13,BL13b} settles the illposedness in Sobolev spaces $W^{d/p+1,p}(\mathbb R^d)$ and similar Besov spaces, it
leaves completely open the end-point cases $p=\infty$. This includes spaces such as $C^1$, $C^{0,1}$ and the like. As a matter of fact, a long standing
open problem is the wellposedness in $C^m$ spaces with $m\ge 1$ being an integer (cf. \cite{KKP14}).
The situation is especially intriguing in view of
the wellposedness results in $C^{m,\alpha}$ for any $m\ge 1$, $0<\alpha<1$.

The purpose of this work is to completely settle the illposedness in the
end-point spaces such as $C^m$ and Lipschitz spaces $C^{m-1,1}$, $m\ge 1$. A very rough formulation of our results is the following
\medskip

\textbf{Theorem.} Let the dimension $d=2,3$. The Euler equation is strongly illposed in $C^{m}(\mathbb R^d)$, $C^{m-1,1}(\mathbb R^d)$
for any $m\ge 1$ being an integer.
\medskip

In particular, the Euler equation is \emph{illposed} in $C^{2}$ or $C^{1,1}$! This is perhaps a bit surprising since we do have wellposedness in
$C^{1,\alpha}$ for any $0<\alpha<1$.  To prove the above theorem, a very tempting thought is to follow directly the scheme
developed in \cite{BL13} and prove the inflation through large Lagrangian deformation with
high frequency perturbation. As it turns out, this is not the case and some work is
needed (especially so for the borderline case $C^1$).

We now give precise formulation of the main results. To avoid dealing with the issue of life-span of
smooth solutions in 3D,
we shall just consider perturbations around local solutions which are obtained by standard energy method.
The perturbations
will be done in such a way that the lifespan of the perturbed solution is not altered too much from the original one.
As such the norm inflation will
be produced strictly within the lifespan of the constructed solution. For clarity of presentation we introduce
 a simple notation. For any given initial velocity $u^{(g)} \in C_c^{\infty}(\mathbb R^d)$, we denote by
\begin{align}
 T_0=T_0(u^{(g)}) = \frac {c_d} {(1+\|u^{(g)} \|_{H^{d/2+2}(\mathbb R^d)} )} \label{lifespan_t0}
\end{align}
as the local life-span of the corresponding local solution to the Euler equation which is obtained
 by the standard energy method. Here in the above $c_d$ is a constant depending only on the dimension $d$,
 and we choose the norm $\| u^{(g)} \|_{H^{d/2+2}}$ just for simplicity.\footnote{Of course any $H^{d/2+1+\epsilon}$ norm
 will also work with a possibly different constant $c_{d,\epsilon}$.}

\begin{thm}[Non-compact case for $C^m(\mathbb R^d)$, $m\ge 1$, $d=2,3$]\label{thm1}
Let $m\ge 1$ be an \emph{integer}. Let the dimension $d=2$ or $3$.
For any given  velocity $u^{(g)}\in C_c^{\infty}(\mathbb R^d)$ and any $\epsilon>0$,
we can find a $C^{\infty}$ perturbation $u^{(p)}:\mathbb R^d\to \mathbb R^d$ such that the following hold true:

\begin{enumerate}
 \item $\| u^{(p)}\|_{L^1(\mathbb R^d)} + \|u^{(p)}\|_{H^m(\mathbb R^d)}+ \|  u^{(p)}\|_{C^m(\mathbb R^d)}
 <\epsilon$.
 \item Let $u_0=u^{(g)} + u^{(p)}$ and $T_1=T_0(u^{(g)})/2$ (see \eqref{lifespan_t0}).
 There exists a unique classical solution $u=u(t,x)$ to the  Euler equation
 \begin{align*}
  \begin{cases}
   \partial_t u + (u\cdot \nabla) u =-\nabla p, \quad 0<t\le T_1, \, x \in \mathbb R^d,\\
   \nabla \cdot u =0,\\
   u \Bigr|_{t=0} = u_0,
  \end{cases}
 \end{align*}
satisfying
\begin{align*}
 \max_{0\le t\le T_1}\Bigl( \| \omega(t,\cdot) \|_{L^1}+\|\omega(t,\cdot)\|_{C^{m-1}}
 \Bigr)<\infty.
\end{align*}
Here $\omega=\nabla \times u$ is the vorticity. Furthermore $u\in C_t^0 H_x^m$ and $u(t,\cdot) \in
C^{\infty}(\mathbb R^d)$ for each $0\le t\le T_1$.

\item For any $0<t_0 \le T_1$, we have
\begin{align} \label{thm1_e1}
 \operatorname{ess-sup}_{0<t \le t_0} \| u(t,\cdot) \|_{C^m(\mathbb R^d) } =+\infty.
\end{align}
More precisely, for any $n=1,2,3,\cdots$, there exist $0<t_n^1<t_n^2<\frac 1n$, and open balls
$B_n = B(x_n,1) \subset \mathbb
R^d$, such that $u(t,\cdot) \in C^{\infty}(B_n)$ for $t\in [t_n^1,t_n^2]$, but
\begin{align*}
\| u(t,\cdot)\|_{C^m(B_n)} >n, \quad\forall\, t \in [t_n^1,t_n^2].
\end{align*}

\end{enumerate}

\end{thm}

\begin{rem}
In 2D, since all smooth solutions are global, we can let $T_1=1$ or any finite positive constant. Similarly
in 3D, if one works with smooth axisymmetric without swirl slows, then one can also let $T_1=1$.
\end{rem}

\begin{rem}
Note that $C^m \subset C^{m-1,1}$ and for any $f\in C^m$, we have $\|f \|_{C^{m-1,1}} \sim \|f \|_{C^m}$.
therefore an immediate corollary of Theorem \ref{thm1} is the
strong illposedness in Lipschizt spaces $C^{m-1,1}$ for any integer $m\ge 1$.
Theorem \ref{thm2D_2Ck}--Theorem \ref{thm3D_1} below also has the corresponding Lipschitz space version. We omit such
statements which are essentially repetitions.
\end{rem}

\begin{rem}
In a recent preprint \cite{MY14}, Misio{\l}ek and Yoneda showed that the solution map for 2D Euler in $C^1$ norm cannot
be continuous. Their idea is to use the construction from \cite{BL13} to deduce a contradiction with a result of
Kato-Ponce \cite{KP86}.
In\footnote{We thank T.M. Elgindi for helpful communications.}\cite{EM}, Elgindi and Masmoudi consider the $C^1$ problem for $d$-dimensional Euler and some
closely related models. In particular by exploiting the velocity-pressure formulation
\eqref{intro_du_e1000a} (as opposed to the
vorticity formulation),
they prove that  one can find a smooth flow
such that $\|u(t=0)\|_{C^1}\ll 1$, yet
$\|u(t)\|_{C^1}$ becomes large in a short time. Note that with some work, their result combined with the patching argument
from \cite{BL13} can already be used to deduce a $C^1$-analogue of Theorem \ref{thm1}
which is already quite interesting. On the other hand, the above mentioned results do not cover
$C^m$ or $C^{m-1,1}$ for $m\ge 2$.
As was also mentioned in \cite{EM}, the compact domain case is a difficult
open problem. Our approach seems to be quite robust and work for all cases.
 See Theorem \ref{thm2D_2Ck}--\ref{thm3D_k} below for details.
\end{rem}

Our next result is concerned with the 2D case with compactly supported data. Needless to say, the results here can be easily
extended to cover the case of periodic boundary conditions. Due to the compactness of space, the solution
will have somewhat limited regularity.

\begin{thm}[2D $C^m$ case, $m\ge 1$ with compactly supported data]\label{thm2D_2Ck}
Let $k\ge 1$ be an integer and let $u^{(g)}\in C_c^{\infty}(\mathbb R^2)$ be any given velocity field.
 For
 any $\epsilon>0$, we can find a
perturbation $u^{(p)}:\mathbb R^2\to \mathbb R^2$ such that the
following hold true:

\begin{enumerate}
 \item $u^{(p)}$ is compactly supported in a ball of radius $\le 1$, $m$-times continuously differentiable and
 $$\| D^m u^{(p)} \|_{\infty}
 <\epsilon.$$

 \item Let $u_0= u^{(g)} + u^{(p)}$ and $\omega_0= \nabla^{\perp}\cdot u_0$. Consider the 2D Euler equation
(in vorticity form)
\begin{align*}
\begin{cases}
\partial_t \omega + (\Delta^{-1} \nabla^{\perp}\omega \cdot \nabla) \omega =0, \quad (t,x) \in \mathbb R \times
\mathbb R^2,\\
\omega\Bigr|_{t=0}=\omega_0.
\end{cases}
\end{align*}

There exists a unique time-global solution $\omega = \omega(t)$ to
the Euler equation satisfying $\omega(t) \in  L^{\infty} \cap \dot
H^{-1}$. Furthermore $\omega \in C_t^0 C_x^m$ and $u=\Delta^{-1}
\nabla^{\perp} \omega \in C_t^0 L_x^2 \cap C_t^0 C_x^{m-1,\alpha}$
for any $0<\alpha<1$.

\item $\omega(t)$ has additional local regularity in the following sense: there exists $x_* \in \mathbb R^2$ such
that for any $x\ne x_*$, there exists a neighborhood $N_x \ni x$,
$t_x >0$ such that $\omega(t, \cdot) \in C^{\infty} (N_x)$ for any
$0\le t \le t_x$.

\item For any $0<t_0 \le 1$, we have
\begin{align*}
 \operatorname{ess-sup}_{0<t \le t_0} \| D^m u(t,\cdot) \|_{\infty} =+\infty.
\end{align*}
More precisely, there exist $0<t_n^1<t_n^2 <\frac 1n$, open precompact sets $\Omega_n$, $n=1,2,3,\cdots$ such that
$u(t) \in C^{\infty}(\Omega_n)$ for all $0\le t \le t_n^2$, and
\begin{align*}
 \| D^m u(t,\cdot) \|_{L^\infty(\Omega_n)} >n, \quad \forall\, t\in[t_n^1,t_n^2].
\end{align*}

\end{enumerate}

\end{thm}

Our next two theorems are about illposedness in 3D with compactly supported data. As is well known, the lifespan
of smooth solutions corresponding to general smooth initial data is a major open problem. The usual
energy estimate (see \eqref{energy_Hs_bound}) easily implies that the solution is global if one can have
a priori control of $C^1$ norm of the solution. Due to this reason, the inflation of  the critical
$C^1$-norm is a bit delicate since one can sometimes confuse it with a likely finite time blowup.\footnote{At least
some work is needed to rule out this latter possibility.}

For simplicity we shall primarily work
with a class of special flows known as axisymmetric flows without swirl (see below). In some sense this is a matter
of taste: we want to separate the study of critical/borderline norm inflation (which is a local question) from
that of finite-time blowup/global in time regularity (i.e. a global question).
More concretely the reason of choosing axisymmetric without swirl flows
is two-fold: a) their properties are akin to 2D flows and more amenable to analysis; b) these flows generally generate
\emph{global} solutions for which lifespan is not an issue (as individual flows). Here in point b), we should stress
that in the study of $C^1$ norm inflation, there is still some issue with the control of  lifespan since we will
work with axisymmetric without swirl flows with low regularity.

As it turns out,
in our construction
there are some subtle technical differences between the
$C^1$ case and $C^m$, $m\ge 2$ case. In the $C^1$ case, the constructed solutions have limited regularity (but they
can be shown to be unique). Roughly
speaking, the $L^q$-norm of vorticity is allowed to diverge like $\log\log q$ as $q\to\infty$. In the $C^m$, $m\ge 2$
case, the solutions behave much better. In particular, the vorticity remains bounded in $C^{m-1}$ norm on the whole
time interval $[0,1]$.

For clarity of presentation, we shall state the results as two separate theorems. The reader should not be surprised
to find some repetition  in the statement of the results.
The first is about illposedness in $C^1$ in 3D with compactly supported  data.
For simplicity we consider vector fields on $\mathbb R^3$ with some symmetry. Let $u^{(g)} \in C_c^{\infty}
(\mathbb R^3)$ be axisymmetric without swirl, i.e.
\begin{align*}
u^{(g)}(x)=u^{(g)}(x_1,x_2,z)= u^{(g),r}(r,z) e_r+u^{(g),z}(r,z) e_z,\quad r=\sqrt{x_1^2+x_2^2},
\end{align*}
where $e_r=\frac 1 r (x_1,x_2,0)$, $e_z=(0,0,1)$.

\begin{thm}[3D $C^1$ case with compactly supported data]\label{thm3D_1}
Let $u^{(g)}\in C_c^{\infty}(\mathbb R^3)$ be any given velocity field which is axisymmetric without swirl.
For any
such $u^{(g)}$ and any $\epsilon>0$, we can find a
perturbation $u^{(p)}:\mathbb R^3\to \mathbb R^3$ such that the
following hold true:

\begin{enumerate}
 \item $u^{(p)}$ is axisymmetric without swirl, compactly supported in a ball of radius $\le 1$, continuously differentiable and
 $$\| D u^{(p)} \|_{\infty}
 <\epsilon.$$

 \item Let $u_0= u^{(g)} + u^{(p)}$ and $\omega_0= \nabla \times u_0$. Consider the 3D Euler equation
in vorticity form
\begin{align*}
\begin{cases}
\partial_t \omega + (u\cdot \nabla) \omega = (\omega \cdot \nabla) u, \quad (t,x) \in \mathbb R \times
\mathbb R^3,\\
u=-\Delta^{-1} \nabla \times \omega, \\
\omega\Bigr|_{t=0}=\omega_0.
\end{cases}
\end{align*}

There exists a unique local solution $\omega = \omega(t,x)$ to the Euler equation on the time interval
$[0,1]$ satisfying
\begin{align*}
\sup_{0\le t\le 1} \|\omega(t,\cdot) \|_{X_{\theta}} <\infty,
\end{align*}
where
\begin{align*}
\| \omega \|_{X_{\theta}} :=\sup_{2\le q<\infty} \frac{\|\omega\|_q} {\theta(q)}
\end{align*}
and $\theta(q)= \log\log (q+100)$. Furthermore $\omega \in C([0,1], X_{\theta})$ and
$u\in C_t^0 L_x^2 \cap C_t^0 C_x^{\alpha}$ for any $0<\alpha<1$.

\item $\omega(t)$ is compactly supported, i.e. for some constant $R_1>0$,
\begin{align*}
\op{supp}(\omega(t,\cdot)) \subset B(0,R_1), \quad\forall\, 0\le t\le 1.
\end{align*}

\item $\omega(t)$ has additional local regularity in the following sense: there exists $x_* \in \mathbb R^3$ such
that for any $x\ne x_*$, there exists a neighborhood $N_x \ni x$,
$t_x >0$ such that $\omega(t, \cdot) \in C^{\infty} (N_x)$ for any
$0\le t \le t_x$.

\item For any $0<t_0 \le 1$, we have
\begin{align*}
 \operatorname{ess-sup}_{0<t \le t_0} \| Du(t,\cdot) \|_{\infty} =+\infty.
\end{align*}
More precisely, there exist $0<t_n^1<t_n^2 <\frac 1n$, open precompact sets $\Omega_n$, $n=1,2,3,\cdots$ such that
$u(t) \in C^{\infty}(\Omega_n)$ for all $0\le t \le t_n^2$, and
\begin{align*}
 \| D u(t,\cdot) \|_{L^\infty(\Omega_n)} >n, \quad \forall\, t\in[t_n^1,t_n^2].
\end{align*}

\end{enumerate}

\end{thm}

\begin{rem}
We should point it out that the uniqueness of the constructed solution is not an issue in Theorem \ref{thm3D_1}.
In \cite{Yu63} Yudovich proved the existence and uniqueness of weak solutions to 2D Euler in bounded domains
for $L^{\infty}$ vorticity data. He then improved (see \cite{Yu95}) the uniqueness result (still for bounded domain
 in  dimensions $d\ge 2$) by allowing vorticty $\omega \in \cap_{p_0 \le p<\infty} L^p$ to grow like
 $\| \omega \|_p \le C \theta (p)$ with $\theta(p)$ increasing  slowly in $p$ (such as
$\theta(p)=\log p$). Here one should note that in the uniqueness result (for $d\ge 3$), the bound on $\|\omega(t)\|_p$ is assumed
to hold for all $0\le t\le T$ where $T>0$ is the lifespan of the solution under consideration.
In Theorem \ref{thm3D_1}, our constructed solution have $\|\omega\|_p$ which grows like $\log\log p$ for all
$0\le t\le 1$. Therefore uniqueness
is guaranteed by using Yudovich's result. For completeness we also mention that in the
Besov setting Vishik \cite{Vishik99} proved the uniqueness of weak solutions to
 Euler (in
$\mathbb R^d$, $d\ge 2$) under the following assumptions:
\begin{itemize}
\item  $\omega \in L^{p_0}$, $1<p_0<d$,
\item For some $a(k)>0$ with the property
\begin{align*}
\int_1^{\infty} \frac 1 {a(k)} d k =+\infty,
\end{align*}
it holds that
\begin{align*}
\Bigl|\sum_{j=2}^{k} \| P_{2^j} \omega \|_{\infty} \Bigr| \le \operatorname{const} \cdot a(k), \qquad \forall\, k\ge 4.
\end{align*}
\end{itemize}
Here $P_{2^j}$ is the Littlewood-Paley projector adapted to the frequency $|\xi| \sim 2^j$. Again for $d\ge 3$
the above bounds are assumed to hold for all $0\le t \le T$ (i.e. not just on initial data).

\end{rem}

We finally state the $3D$ $C^m$ case for $m\ge 2$.

\begin{thm}[3D  $C^m$ case, $m\ge 2$ with compactly supported data]\label{thm3D_k}
Let $k\ge 2$ be an \emph{integer}. Let $u^{(g)}\in C_c^{\infty}(\mathbb R^3)$ be any given
 velocity field which is axisymmetric without swirl. For any $\epsilon>0$, we can find a
perturbation $u^{(p)}:\mathbb R^3\to \mathbb R^3$ such that the
following hold true:

\begin{enumerate}
 \item $u^{(p)}$ is axisymmetric without swirl, compactly supported in a ball of radius $\le 1$, $m$-times continuously differentiable and
 $$\| D^m u^{(p)} \|_{\infty}
 <\epsilon.$$

 \item Let $u_0= u^{(g)} + u^{(p)}$ and $\omega_0= \nabla \times u_0$. Consider the 3D Euler equation
in vorticity form
\begin{align*}
\begin{cases}
\partial_t \omega + (u\cdot \nabla) \omega = (\omega \cdot \nabla) u, \quad (t,x) \in \mathbb R \times
\mathbb R^3,\\
u=-\Delta^{-1} \nabla \times \omega, \\
\omega\Bigr|_{t=0}=\omega_0.
\end{cases}
\end{align*}

There exists a unique local solution $\omega = \omega(t,x)$ to the Euler equation on the time interval
$[0,1]$ satisfying
\begin{align*}
\sup_{0\le t\le 1} (\|\omega(t,\cdot) \|_{2} +\| \omega(t,\cdot)\|_{{\infty}}) <\infty.
\end{align*}
Furthermore $\omega \in C_t^0 C_x^{m-1}$ and $u\in C_t^0 L_x^2 \cap
C_t^0 C_x^{m-1,\alpha}$ for any $0<\alpha<1$.

\item $\omega(t)$ is compactly supported, i.e. for some constant $R_1>0$,
\begin{align*}
\op{supp}(\omega(t,\cdot)) \subset B(0,R_1), \quad\forall\, 0\le t\le 1.
\end{align*}

\item $\omega$ has additional local regularity in the following sense: there exists $x_* \in \mathbb R^3$ such
that for any $x\ne x_*$, there exists a neighborhood $N_x \ni x$,
$t_x >0$ such that $\omega(t, \cdot) \in C^{\infty} (N_x)$ for any
$0\le t \le t_x$.

\item For any $0<t_0 \le 1$, we have
\begin{align*}
 \operatorname{ess-sup}_{0<t \le t_0} \| D^m u(t,\cdot) \|_{\infty} =+\infty.
\end{align*}
More precisely, there exist $0<t_n^1<t_n^2 <\frac 1n$, open precompact sets $\Omega_n$, $n=1,2,3,\cdots$ such that
$u(t) \in C^{\infty}(\Omega_n)$ for all $0\le t \le t_n^2$, and
\begin{align*}
 \| D^m u(t,\cdot) \|_{L^\infty(\Omega_n)} >n, \quad \forall\, t\in[t_n^1,t_n^2].
\end{align*}

\end{enumerate}

\end{thm}

We now give a brief overview of the proof and go over some key technical points.

\underline{The ``mechanism" of $C^m$ norm inflation}. We begin with
some heuristics. Consider the $C^1$ case. There are at least two
ways to see why the $C^1$ norm should inflate.

The first is through vorticity formulation. Consider for simplicity 2D Euler which reads
\begin{align*}
\begin{cases}
\partial_t \omega +(u\cdot \nabla) \omega=0,\\
u=\Delta^{-1} \nabla^{\perp} \omega,\\
\omega \Bigr|_{t=0} =\omega_0.
\end{cases}
\end{align*}
Define the characteristic line
\begin{align*}
\begin{cases}
\partial_t \phi(t,x) = u(t,\phi(t,x)), \\
\phi(0,x)=x.
\end{cases}
\end{align*}
Then $\omega(t) = \omega_0 \circ \phi(t)^{-1}$, where $\phi(t)^{-1}$ is the inverse map of $\phi(t)$. Since
$u(t)=\Delta^{-1} \nabla^{\perp} \omega(t)$, we then get
\begin{align*}
Du(t) = \Delta^{-1} \nabla^{\perp} D  ( \omega_0 \circ \phi(t)^{-1} ).
\end{align*}
Note that $\mathcal R=\Delta^{-1} \nabla^{\perp} D$ is a Riesz-type operator. One can even consider a slightly
more general expression:
\begin{align*}
X=\mathcal R ( \omega_0 \circ \phi),
\end{align*}
where $\phi$ is some general map. The assumption is that $\|Du_0\|_{\infty} \lesssim 1$,
$\mathcal R=\Delta^{-1}D \nabla^{\perp}$ (i.e. we do not work with more general Riesz operators) and one would like to
show $\|X\|_{\infty}$ is large. Note that the situation here is a bit delicate. For example if $\phi$ is the
identity map or any linear orthogonal transformation on $\mathbb R^2$, then since Laplacian commutes with any such
transformations, one gets
$$\|X\|_{\infty} \lesssim \|Du_0\|_{\infty} \lesssim 1.$$
Therefore to produce large $\|X\|_{\infty}$ norm, one must look for
it within a class of special $\phi$ such that the Riesz transform
does not act ``isotropically". A preliminary calculation (see
Proposition \ref{prop_A_nonzero})  shows that any non-orthogonal
(anisotropic) transformations can be used to produced large
$\|X\|_{\infty}$ norm (of course it must conspire with special
$\omega_0$ to achieve this). The main idea of the proof is to
``steer" the nonlinear Euler flow map such that it behaves
an-isotropically in the above sense.

The second way to see $C^1$ norm inflation is through the original
velocity formulation. Let $\partial$ be any one of the derivatives
$\partial_j$, $1\le j \le d$. Then taking $\partial$ on both sides
of \eqref{V_usual} gives
\begin{align*}
\partial_t ( \partial u ) + (u\cdot \nabla ) ( \partial u) + (\partial u \cdot \nabla)u
=- \partial \nabla p.
\end{align*}
Since $-\Delta p =\nabla \cdot \Bigl( ( u\cdot \nabla) u\Bigr)= Du \otimes Du$ (here $Du\otimes Du$ denotes
the summation $\sum_{i,j} \partial_i u_j \partial_j u_i$), we can solve for pressure and plug it into the above
to get

\begin{align} \label{intro_du_e1000a}
\partial_t ( \partial u )& + (u\cdot \nabla ) ( \partial u) + (\partial u \cdot \nabla)u \notag \\
&= \underbrace{\Delta^{-1} \nabla \partial}_{\mathcal R_{ij}}( Du \otimes Du).
\end{align}
Assume now\footnote{The key idea in \cite{BL13} is that instead of controlling some quantity $Y$ directly,
one can assume first it is bounded, and then work toward a contradiction or an estimate. This line of thought is also
exploited in the current work.}
 on some time interval $[0,T]$, one has $\|Du \|_{\infty} \lesssim 1$.  With this a priori assumption, one can then
 work toward a contradiction by using \eqref{intro_du_e1000a}. Indeed in \eqref{intro_du_e1000a}, the second
 term is transport which preserves $L^{\infty}$ and is harmless; the third terms is also OK since it is assumed
 to be bounded; the term on the RHS, however, has a Riesz transform in the front. Since $\mathcal R_{ij}$ is
 unbounded on $L^{\infty}$, this term can be used to produce large $L^{\infty}$ growth and hence $C^1$ norm inflation
 (of velocity).

We now sketch the main steps of the proof. In this short introduction we shall only explain the 2D case.
The 3D case is a little bit more involved especially in the $C^1$ case.

\texttt{Step 1a}: Local inflation of 2D $C^1$ norm. We take initial stream, velocity and vorticity in the form:
\begin{align}
 \psi_0(x)  & = \sum_{ j\le M} 2^{-2j} a(2^j x), \notag \\
  u_0(x) & = \sum_{j \le M} 2^{-j} (\nabla^{\perp} a)(2^j x), \notag \\
  \omega_0(x) & = \sum_{j \le M} ({\Delta a})(2^j x), \notag
\end{align}
where the function $a\in C_c^{\infty}(\mathbb R^2)$ is odd in the
$x_1$ and $x_2$ variables and is supported in $|x|\sim 1$. In some
sense the function $a(x)$ will be the main degree of freedom in our
construction. We then consider one of the entries of $(Du)(t,0)$
which is the quantity $(\partial_1 u_1)(t,0)$. Let $\phi=\phi(t,x)$
be the forward characteristic line:
\begin{align*}
\begin{cases}
\partial_t \phi(t,x) = u(t,\phi(t,x)),\\
\phi(0,x)=x.
\end{cases}
\end{align*}
Then since vorticity in 2D is transported along the characteristic line,  we get (after some algebra)
\begin{align*}
\partial_1 u_1(t,0) \sim \sum_{j\le M} \int_{\mathbb R^2} \Delta a(x) F(2^j \phi(t,2^{-j}x) ) dx,
\end{align*}
where $F(z_1,z_2)=z_1z_2/|z|^4$ is essentially the kernel of the Riesz transform
$\Delta^{-1}\partial_1 \partial_2$. We then make a short-time flow expansion (see Lemma \ref{lem_ga20}
and the derivations in Section \ref{sec:2DC1})
which gives for $\lambda =2^j$, $x \in \op{supp}(a)$,
\begin{align*}
\lambda \phi(t,\frac 1 {\lambda} x) & =x+ t \nabla^{\perp} a(x)  \notag \\
& \quad + \int_0^t (t-\tau) (\nabla p)(\tau, \frac x {\lambda} ) d\tau +\op{error}.
\end{align*}
It follows that
\begin{align*}
|(\partial_1 u_1)(t,0)| & \gtrsim Mt \underbrace{\int_{\mathbb R^2}
\Delta a(x) \nabla F(x)
\cdot \nabla^{\perp}a(x) dx}_{\op{I}_{\op{main}}} \notag \\
& \quad + \op{error}.
\end{align*}
It is possible to choose $a(x)$ such that $\op{I}_{\op{main}}$ does
not vanish. Thus the main part of $(\partial_1 u_1)(t,0)$ is roughly
of order $Mt$. However, as it turns out, in the ``error" term,
 the contribution of the pressure is not small unless
 we make a judicious choice of the function $a(x)$
 (see Lemma \ref{lem_gb10}).   Define $A= \| Du\|_{L_t^\infty L_x^{\infty}}$.
 With a little bit work we have
\begin{align*}
|\op{error}|\lesssim t^2 A^2 M +t^3 A^3 M^2   + t^4 M^3
+ t^6 A^6 M^3.
\end{align*}
Now discuss two cases. In both cases we shall choose
$M\gg 1$.  If $A\gg \log M$, then we are done (which is precisely inflation of $C^1$ norm); Otherwise
we have $A\lesssim \log M$.  Now choose $t=K/M$
with $1\ll K\ll M$, $KA \ll M$. This renders the error terms under control, and thus
\begin{align*}
|(\partial_1 u_1)(K/M,0)| \ge \op{const}
\cdot K \gg 1,
\end{align*}
 achieving the desired local $C^1$-norm inflation.

\texttt{Step 1b}: Local inflation of 2D $C^k$, $k\ge 2$ norm. Without loss of generality consider $k=2$.
We shall adopt a different strategy from the $C^1$ case. The idea is to do a high-low frequency splitting and
 ``decouple" the flow. Take initial stream
function in  the form
\begin{align*}
\psi_0(x) = \psi_0^{(l)}(x) + \psi_0^{(h)}(x),
\end{align*}
where the low frequency part $\psi_0^{(l)}$ will be used to ``steer"
the flow, and
\begin{align*}
\psi_0^{(h)}(x) = \sum_{M\le j\le M+\sqrt M}
2^{-3j} a_0(2^jx).
\end{align*}
Here the function $a_0 \in C_c^{\infty}(\mathbb R^2)$ is supported in $|x| \sim 1$. Define $\omega^{(l)}$, $\omega$ respectively as solutions to the following systems:
\begin{align*}
 \begin{cases}
  \partial_t \omega^{(l)} + (u^{(l)} \cdot \nabla) \omega^{(l)} =0, \\
  u^{(l)} = \Delta^{-1} \nabla^{\perp}\omega^{(l)}, \\
  \omega^{(l)}\Bigr|_{t=0}=\omega_0^{(l)}=\Delta \psi_0^{(l)};
 \end{cases}
\end{align*}
\begin{align*}
 \begin{cases}
  \partial_t \omega + (u\cdot \nabla) \omega=0, \\
  u= \Delta^{-1} \nabla^{\perp} \omega,\\
  \omega \Bigr|_{t=0}= \omega_0=\Delta \psi_0;
 \end{cases}
\end{align*}

Let $\tilde \omega$ solve the \emph{linear} system
\begin{align*}
 \begin{cases}
  \partial_t \tilde \omega + (u^{(l)} \cdot \nabla) \tilde \omega =0, \\
  \tilde \omega \Bigr|_{t=0}= \omega_0.
 \end{cases}
\end{align*}

Define $\tilde u= \Delta^{-1} \nabla^{\perp} \tilde \omega$. We then show that
(see Lemma \ref{lem_gd1}) the nonlinear flow can be decoupled:
\begin{align*}
\| D^2 u - D^2 \tilde u \|_{L_t^{\infty}L_x^{\infty}} \ll 1.
\end{align*}

Let $\tilde \omega^{(h)}$ solves
\begin{align*}
 \begin{cases}
  \partial_t \tilde \omega^{(h)} + (u^{(l)} \cdot \nabla) \tilde \omega^{(h)} =0, \\
  \tilde \omega^{(h)} \Bigr|_{t=0}= \omega_0^{(h)}.
 \end{cases}
\end{align*}

Denote $\tilde u^{(h)} = \Delta^{-1} \nabla^{\perp} \tilde \omega^{(h)}$. Then since
$\tilde \omega= \tilde \omega^{(h)} + \omega^{(l)}$, we get
\begin{align*}
\| D^2 u - D^2 \tilde u^{(h)} \|_{L_t^{\infty}L_x^{\infty}} \lesssim
1.
\end{align*}
(here we are regarding the flow $\omega^{(l)}$ as $O(1)$ in terms of estimates.)
Therefore as far as $C^2$ norm is concerned, we have
\begin{align*}
  & \text{nonlinear Euler flow with initial data $\omega_0^{(l)}+\omega_0^{(h)}$} \\
\approx &\; \text{linear flow driven by $u^{(l)}$ with $\omega_0^{(h)}$ as initial data}
\end{align*}
In the literature, this is often called a passive\footnote{In the
``active scalar" case, the transport will be self-induced. The 2D
Euler equation itself is already an example: $\partial_t \omega
+u\cdot \nabla \omega =0$ and $u$ is related to $\omega$ by the
relation $u=\Delta^{-1} \nabla^{\perp} \omega$.}
 scalar since the vorticity $\tilde \omega^{(h)}$ is transported by
a given external flow.

Define the flow map
\begin{align*}
\begin{cases}
\partial_t \phi^{(l)}(t,x) = u^{(l)}(t,\phi^{(l)}(t,x)), \\
\phi^{(l)}(0,x)=x.
\end{cases}
\end{align*}

To produce $C^2$-norm inflation, it suffices to give a good lower bound of
 the quantity (below $\tilde u^{(h)}=(\tilde u^{(h)}_1, \tilde u^{(h)}_2)$)
\begin{align*}
X(t)=(\partial_{11} \tilde u^{(h)}_2)(t,\phi^{(l)}(t,0)).
\end{align*}
As it turns out, for $M\gg 1$, one has
\begin{align*}
|X(t)| &\gtrsim \sqrt M \underbrace{\int_{\mathbb R^2} F_0( (D\phi^{(l)})(t,0) x) (\Delta a_0)(x) dx}_{I_{m}(t)} \notag \\
& \quad + \text{negligible error},
\end{align*}
where
\begin{align*}
F_0(x_1,x_2)=\frac{x_1^3-3x_1 x_2^2}{|x|^6}.
\end{align*}
We then choose a good low frequency data $\omega_0^{(l)}$ such that for $t=0+$, the matrix
 $D\phi^{(l)}(t,0)$ does not belong to $SO(2)$ (see Proposition \ref{prop_A_nonzero}).
 Accordingly we choose a good  function $a_0$
such  that
\begin{align*}
I_{m}(t) \gtrsim t.
\end{align*}
This in turn yields the desired local $C^2$ inflation since $M$ can be taken large.

\texttt{Step 2:} Patching for 2D. There are two cases: unbounded
case and compact domain case. The unbounded case is easy. One can
just revisit a patching argument already appeared in our earlier
work \cite{BL13} with some minor modifications. The nice feature of
the construction is that solution will be smooth in each patches. In
the compact domain case, we have to analyze the interactions between
patches. For this we have to build several auxiliary lemmas (see
Section \ref{sec_2D_patch}) to control the perturbation errors
whilst still producing norm inflation. For the $C^1$ case the main
building block is Lemma \ref{lem_ha1} which in some sense
streamlines the argument. On the other hand, patching for the $C^m$,
$m\ge 2$ case is easier in view of flow decoupling (see the end of
Section \ref{sec_2D_patch} for more details).

\subsection*{Acknowledgements}
J. Bourgain was supported in part by NSF No. DMS-1301619.
D. Li was supported  by an Nserc discovery grant.

\section{Notation and Preliminaries}
In this section we collect some notation and preliminary results used in this paper.


 For any function $f:\; \mathbb R^d\to
\mathbb R$, we use $\|f\|_{L^p}$ or sometimes $\|f\|_p$ to denote
the  usual Lebesgue $L^p$ norm  for $1 \le p \le
\infty$. For any vector valued function $f:\; \mathbb R^d \to \mathbb R^m$, we still
use the same notation $\|f\|_p$ or $\|f \|_{L^p(\mathbb R^d)}$ instead of $\|f\|_{L^p(\mathbb R^d)^m}$ to denote the sum of
$L^p$-norm of all its components.

For any vector-valued function $f:\, \mathbb R^d \to \mathbb R^k$, we shall denote by $Df$ the Jacobian matrix whose
entries are given by $(Df)_{ij}=\partial_j f_i$. We shall adopt the usual multi-index notation. For example, for
a multi-index $\beta=(\beta_1,\cdots,\beta_d)$, $\beta_i \ge 0$ and $\beta_i \in \mathbb Z$, we denote
$|\beta|=\beta_1+\cdots+\beta_d$, and
\begin{align*}
D^{\beta} = \partial_{x_1}^{\beta_1}\cdots \partial_{x_d}^{\beta_d}.
\end{align*}
Sometimes we write $\partial_{x_i}$ simply as $\partial_i$ whenever there is no obvious confusion.


 For any two quantities $X$ and $Y$, we denote $X \lesssim Y$ if
$X \le C Y$ for some harmless constant $C>0$. Similarly $X \gtrsim Y$ if $X
\ge CY$ for some $C>0$. We denote $X \sim Y$ if $X\lesssim Y$ and $Y
\lesssim X$. We shall write $X\lesssim_{Z_1,Z_2,\cdots, Z_k} Y$ if
$X \le CY$ and the constant $C$ depends on the quantities
$(Z_1,\cdots, Z_k)$. Similarly we define $\gtrsim_{Z_1,\cdots, Z_k}$
and $\sim_{Z_1,\cdots,Z_k}$. We shall use the notation $X\le C(\alpha_1,\cdots,\alpha_k) Y$ to
explicitly specify the dependence of the constant $C$ on the quantities/parameters $\alpha_1$,
$\cdots$, $\alpha_k$.


We shall denote by $X+$ any quantity of the form $X+\epsilon$ for any $\epsilon>0$.
 The notation $X-$ is similarly
defined. The notation $X+$ and $X-$ is particularly handy when
making some computations with ``critical/borderline" thresholds such
as exponents in some inequalities.


For any center $x_0 \in \R^d$ and radius $R>0$, we shall use the notation
 $B(x_0,R) := \{ x \in \R^d: |x-x_0| < R \}$ to
denote the open Euclidean ball. More generally for any set $A\subset \mathbb R^d$, we denote
\begin{align}
B(A,R):=\{y \in \R^d:\; |y-x|<R\text{ for some $x\in A$} \}. \label{def_ball}
\end{align}


For any integer $k\ge 0$ and any open set $U\subset \mathbb R^d$,
we use the notation $C^k_{\op{loc}}(U)$ to denote the set of continuous functions on $U$ whose  derivatives
$D^{\beta} u$, $|\beta|\le k$ are all continuous in $U$. Following the convention, the space $C^{\infty}(U)$ is defined
 as $C^{\infty}(U)= \bigcap_{k=1}^{\infty} C^k_{\op{loc}}(U)$. The space of smooth functions with compact support
in $U$ is denoted by $C_c^{\infty}(U)$. The Banach space $C^k(U)$ consists of functions
$u\in C^k_{\op{loc}}(U)$ with the norm
\begin{align*}
\| u\|_{C^k(U)}:= \sum_{|\beta|\le k} \| D^{\beta} u\|_{L^{\infty}(U) } <\infty.
\end{align*}
For any $0<\gamma\le 1$, the H\"older semi-norm $\|\cdot\|_{\dot C^{\gamma}}$ is defined by
\begin{align*}
\| u \|_{\dot C^{\gamma}(U)} := \sup_{\substack{x \ne y \\ x, y \in U}} \frac{|u(x)-u(y)|}{|x-y|^{\gamma}}.
\end{align*}
Also
\begin{align*}
&\| u \|_{C^{\gamma}(U)} := \| u\|_{L^{\infty}(U)} + \| u\|_{\dot C^{\gamma}(U)},\\
&\| u\|_{C^{k,\gamma}(U)} :=\| u\|_{C^{k}(U)} + \sum_{|\beta|=k} \| D^{\beta} u \|_{\dot C^{\gamma}(U)}.
\end{align*}
The H\"older space $C^{k,\gamma}(U)$ consists of functions $u\in C^k(U)$ with finite $C^{k,\gamma}$-norm.
Note that  the space $C^{k,1}$ is usually called Lipschitz. Sometimes the norm $\|\cdot\|_{C^{0,1}}$ is denoted
by $\|\cdot \|_{\op{Lip}}$.

In the introduction, we have mentioned a couple of results using Besov spaces. A convenient way to introduce these spaces is to use the Littlewood--Paley frequency projection operators. Let $\varphi(\xi)$
be a smooth bump function supported in the ball $|\xi| \leq 2$ and
equal to one on the ball $|\xi| \leq 1$. For any real number $N>0$ and $f\in \mathcal S^{\prime} (\mathbb R^d)$, define the frequency localized (LP) projection operators:
\begin{align*}
\widehat{P_{\le N }f}(\xi) &:=  \varphi(\xi/N )\hat f (\xi), \notag\\
\widehat{P_{> N }f}(\xi) &:=  [1-\varphi(\xi/N)]\hat f (\xi), \notag\\
\widehat{P_N f}(\xi) &:=  [\varphi(\xi/N) - \varphi (2 \xi /N )] \hat
f (\xi). 
\end{align*}
Similarly one can define $P_{<N}$, $P_{\geq N }$, and $P_{M < \cdot
\leq N} := P_{\leq N} - P_{\leq M}$, whenever $N>M>0$ are real numbers. We will usually
use these operators when $M$ and $N$ are dyadic numbers. The summation over $N$ or $M$ are
understood to be over dyadic numbers. Occasionally for convenience of notation we allow $M$ and $N$ not to be a power of
$2$.


For any $s\in \mathbb R$, $1\le p,q\le \infty$,  the homogeneous Besov seminorm $\|\cdot \|_{\dot B^s_{p,q}}$ is
defined  as
\begin{align*}
 \| f \|_{\dot B^{s}_{p,q}}
  := \begin{cases}
  \Bigl(\sum_{N>0}   N^{sq}\| P_N f \|_{L^p(\mathbb R^d)}^q \Bigr)^{\frac 1q}, \qquad \text{if $1\le q<\infty$,} \\
  \sup_{N>0} N^{s} \| P_N f \|_{L^p(\mathbb R^d)}, \qquad \text{if $q=\infty$.}
\end{cases}
\end{align*}
The inhomogeneous Besov norm $\|f \|_{B^{s}_{p,q}}$ of $f \in\mathcal S^{\prime}(\mathbb R^d)$ is
\begin{align*}
 \|f \|_{B^s_{p,q}} = \| f \|_p + \| f \|_{\dot B^s_{p,q}}.
\end{align*}

Sometimes the space $B^s_{\infty,\infty}(\mathbb R^d)$ is also called the
Zygmund space denoted by $C_*^s(\mathbb R^d) $. If $s=k+\gamma$, $k\ge 0$ is an integer, and $0<\gamma<1$, then
$C_*^s$ coincides with the H\"older space
$C^{k,\gamma}$ defined earlier.

In some parts of this paper (see the sections concerning 3D axisymmetric flows), we will need to
use Lorentz spaces. For convenience we recall the definitions here.
For a measurable function $f$, the nonincreasing rearrangement $f^*$ is defined by
\begin{align*}
f^*(t) = \inf \Bigl\{s:\, \operatorname{Leb}(x:\, |f(x)|>s) \le t \Bigr\}.
\end{align*}
For $1\le p,q<\infty$, the Lorentz space $L^{p,q}$ is the set of functions $f$ which satisfy
\begin{align*}
\| f \|_{L^{p,q}}:= \Bigl( \int_0^{\infty} (t^{\frac 1p} f^*(t) )^q \frac {dt} t\Bigr)^{\frac 1q}
<\infty.
\end{align*}
For $q=\infty$, $L^{p,\infty}$ is the set of functions such that
\begin{align*}
\| f \|_{L^{p,\infty}}=\sup_{t>0}t^{\frac 1p} f^*(t) <\infty.
\end{align*}
For $p=\infty$, we set $L^{\infty,q}=L^{\infty}$ for all $1\le q\le \infty$.
Note that $L^{p,p}=L^p$. For $1<p<\infty$, the space $L^{p,q}$ coincides with the real interpolation
from Lebesgue spaces.

Denote by $\op{SO}(d)$ the set of $d\times d$ orthogonal matrices with determinant $1$ (the matrix entries are
real-valued). It is  well-known that the usual Laplacian on $\mathbb R^d$ is invariant under any orthogonal
transformations (or put it slightly differently, it commutes with any orthogonal transformations). This property is highlighted
in the following proposition. For simplicity we just state the case dimension $d=2$.

\begin{prop} \label{prop_A_nonzero}
Let $K_0(x)= \log|x|$ for $0\ne x \in \mathbb R^2$. Fix any integer $m\ge 1$ and denote
$K_m(x) = \partial_{1}^m K_0(x)$. Consider any linear map $T:\, x \to Ax$ with $\det (A)=1$. If
$A\notin SO(2)$, then there exists $x_0 \in \mathbb R^2$ with $|x_0|=1$, such that
\begin{align*}
\Bigl( \Delta( K_m \circ T) \Bigr)(x_0) \ne 0.
\end{align*}
\end{prop}
\begin{rem}
Since $\Delta K_0 (x) =0$ for $x\ne 0$, easy to check that $\Delta ( K_m \circ T) =0$  (on $\mathbb R^2\setminus\{0\}$)
if $T$ is an orthogonal
transformation. The easiest case is of course the identity transformation.
\end{rem}

\begin{proof}[Proof of Proposition \ref{prop_A_nonzero}]
Denote
\begin{align*}
A= \begin{pmatrix}
a_{11} \quad a_{12} \\
a_{21} \quad a_{22}
\end{pmatrix}.
\end{align*}
Since $\det (A)=1$, easy to check that $T$ is bijective on $\mathbb R^2$.

Then a simple computation (by using the fact $\Delta K_m=0$ on $\mathbb R^2\setminus\{0\}$) gives
\begin{align*}
\Delta ( K_m \circ T) &= (a_{11}^2 +a_{12}^2 -a_{21}^2 -a_{22}^2) (\partial_{11} K_m) \circ T \notag \\
& \qquad + 2 (a_{11} a_{21} +a_{12} a_{22}) (\partial_{12} K_m) \circ T,
\end{align*}
where the above identity holds on $\mathbb R^2 \setminus\{0\}$.

By using the QR decomposition and the fact that Laplacian commutes with $Q\in SO(2)$, we can assume
$A$ is upper triangular, i.e. $a_{21}=0$. Since $\det (A)=1$, we can write $a_{11}=\lambda$ ($\lambda \ne 0$),
$a_{22}=\lambda^{-1}$, and thus
\begin{align*}
\Delta( K_m \circ T) &= (\lambda^2 +a_{12}^2 -\lambda^{-2}) (\partial_{11} K_m) \circ T \notag \\
&\qquad + 2 (a_{12} \cdot \lambda^{-1}) (\partial_{12} K_m) \circ T.
\end{align*}

Now discuss two cases.

Case 1: $a_{12}=0$. Then clearly $\lambda \ne 1$ (otherwise $A\in SO(2)$).

By a simple induction on $m$, it is easy to check that
\begin{align*}
\partial_{11} K_m = \frac{c_{m+2} x_1^{m+2} + x_2 g_{m+1} (x)} {|x|^{2(m+2)}},
\end{align*}
where $c_{m+2} \ne 0$ and $g_{m+1}$ is a polynomial homogeneous of degree $m+1$. Clearly $(\partial_{11} K_m)(x_1=1,x_2=0)
\ne 0$.

Case 2: $a_{12} \ne 0$. Then we only need to show that $\partial_{12} K_m$ does not completely vanish on
$\mathbb R^2 \setminus \{0\}$. Assume not, i.e. $\partial_{2} \partial_1 K_m \equiv 0$ on $\mathbb R^2 \setminus\{0\}$.
Then we get $\partial_1 K_m (x) = f(x_1)$, for some function $f$. Taking $x=(x_1,x_2)$, with $x_1$ fixed and $|x_2|\to\infty$,
we get $\partial_1 K_m (x) \to 0$ and hence $f(x_1) \equiv 0$. Therefore $\partial_1 K_m(x)\equiv 0$. This is
clearly impossible, since $\partial_1 K_m$ has the form
\begin{align*}
\partial_{1} K_m = \frac{c_{m+1} x_1^{m+1} + x_2 g_{m} (x)} {|x|^{2(m+1)}},
\end{align*}
where $c_{m+1} \ne 0$, and $g_m$ is a polynomial homogeneous of degree $m$.

Therefore in both cases, we can find nonzero vector $y_0$, such that $$\Bigl( \Delta(K_m \circ T) \Bigr)(y_0) \ne 0.$$
Letting $x_0=y_0/|y_0|$ then gives the result.
\end{proof}

\subsection{Local wellposedness of 3D Euler in $C^{1,\alpha}$, $0<\alpha<1$}
As was already mentioned in the introduction, the purpose of this paper is to show illposedness of
incompressible Euler in integer $C^k$ spaces. As such, it is fairly instructive to review the standard
local wellposedness theory in $C^{1,\alpha}$, $0<\alpha<1$ and understand the limitations of standard methods
 at the endpoints
$\alpha=0$ and $\alpha=1$. In this subsection we give a short review
of such a result in 3D. For a textbook account, one can see for example Chapter 4 of \cite{MB} (or
Chapter 7 of \cite{BCD_book} for a more detailed discussion in favor of Fourier analysis). The wellposedness theory
in \cite{MB} uses the (non-local) particle-trajectory method and assumes the initial vorticity $\omega_0$ is compactly
supported. Here we shall assume $u_0 \in C^{1,\alpha}(\mathbb R^3)$ and\footnote{In view of uniqueness results,
some form of integrability of $\omega_0$ is usually needed. Typically one assumes $\omega_0 \in L^p$ for some $1<p<d$.
Alternatively one can pose assumptions (at the level of velocity) for $u_0$ or make use of Besov spaces such as in \cite{BCD_book}. For example in \cite{PP04}
one has uniqueness
of solutions (for velocity) in $B^1_{\infty, 1}(\mathbb R^d)$.
 Here for simplicity of presentation
we just add the (natural) assumption $u_0 \in L^2(\mathbb R^3)$. Of course this assumption can be dropped.
}
 $u_0 \in L^2(\mathbb R^3)$. The presentation below is just a variation of the theme of standard arguments.
 Therefore it will be somewhat sketchy (but we hope it is self-contained and still conveys the main ideas).
In the whole proof, there are only two technical points: 1) Estimate of H\"older norm under a bi-Lipschitz
map; 2) Boundedness of Riesz-type singular integral operators in H\"older spaces.

We begin with a rather simple lemma. Note that
it holds for any $0\le \alpha \le 1$.

\begin{lem} \label{lem_preli_1}
Let $0\le \alpha \le 1$. Let $\phi:\,\mathbb R^d\to \mathbb R^d$ be
a bi-Lipschitz map such that $\det(D\phi)\equiv 1$. Then for any $f
\in C^{\alpha}(\mathbb R^d)$,  we have
\begin{align}
& \| f \circ \phi \|_{C^{\alpha}} \le
\|f \|_{C^{\alpha}} (1+\|  D\phi\|_{\infty}^{\alpha}), \label{lem_preli_1_e1} \\
& \| f \|_{C^{\alpha}} \lesssim_d \|f \circ \phi \|_{C^{\alpha}}
(1+\|D \phi\|_{\infty}^{(d-1)\alpha}).
\label{lem_preli_1_e2}
\end{align}
\end{lem}
\begin{proof}[Proof of Lemma \ref{lem_preli_1}]
The inequalities
obviously hold for $\alpha=0$.
So assume $0<\alpha\le 1$. Clearly \eqref{lem_preli_1_e1}
just follows from the
estimates:
\begin{align*}
&\| f \circ \phi\|_{\infty} \le \|f \|_{\infty}, \\
& \| f \circ \phi \|_{\dot C^{\alpha}}
\le \| f \|_{\dot C^{\alpha}} \| D
\phi \|_{\infty}^{\alpha}.
\end{align*}
For \eqref{lem_preli_1_e2}, one can use the identity
$f= f \circ \phi \circ \phi^{-1}$ to get,
\begin{align*}
\| f \|_{\dot C^{\alpha}}
\le \| f \circ \phi \|_{\dot C^{\alpha}}
\| D ( \phi^{-1} ) \|_{\infty}^{\alpha}.
\end{align*}
By using $\phi^{-1} \circ \phi = \op{Id}$ and differentiation, we have $\| D (\phi^{-1} ) \|_{\infty}
\le \| (D \phi)^{-1} \|_{\infty}$ (here $(D\phi)^{-1}$
is the matrix inverse of $D\phi$). Since $\det (D\phi)
\equiv 1$, we get (below $\op{adj}$ denotes the adjugate
matrix),
\begin{align*}
(D\phi)^{-1} = \op{adj}\bigl( (D\phi) \bigr).
\end{align*}
Thus
\begin{align*}
\| (D \phi)^{-1} \|_{\infty} \lesssim_d
\| D \phi \|_{\infty}^{d-1}.
\end{align*}
Therefore \eqref{lem_preli_1_e2} follows.

\end{proof}

\begin{thm}[Local wellposedness of 3D Euler in $C^{1,\alpha}$, $0<\alpha<1$] \label{preli_thm_3D_1}
Let $0<\alpha<1$. Consider the 3D Euler equation \eqref{V_usual}. Assume the initial velocity $u_0
\in L^2(\mathbb R^3) \cap C^{1,\alpha}(\mathbb R^3)$.
There there exist $T_0=T_0(\|u_0\|_{C^{1,\alpha}(\mathbb R^3)}+\|u_0\|_{L^2(\mathbb R^3)} )>0$ and a unique local solution
$u\in L^{\infty}_t([0,T_0];C^{1,\alpha}_x(\mathbb R^3)) \cap C_t^0([0,T_0];L_x^2(\mathbb R^3))$.
Also $\nabla p \in L^{\infty}_t([0,T_0];C^{1,\alpha}_x(\mathbb R^3)) \cap C_t^0([0,T_0];L_x^2(\mathbb R^3))$.
\end{thm}
\begin{rem}
It can be shown that $u(t)$ and $\nabla p(t)$ is weakly continuous
in time with values in $C_x^{1,\alpha}$. Of course we should point
it out that the  lifespan of the local solution is governed by
critical quantities such as $\int_0^t \|(Du)(s,\cdot)\|_{\infty} ds
$, $\int_0^t \|\omega(s,\cdot)\|_{\infty}ds $ and so on in the
spirit of the usual Beale-Kato-Majda criteria.
\end{rem}

\begin{proof}[Proof of Theorem \ref{preli_thm_3D_1}]
We proceed in several steps.

\texttt{Step 1}: Uniqueness. The uniqueness proof actually (almost) contains the contraction argument needed later. Hence we present it first. Let $u^{A}$ and
$u^B$ be two solutions corresponding to the same initial data $u_0$. Set $\eta= u^{A}-u^B$, and we have
\begin{align*}
\begin{cases}
\partial_t \eta + (\eta\cdot \nabla )u^A + (u^B\cdot
\nabla )\eta
=-\nabla(p^A-p^B), \\
\eta(0)=0.
\end{cases}
\end{align*}
Since $\|\nabla u^A\|_{\infty} \lesssim 1$,
an $L^2$ estimate
 on $\eta$ then gives
$\eta\equiv 0$.

\texttt{Step 2}: A priori estimate. Assume $W$ is a smooth
solution to
\begin{align*}
\begin{cases}
\partial_t W + (U\cdot \nabla )W = (W\cdot \nabla) U,\\
U=-\Delta^{-1} \nabla \times W,\\
W\Bigr|_{t=0}=W_0 \in H_x^{\infty}(\mathbb R^3) =
\bigcap_{l=0}^{\infty} H_x^l(\mathbb R^3).
\end{cases}
\end{align*}
Define the forward characteristics
\begin{align*}
\begin{cases}
\partial_t  \Phi(t,x) = U(t, \Phi(t,x)),\\
\Phi(0,x)=x.
\end{cases}
\end{align*}
Then  easy to check that
\begin{align} \label{lem_preli_1_e20a}
\| D \Phi(t) \|_{\infty} \le e^{\int_0^t \|DU(s)\|_{\infty} ds }.
\end{align}
On the characteristics,
\begin{align*}
W(t,\Phi(t,x))= W_0(x)+
\int_0^t (W\cdot \nabla U)(s, \Phi(s,x)) ds.
\end{align*}
By Lemma \ref{lem_preli_1},
\begin{align}
\|W(t)\|_{C^{\alpha}}
& \le (1+\|D\Phi(t)\|_{\infty}^{2\alpha})
\Bigl( \|W_0\|_{C^{\alpha}}  \notag \\
& \qquad + \int_0^t \|W\cdot \nabla U\|_{C^{\alpha}}
(1+\|D \Phi(s) \|_{\infty}^{2\alpha}) ds \Bigr).
\label{lem_preli_1_e20b}
\end{align}
By splitting into low and high frequencies and
the assumption $0<\alpha<1$, easy to
check that
\begin{align}
\|D U(t)\|_{C^{\alpha}} &\lesssim
\|W(t) \|_{C^{\alpha}} + \|U(t)\|_2 \notag \\
& \lesssim \|W(t)\|_{C^{\alpha}} + \|U_0\|_2.
\label{lem_preli_1_e20c}
\end{align}
Define
\begin{align*}
A(t) = e^{\int_0^t (3+\|W(s) \|_{C^{\alpha}} )^2 ds }.
\end{align*}
By \eqref{lem_preli_1_e20a} and \eqref{lem_preli_1_e20c}, one has
$\|D \Phi(t)\|_{\infty} \lesssim  A(t)$ (here we suppress the dependence of constants on $\|U_0\|_2$).
From \eqref{lem_preli_1_e20b}, we obtain
\begin{align*}
A^{\prime}(t) &\lesssim A(t) A(t)^{4\alpha}
\Bigl( \|W_0\|_{C^{\alpha}}  \notag \\
&\quad + \int_0^t (\|U_0\|_2+\|W\|_{C^{\alpha}})^2
ds (1+A(t)^{2\alpha}) \Bigr)^2 \notag \\
& \lesssim A(t)^{1+4\alpha}(
\|W_0\|_{C^{\alpha}}^2 +A(t)^{4\alpha+1}).
\end{align*}
Clearly this implies for some $T_0=T_0(\|W_0\|_2
+\|W_0\|_{C^{\alpha}})$, we have the estimate
\begin{align*}
\|W(t)\|_{C^{\alpha}} \lesssim \|W_0\|_{C^{\alpha}}
+\|W_0\|_2,\quad \forall
\, 0\le t\le T_0.
\end{align*}

\texttt{Step 3}: Mollification and contraction.
For each dyadic $N\ge 2$, let $u_0^{(N)}=P_{\le N} u_0$ ($P_{\le N}$ is the Littlewood-Paley projector
projected to frequency $|\xi|\lesssim N$)
and define $u^{(N)}$ to be the corresponding solution to the Euler equation.\footnote{Here we appeal to
the local wellposedness theory in Sobolev $H^m$ spaces. Easy to check that $u_0^{(N)} \in H^{\infty}_x$ and $u^{(N)}(t)$ is smooth.}  Denote $\omega^{(N)}
=\nabla \times u^{(N)}$. By using the estimates
in Step 2, we easily deduce that for some
$T_0=T_0(\|u_0\|_{C^{1,\alpha}}+ \|u_0\|_2)>0$,
$\omega^{(N)}$ have at least life span $[0,T_0]$ on which
\begin{align*}
\sup_{N\ge 2} \sup_{0\le t\le T_0} \|\omega^{(N)}
(t) \|_{C^{\alpha}} \lesssim \|u_0\|_{C^{1,\alpha}}
+ \| u_0\|_2.
\end{align*}
With the uniform estimates in hand, we can then perform a contraction argument
in $L_x^2$ (similar to Step 1) and show that $u^{(N)}$
is Cauchy in $C_t^0([0,T_0];L_x^2(\mathbb R^3))$ and converges to the limit solution $u$. The H\"older regularity of $u$ and $\nabla p$
can be easily checked (one just takes any two points
$x\ne y$ and send $N$ to infinity as usual). We omit
routine details.
\end{proof}

\begin{rem}
As is clear from the above proof, the assumption
$0<\alpha<1$ is only used in \eqref{lem_preli_1_e20c},
where we need to bound $DU$ in terms of the
vorticity $W$. This is a manifestation of the
unboundedness of Riesz-type operators in Lipschitz
spaces.
\end{rem}

\section{Estimate of the flow map for the $2D$ $C^1$ case}
Let
\begin{align*}
 \begin{cases}
  \partial_t \phi(t,x) = u(t,\phi(t,x)), \\
  \phi(0,x)=x,
 \end{cases}
\end{align*}
where $u$ is a smooth solution to
\begin{align*}
 \begin{cases}
  \partial_t \omega + (u\cdot \nabla) \omega=0, \\
  u=\Delta^{-1} \nabla^{\perp} \omega,\\
  \omega \Bigr|_{t=0} = \omega_0.
 \end{cases}
\end{align*}
Denote the initial velocity $u_0=\Delta^{-1} \nabla^{\perp} \omega_0$.
Assume on the time interval $[0,T_0]$, $T_0\le 1$, we have
\begin{itemize}
 \item $\max_{0\le t\le T_0} \|Du(t,\cdot)\|_{\infty} =A\ge 1$;
 \item $u(t,0)\equiv 0$;
 \item $\|\omega_0 \|_{\infty} +\|\omega_0\|_1\lesssim 1$.
\end{itemize}

Then

\begin{lem}[Rough Control of the flow map] \label{lem_ga10}
 For any $\lambda \ge 3$, $0\le t \le T_0$, $|x|\lesssim 1$, we have
 \begin{align*}
  &  |\lambda \phi(t,\frac 1 {\lambda} x) - x - t \lambda u_0(\frac 1 {\lambda } x ) | \notag \\
  \lesssim & \; A^2 t^2 e^{tA}  \log(3+ \|u_0\|_{H^3}).
 \end{align*}
\end{lem}

\begin{proof}[Proof of Lemma \ref{lem_ga10}]
By using the definition of the characteristic line, we have
\begin{align*}
 \phi(t,x) - x & = \int_0^t u(s,\phi(s,x)) ds \notag \\
 & = u_0(x) t + \int_0^t \int_0^s \partial_{\tau} (u(\tau,\phi(\tau,x))) d\tau ds \notag \\
 & = u_0(x) t + \int_0^t (t-\tau) \partial_{\tau} ( u(\tau, \phi(\tau,x) ) ) d\tau.
\end{align*}

Now since
\begin{align*}
 \partial_t u + (u\cdot \nabla) u = -\nabla p,
\end{align*}
we get
\begin{align*}
 \partial_t \bigl( u(t,\phi(t,x)) \bigr) = -(\nabla p)(t,\phi(t,x)).
\end{align*}
We now only need to estimate
\begin{align*}
 \| \lambda (\nabla p)(t, \phi(t,\frac x {\lambda})) \|_{L_x^{\infty}(|x|\lesssim 1)}.
\end{align*}
But this follows from Lemma \ref{lem_ga10a} below.

\end{proof}

 \begin{lem} \label{lem_ga10a}
 For all $\lambda \ge 3$, $0\le t \le T_0$,
  \begin{align*}
  \| \lambda (\nabla p) (t, \phi(t,\frac x {\lambda})) \|_{L_x^{\infty}(|x|\lesssim 1)}
   \lesssim A^2 e^{tA } \log(3+\|u_0\|_{H^3}).
  \end{align*}
 \end{lem}

\begin{proof}[Proof of Lemma \ref{lem_ga10a}]
 Easy to check $(\nabla p)(t,0) \equiv 0$ (this follows from $u(t,0)\equiv 0$). Then
 \begin{align}
  | (\nabla p(t, \phi(t,\frac x {\lambda})) | & \le | (\nabla p)(t,\phi(t,\frac x {\lambda})) -(\nabla p)(t,0)| \notag \\
  & \lesssim \| D^2 p(t,\cdot)\|_{L_x^{\infty}} \cdot |\phi(t,\frac x {\lambda})|. \label{ga100_1}
 \end{align}
Observe
\begin{align*}
 \lambda |\phi(t,\frac x {\lambda})| = \lambda |\phi(t,\frac x{\lambda}) -\phi(t,0)| \le \|D\phi\|_{\infty} |x|.
\end{align*}
Since (below $\operatorname{Id}$ denotes the identity matrix)
\begin{align} \label{dphi_eq1}
\begin{cases}
\partial_t D\phi = Du D\phi, \\
 D\phi(0,x)=\operatorname{Id},
 \end{cases}
\end{align}
easy to prove that
\begin{align*}
 \| D \phi(t)\|_{\infty} \le e^{t A}.
\end{align*}
Hence
\begin{align} \label{ga100_2}
 |\lambda \phi(t,\frac x {\lambda}) |_{L_x^{\infty}(|x|\le 1)} \le e^{tA}.
\end{align}

Now we estimate $\|D^2 p(t)\|_{\infty}$. Recall
\begin{align*}
 - \Delta p = \nabla \cdot ( (u\cdot \nabla) u ) = \partial_j u_k \partial_k u_j = O((\partial u)^2).
\end{align*}
Obviously (here $\mathcal R_{ij}$ denotes Riesz transform)
\begin{align*}
 D^2 p = \mathcal R_{ij} ( O((\partial u)^2)).
\end{align*}
Since by assumption $\| \partial u \|_{\infty} \lesssim A$, the usual log-interpolation inequality then gives
\begin{align}
 \| D^2 p\|_{\infty} & \lesssim A^2 \log(3+ \|u\|_{H^3}) \notag \\
 & \lesssim A^2 \log (3+ \|u_0\|_{H^3}). \label{ga100_3}
\end{align}
Plugging \eqref{ga100_3} and \eqref{ga100_2} into \eqref{ga100_1} then gives the result.

\end{proof}

\begin{lem}[Better Control of the flow map] \label{lem_ga20}
 For any $\lambda \ge 3$, $0\le t \le T_0$, $|x|\lesssim 1$, we have
 \begin{align*}
  &  |\lambda \phi(t,\frac 1 {\lambda} x) - x - t \lambda u_0(\frac 1 {\lambda } x )-\lambda \int_0^t (t-\tau)
  (\nabla p)(\tau, \frac {x} \lambda) d\tau | \notag \\
  \lesssim & t^3 A^3 e^{tA}\log(3+\|u_0\|_{H^3})
 \end{align*}
\end{lem}

\begin{proof}[Proof of Lemma \ref{lem_ga20}]
 We have
 \begin{align*}
   &\lambda | (\nabla p)(t, \phi(t,\frac x {\lambda}))- (\nabla p)(t, \frac x {\lambda})| \notag \\
  \lesssim & \|D^2 p \|_{\infty} | \lambda \phi(t, \frac {x} {\lambda}) -x|.
 \end{align*}

Now easy to check that (below $\operatorname{Id}$ denotes the identity matrix, also recall \eqref{dphi_eq1})
\begin{align*}
 &\| \lambda \phi(t, \frac x {\lambda}) -x \|_{L_x^{\infty}(|x| \lesssim 1)} \notag \\
\lesssim & \| (D\phi)(t, z)-\operatorname{Id}\|_{L_z^{\infty}(|z| \lesssim 1)} \notag \\
\lesssim & t\| \partial_t (D \phi) \|_{\infty} \notag \\
\lesssim & t\| Du (t)\|_{\infty} \|D\phi(t)\|_{\infty} \lesssim At e^{tA}.
\end{align*}

By \eqref{ga100_3},
\begin{align*}
 \| D^2 p\|_{\infty} \lesssim A^2 \log(3+\|u_0\|_{H^3}).
\end{align*}

Hence
 \begin{align*}
  &  |\lambda \phi(t,\frac 1 {\lambda} x) - x - t \lambda u_0(\frac 1 {\lambda } x )-\lambda \int_0^t (t-\tau)
  (\nabla p)(\tau, \frac {x} \lambda) d\tau | \notag \\
\lesssim & t^3 A^3 e^{tA}\log(3+\|u_0\|_{H^3}).
 \end{align*}

\end{proof}

\section{The main argument for local $2D$ $C^1$ norm inflation} \label{sec:2DC1}
Take initial data in the following form:
\begin{align}
 \psi_0(x)  & = \sum_{ 100 \le j\le M} 2^{-2j} a(2^j x), \notag \\
  u_0(x) & = \sum_{100 \le j \le M} 2^{-j} (\nabla^{\perp} a)(2^j x), \notag \\
  \omega_0(x) & = \sum_{100 \le j \le M} ({\Delta a})(2^j x),    \label{h1}
\end{align}
where $\psi_0$, $u_0$, $\omega_0$ are stream function, velocity and
vorticity respectively. The assumptions on $M$ and the function $a$
will be specified later. For the moment, we assume
$\operatorname{supp}(a) \subset\{x:\, \lambda_1 < |x|<\lambda_2 \}$
with $\lambda_2<2\lambda_1$, $\lambda_1\sim 1$, $\lambda_2 \sim 1$.
This is just to ensure that the functions $a(2^j x)$ have
non-overlapping supports. Also to simplify matters, assume $a(x)$ is
odd in $x_1$ and $x_2$(i.e. $a(x_1,x_2)=-a(-x_1,x_2)=-a(x_1,-x_2)$,
for any $x=(x_1,x_2)$). Easy to check\footnote{In later sections
(cf. Section 7), the symmetry assumption is removed by tracing the
flow of origin in time. Here for simplicity of presentation, we
consider the case that origin is a stagnation point.} that the
$u(t,0)\equiv 0$, i.e. the origin is a stagnation point. Also
$\phi(t,0)\equiv 0$.

We will work with the quantity $(\partial_1 u_1)(t,0)$. Recall that $\omega(t,\phi(t,x))= \omega_0(x)$,
$u=\Delta^{-1}\nabla^{\perp}\omega=(-\Delta^{-1} \partial_2 \omega, \Delta^{-1} \partial_1 \omega)$, therefore
\begin{align*}
 (\partial_1 u_1)(t,0) & = -(\Delta^{-1}\partial_1 \partial_2 \omega)(t,0) \\
 & = \frac 1 {\pi}\int_{\mathbb R^2} \omega(t,x) \frac{x_1 x_2} {|x|^4} dx \notag \\
 & = \frac 1 {\pi } \int_{\mathbb R^2} \omega_0(x) F(\phi(t,x)) dx,
\end{align*}
where we have denoted
\begin{align*}
 F(z) = \frac {z_1 z_2} { |z|^4}, \quad z=(z_1,z_2) \in \mathbb R^2.
\end{align*}

By using \eqref{h1}, we get (below we denote $\Delta a = b$ for simplicity)
\begin{align*}
 \pi (\partial_1 u_1)(t,0) &= \sum_{100\le j\le M} \int_{\mathbb R^2}  b(2^j x) F(\phi(t,x)) dx \notag \\
 & = \sum_{100 \le j\le M} \int_{\mathbb R^2} b(x) F(2^j \phi(t,2^{-j} x) ) dx.
\end{align*}
Here we have made a change of variable $x \to 2^{-j} x$ and absorbed the scaling factors into $F$.

Observe that, for $\lambda= 2^j$, $100\le j\le M$, $x \in \operatorname{supp}(a)$, we have by \eqref{h1} (since $u_0$ have
non-overlapping supports),
\begin{align*}
 \lambda u_0(\frac x {\lambda}) = (\nabla^{\perp} a)(x).
\end{align*}
Easy to check that $\log ( \|u_0\|_{H^3}) \sim M$.
Therefore by Lemma \ref{lem_ga20}, for $\lambda=2^j$, $x\in \operatorname{supp}(a)$,
\begin{align*}
 \lambda \phi(t,\frac 1 {\lambda} x) &= x + t \lambda u_0(\frac 1 {\lambda} x ) + \lambda \int_0^t
 (t-\tau) (\nabla p)(\tau, \frac x {\lambda}) d\tau \notag \\
 & \qquad + O(t^3 A^3 e^{tA} \log(3+\|u_0\|_{H^3})  ) \notag \\
 & =x + t (\nabla^{\perp} a)(x)+  \lambda \int_0^t
 (t-\tau) (\nabla p)(\tau, \frac x {\lambda}) d\tau \notag \\
 & \qquad + O(t^3A^3 e^{tA} M).
\end{align*}

We shall take $t = K/M$, where $1\ll K\ll M$ and $AK \ll M$. For such a small $t$, easy to check that $2^j \phi(t,2^{-j}x) \sim |x| \sim 1 $ for
$x\in \operatorname{supp}(a)$. Then Taylor expanding $F$ around the point $x$ (with $x \in \operatorname{supp}(a)$) gives
\begin{align*}
&F(\lambda \phi(t,\frac x {\lambda}))  \notag \\
=& F(x) + (\nabla F)(x) \cdot \bigl( t \nabla^{\perp} a(x) +
\lambda \int_0^t (t-\tau) (\nabla p)(\tau,\frac x {\lambda})
d\tau \notag \\
& \qquad + O(t^3 A^3 M  ) \bigr) \notag \\
& + O( \|D^2 F\|_{L_x^{\infty}{(|x|\sim 1)}} ) \cdot O(t^2 +(t^2 M)^2 + t^6 A^6 M^2 ).
\end{align*}

Clearly then
\begin{align}
  & \pi(\partial_1 u_1)(t,0) \notag \\
 = & \sum_{100\le j\le M} \int_{\mathbb R^2} b(x) F(x) dx \label{gb1} \\
 & + Mt \int_{\mathbb R^2} b(x) (\nabla F)(x) \cdot (\nabla^{\perp} a)(x) dx \label{gb2} \\
 & + \sum_{100\le j\le M} \int_{\mathbb R^2} b(x) (\nabla F)(x)\cdot \int_0^t (t-\tau) 2^j (\nabla p)(\tau, 2^{-j} x) d\tau \label{gb3} \\
 & + \operatorname{error}, \notag
\end{align}
where
\begin{align*}
\| \operatorname{error}\|_{\infty} \lesssim  t^3 A^3 M^2  + t^2 M + t^4 M^3
+ t^6 A^6 M^3.
\end{align*}

The contribution of the ``error'' term is negligible as long as we choose  $t= K/M$, $K\ll M$.

For the first term \eqref{gb1}, easy to check that it is zero since
$b(x)=\Delta a(x)$ and $F=\op{const}\cdot \partial_1\partial_2
(\log|x|)$.

We now show that the third term \eqref{gb3} has the bound
\begin{align*}
 \| \eqref{gb3} \|_{\infty} \lesssim t^2 A^2 M.
\end{align*}
The crucial point here is that we get $M$ instead of $M^2$.  This bound obviously follows from the following estimate:

\begin{lem} \label{lem_gb10}
Let $a(x)$ satisfy \eqref{gb_40b}--\eqref{gb_40d}, then
for any $0\le \tau \le t$,
\begin{align}
 |\sum_{100 \le j \le M} \int_{\mathbb R^2} b(x) (\nabla F)(x) \cdot 2^j (\nabla p)(\tau, 2^{-j}x) dx |
 \lesssim A^2 M.  \notag
\end{align}

\end{lem}

\begin{proof}[Proof of Lemma \ref{lem_gb10}]
We suppress the time dependence in this proof. For example we write $p(x)$ instead of $p(\tau,x)$.
Recall that
\begin{align*}
 \nabla p = \Delta^{-1} \nabla \nabla \cdot ( (u\cdot \nabla) u ).
\end{align*}
Since $\Delta^{-1} \partial_i \partial_j$ is a homogeneous operator which commutes with scaling, we get
\begin{align*}
 \sum_j 2^j (\nabla p)(2^{-j} x) &= \Delta^{-1} \nabla \nabla \cdot ( \sum_j 2^j u(2^{-j}x) \cdot (\nabla u) (2^{-j} x) ) \notag \\
 &= \Delta^{-1} \nabla \Bigl(  \underbrace{\sum_j  (\partial u) (2^{-j} x) (\partial u) (2^{-j} x)}_H  \Big).
\end{align*}
Note that by assumption $\|\partial u \|_{\infty} \lesssim A$, and hence
\begin{align*}
 \| H \|_{\infty} \lesssim M A^2.
\end{align*}

Now observe
  \begin{align*}
    &\sum_{100 \le j\le M} \int_{\mathbb R^2} b(x) (\nabla F)(x) \cdot 2^j (\nabla p)(\tau, 2^{-j}x) dx  \notag \\
   =  & \int_{\mathbb R^2} b(x) (\nabla F)(x) \cdot \Delta^{-1}\nabla H dx\notag \\
   = -& \int_{\mathbb R^2} \Delta^{-1}\nabla \cdot (b(x) (\nabla F)(x))  H dx.
  \end{align*}

We only need to show that $\Delta^{-1} \nabla \cdot ( b \nabla F) \in L^1_x(\mathbb R^2)$. Clearly for $|x| \lesssim 1$,
\begin{align*}
 \| \Delta^{-1} \nabla \cdot ( b \nabla F) \|_{L_x^1{(|x| \lesssim 1)}} \lesssim \| \Delta^{-1} \nabla \cdot ( b \nabla F) \|_{L_x^{\infty}}
 \lesssim 1.
\end{align*}
Only need to check the regime $|x| \gg 1$. Consider a general vector function $g=(g_1,g_2) \in C_c^{\infty}(\mathbb R^2)$, easy to check that
if
\begin{align}
&\int_{\mathbb R^2} g(y) dy=0,  \label{lem_gb10_e1}\\
&\int_{\mathbb R^2} y\cdot g(y)dy=0, \;\;
\int_{\mathbb R^2} (x\cdot y) (x\cdot g(y))dy =0, \qquad \forall\, x \in \mathbb R^2, \label{lem_gb10_e2}
\end{align}
then
\begin{align*}
 (\Delta^{-1} \nabla \cdot g)(x) = O(|x|^{-3}), \quad |x|\gg 1,
\end{align*}

Therefore we only need to check the conditions \eqref{lem_gb10_e1}--\eqref{lem_gb10_e2}. Note that condition \eqref{lem_gb10_e2} is equivalent to
the following
\begin{align}
 \int_{\mathbb R^2} y_1 g_1(y) dy = \int_{\mathbb R^2} y_2 g_2 (y) dy =0,  \label{lem_gb10_e4}\\
 \int_{\mathbb R^2} (y_1 g_2(y) + y_2 g_1(y)) dy =0.  \label{lem_gb10_e5}
\end{align}

Now recall $g(y)=  (\Delta a)(y) \nabla F(y)$, $ F(y) = \operatorname{const} \cdot \partial_1\partial_2 (\log |y|)$.
Easy to check \eqref{lem_gb10_e1} is always satisfied.  Then we just require
\begin{align*}
 &\int_{\mathbb R^2} (\Delta a)(y) y_1 \partial_1 F(y) dy =0, \\
 &\int_{\mathbb R^2} (\Delta a)(y) y_2\partial_2 F(y) dy =0,\\
 &\int_{\mathbb R^2} (\Delta a)(y) ( y_1 \partial_2 F(y) + y_2 \partial_1 F(y)) dy =0.
\end{align*}
Therefore the lemma is proved.
\end{proof}

\begin{lem}[Existence of the function $a(x)$] \label{lem_gb30}
Let $A$ be radial and $A \in C_c^{\infty}(x:\, \rho_1
<|x|<\rho_2<\infty )$ for some $\rho_1$, $\rho_2>0$. Let
\begin{align*}
a(x) = A(x) x_1^3 x_2.
\end{align*}
There exists $A(x)$ such that
\begin{align}
 &\int_{\mathbb R^2} (\Delta a)(x) (\nabla F)(x) \cdot \nabla^{\perp} a(x) dx  >0, \label{gb_40a}\\
 & \int_{\mathbb R^2} (\Delta a)(x) x_1 \partial_1 F(x) dx =0, \label{gb_40b}\\
 & \int_{\mathbb R^2} (\Delta a)(x) x_2 \partial_2 F(x) dx=0, \label{gb_40c}\\
 & \int_{\mathbb R^2} (\Delta a)(x) (x_1 \partial_2 F(x) + x_2 \partial_1 F(x) )dx =0. \label{gb_40d}
\end{align}
In fact, instead of \eqref{gb_40d}, $a(x)$ satisfies the stronger condition:
\begin{align}
\int_{\mathbb R^2} (\Delta a)(x) x_1 \partial_2 F(x) dx =\int_{\mathbb R^2}
(\Delta a)(x)  x_2 \partial_1 F(x) dx =0. \label{gb_40d_p}
\end{align}
\end{lem}
\begin{rem}
Of course one can choose other possible forms of $a$ to make
\eqref{gb_40a}--\eqref{gb_40d} hold. It is quite possible that a
simpler choice of $a(x)$ is available.
\end{rem}

\begin{proof}[Proof of Lemma \ref{lem_gb30}]
This is just a calculation. Denote $r=|x|$. Since $A$ is radial, we
shall write $A(r)$ and denote $A^{\prime}$, $A^{\prime\prime}$ the
radial derivatives. Then clearly
\begin{align*}
&\Delta a = ( A^{\prime\prime}+ \frac 9 r A^{\prime} ) x_1^3 x_2 + 6 A x_1 x_2, \\
&\partial_1 F = \frac {x_2} {r^6} (r^2-4x_1^2), \\
& \partial_2 F= \frac {x_1} {r^6} (r^2-4x_2^2),\\
& -\partial_2 a = - (A^\prime \cdot \frac{x_2^2} r x_1^3 + A x_1^3),\\
& \partial_1 a =  A^{\prime} \cdot \frac{x_1^4} r x_2 + 3 A x_1^2
x_2.
\end{align*}

Using parity, easy to check \eqref{gb_40d}  and \eqref{gb_40d_p}
hold.

In terms of $A(x)$, the conditions \eqref{gb_40b} and \eqref{gb_40c}
are
\begin{align*}
& \int_{\mathbb R^2} \Bigl( (A^{\prime\prime} + \frac 9 r A^{\prime}
)x_1^3 x_2 + 6 A x_1 x_2) \cdot x_1 \cdot \frac{x_2} {r^6}
(r^2-4x_1^2) dx_1 dx_2=0, \\
& \int_{\mathbb R^2} \Bigl( (A^{\prime\prime} + \frac 9 r A^{\prime}
)x_1^3 x_2 + 6 A x_1 x_2) \cdot x_2 \cdot \frac{x_1} {r^6}
(r^2-4x_2^2) dx_1 dx_2=0.
\end{align*}

By using polar coordinates and a short computation, we get the
equivalent conditions
\begin{align*}
& \int_0^{\infty} r (8A + 9r A^{\prime} + r^2 A^{\prime\prime}) dr
=0, \\
& \int_0^{\infty} r (24 A +9r A^{\prime} + r^2 A^{\prime\prime} )
dr=0.
\end{align*}
Integrating by parts in $r$, we get the equivalent condition
\begin{align}
\int_0^{\infty} A(r) rdr =0. \label{fvC_1}
\end{align}
This single condition guarantees that \eqref{gb_40b} and
\eqref{gb_40c} hold.

On the other hand,
\begin{align*}
 &\eqref{gb_40a}\\
  =& -\int_{\mathbb R^2}
 \Bigl( (A^{\prime\prime} + \frac 9 r A^{\prime}) x_1^3 x_2
 + 6 A x_1 x_2 \Bigr)\cdot( Ax_1^3+A^{\prime} \cdot
 \frac{x_2^2}r \cdot x_1^3) \cdot
 \frac{x_2(r^2-4x_1^2)} {r^6} dx \notag \\
 & + \int_{\mathbb R^2}
 \Bigl( (A^{\prime\prime} + \frac 9 r A^{\prime}) x_1^3 x_2
 + 6 A x_1 x_2 \Bigr)\cdot( 3Ax_1^2 x_2+A^{\prime} \cdot
 \frac{x_1}r \cdot x_1^3x_2) \cdot
 \frac{x_1(r^2-4x_2^2)} {r^6} dx. \notag
 \end{align*}

 A slightly involved
computation using polar coordinates yields that
\begin{align*}
\eqref{gb_40a}= \frac 1 {32} \pi \int_0^{\infty} r^4 (33 A
A^{\prime} + 9 (A^{\prime})^2 r + 3 A A^{\prime\prime} r +
A^{\prime} A^{\prime\prime} r^2) dr.
\end{align*}

After several integrating by parts, we get the equivalent condition
\begin{align}
\int_0^{\infty} (A^{\prime})^2 r^5 dr - 12 \int_0^{\infty} A^2 r^3
dr >0. \label{fvC_2}
\end{align}

We now only need to choose $A(x)$ such that \eqref{fvC_1} and
\eqref{fvC_2} hold. To this end, let $m \in C_c^{\infty}(\mathbb R)$
be  such that $\int m(x) dx =0$. Define
\begin{align*}
A(r) = \frac 1 r m( \frac {r-1} {\epsilon}).
\end{align*}
Easy to check that \eqref{fvC_1} hold. On the other hand for
$\epsilon$ small, we have
\begin{align*}
\int_0^{\infty} (A^{\prime})^2 r^5 dr \gtrsim \epsilon^{-1},
\end{align*}
whereas
\begin{align*}
\int_0^{\infty} A^2 r^3 dr \lesssim \epsilon.
\end{align*}

Easy to see that \eqref{fvC_2} holds for sufficiently small
$\epsilon$.

\end{proof}

Now denote $C_1=\frac {3\pi}2 \int_0^{\infty} \frac{(A^{\prime}(r))^2} {r}dr$ as in \eqref{gb_40a}. Then
\begin{align*}
 &\pi (\partial_1 u_1)(t,0) \notag \\
  \ge & Mt C_1 - O(t^2 A^2 M)  \notag \\
 & \quad - O(t^3 A^3 M^2  + t^2 M + t^4 M^3
+ t^6 A^6 M^3).
\end{align*}
Choosing $t=K/M$ with $1\ll K\ll M$, $ KA\ll M$, then obviously yields the bound
\begin{align*}
 (\partial_1 u_1)(K/M,0) \ge K C_1/2.
\end{align*}

This gives the desired inflation of $C^1$ norm.

\section{$2D$ $C^m$, $m\ge 2$ case: decoupling of the flow map} \label{sec:2DCm_flow}
In this section we will state and prove a basic flow decoupling lemma which is the key to
obtaining norm inflation for $C^m$ norms when $m\ge 2$. To simplify the presentation, we
shall just give the details for the case $m=2$, i.e. $C^2$ case. The general case $C^m$, $m\ge 3$
is a simple change of numerology and we leave it to interested readers.

Consider the following systems:
\begin{align*}
 \begin{cases}
  \partial_t \omega^{(l)} + (u^{(l)} \cdot \nabla) \omega^{(l)} =0, \\
  u^{(l)} = \Delta^{-1} \nabla^{\perp}\omega^{(l)}, \\
  \omega^{(l)}\Bigr|_{t=0}=\omega_0^{(l)};
 \end{cases}
\end{align*}

\begin{align*}
 \begin{cases}
  \partial_t \omega + (u\cdot \nabla) \omega=0, \\
  u= \Delta^{-1} \nabla^{\perp} \omega,\\
  \omega \Bigr|_{t=0}= \omega_0^{(l)} + \omega_0^{(h)};
 \end{cases}
\end{align*}

Let $\tilde \omega$ solve the \emph{linear} system
\begin{align*}
 \begin{cases}
  \partial_t \tilde \omega + (u^{(l)} \cdot \nabla) \tilde \omega =0, \\
  \tilde \omega \Bigr|_{t=0}= \omega_0=\omega_0^{(l)} + \omega_0^{(h)}.
 \end{cases}
\end{align*}

Let $\tilde u= \Delta^{-1} \nabla^{\perp} \tilde \omega $ be the velocity corresponding to $\tilde \omega$. Let
$u_0^{(l)}$, $u_0^{(h)}$ be the velocities corresponding to $\omega_0^{(l)}$, $\omega_0^{(h)}$ respectively.
We assume $0<T_0 \lesssim 1$, and
\begin{align*}
 \| u_0^{(l)}\|_2 + \| u_0^{(h)} \|_2 + \| \omega_0^{(l)} \|_{\infty} + \| \omega_0^{(h)} \|_{\infty} \lesssim 1.
\end{align*}
These conditions guarantee that on the time interval $[0,T_0]$,
\begin{align} \label{tmp_lemgd1_cond_Hk_e1}
 \|\tilde u(t) \|_{H^k}+\| u(t) \|_{H^k}
 \lesssim C_k (1+\| u_0^{(l)}\|_{H^k} + \|u_0^{(h)} \|_{H^k} ), \qquad k\ge 3,
\end{align}
where $C_k$ is some constant depending on $k$.

The following lemma gives quantitative estimate of the difference between $u$ and $\tilde u$. The advantage of this
simple lemma is that it gives us the flexibility to decouple the flow map.

\begin{lem} \label{lem_gd1}
\begin{align*}
 &\max_{0\le t \le T_0} \|D^2u(t,\cdot) -D^2 \tilde u (t,\cdot)\|_{\infty} \notag \\
& \lesssim
\bigl( \log(\|u_0^{(l)}\|_{H^4}+\|u_0^{(h)}\|_{H^4}+3 ) \bigr)^2
 e^{ \|Du^{(l)}\|_{L_t^{\infty}L_x^{\infty}([0,T_0])} T_0}
 (T_0
(1+\|D\omega_0\|_{\infty})
 )
  \| \omega_0^{(h)}\|_{\infty} \notag \\
 &+  \bigl( \log(\|u_0^{(l)}\|_{H^4}+\|u_0^{(h)}\|_{H^4}+3 ) \bigr)^2
 e^{ 2\|Du^{(l)}\|_{L_t^{\infty}L_x^{\infty}([0,T_0])} T_0}
 (T_0
(1+\|D\omega_0\|_{\infty})
 )^3
  \| u_0^{(h)}\|_{2} \notag \\
&+(\|D^2\omega_0 \|_{\infty}+3)\bigl( \log(\|u_0^{(l)}\|_{H^4}+\|u_0^{(h)}\|_{H^4}+3 ) \bigr)\notag \\
& \qquad \cdot e^{ 2\|Du^{(l)}\|_{L_t^{\infty}L_x^{\infty}([0,T_0])} T_0}
 (T_0
(1+\|D\omega_0\|_{\infty})
 )
  \| u_0^{(h)}\|_{2} \notag \\
 &+(\|D^2\omega_0 \|_{\infty}+3)\bigl( \log(\|u_0^{(l)}\|_{H^4}+\|u_0^{(h)}\|_{H^4}+3 ) \bigr) \notag \\
 & \qquad \cdot e^{ \frac 32\|Du^{(l)}\|_{L_t^{\infty}L_x^{\infty}([0,T_0])} T_0}
 T_0
  \| u_0^{(h)}\|_{2}^{\frac 12} \|\omega_0^{(h)}\|_2^{\frac 12}.
\end{align*}

\end{lem}

\begin{rem} \label{rem_gd1}
In our later construction, we shall take  $u_0^{(h)} = \nabla^{\perp} \psi_0^{(h)}$ roughly of the form
\begin{align*}
& \psi_0^{(h)}(x) \sim \sum_{M\le j\le M+\sqrt M} 2^{-3j} a_0(2^j x),\\
& u_0^{(h)} (x) \sim \sum_{M\le j\le M+\sqrt M} 2^{-2j} (\nabla^{\perp} a_0)(2^j x).
 \end{align*}
Then $\|D^2 \omega_0\|_{\infty} \lesssim 2^{M+\sqrt{M}}$,  $\log(\|u_0\|_{H^4}) \lesssim M$,
$\|u_0^{(h)}\|_2 \lesssim 2^{-3M}$, $\|\omega_0^{(h)}\|_2 \lesssim 2^{-2M}$. Clearly it gives
\begin{align*}
 \max_{0\le t\le T_0} \| D^2 u(t) -D^2 \tilde u(t,\cdot)\|_{\infty} \ll 1,
\end{align*}
 i.e. we can decouple the flow map. Note also that $\tilde u(t)=u^{(l)}(t) + {\tilde u}^{(h)}(t)$, where
 ${\tilde u}^{(h)} = \Delta^{-1} \nabla^{\perp} {\tilde \omega}^{(h)}$, and ${\tilde \omega}^{(h)}$ solves
 \begin{align*}
\begin{cases}
 \partial_t {\tilde \omega}^{(h)}(t) + u^{(l)}(t) \cdot \nabla {\tilde \omega}^{(h)}(t) =0,\\
 {\tilde \omega}^{(h)} \Bigr|_{t=0} = \omega^{(h)}_0.
 \end{cases}
 \end{align*}
In producing $C^2$-norm inflation, we only need to examine ${\tilde u}^{(h)}(t)$.
\end{rem}

\begin{rem}
In the proof of Lemma \ref{lem_gd1} below, we shall often (sometimes without explicit mentioning)
make use of the following simple (maximum principle)
estimate:
namely if $u$ is a smooth solution to the linear equation
\begin{align*}
\begin{cases}
\partial_t u + (a \cdot \nabla) u = f,\\
u(0)=u_0,
\end{cases}
\end{align*}
then
\begin{align}
\|u(t,\cdot)\|_{\infty} \le \|u_0\|_{\infty} + \int_0^t \| f(s,\cdot) \|_{\infty}ds. \label{infty_mp_e1}
\end{align}
Sometimes one writes the above inequality formally as a differential
inequality:
\begin{align*}
\partial_t \|u\|_{\infty} \le \|f\|_{\infty}.
\end{align*}

\end{rem}

\begin{proof}[Proof of Lemma \ref{lem_gd1}]
This is a rather standard energy estimate. We sketch the details.

\underline{Estimate of $\|u^{(l)}-u\|_2$}: Since
\begin{align*}
 \begin{cases}
  \partial_t u^{(l)} + (u^{(l)} \cdot \nabla) u^{(l)} =  - \nabla p^{(l)}, \\
  \partial_t u + (u\cdot \nabla) u = - \nabla p,
 \end{cases}
\end{align*}
we get
\begin{align*}
 \partial_t (u^{(l)} -u) + ( (u^{(l)} -u) \cdot \nabla ) u^{(l)} + (u\cdot \nabla) (u^{(l)} -u) =- \nabla (p^{(l)} -p).
\end{align*}
Clearly
\begin{align*}
 \frac 12 | \partial_t ( \| u^{(l)} - u \|_2^2 ) | \le \| u^{(l)} - u\|_2^2 \| D u^{(l)} \|_{\infty}.
\end{align*}
Therefore
\begin{align*}
 \max_{0\le t \le T_0} \| u^{(l)}(t) -u(t)\|_2 \le e^{T_0 \| Du^{(l)}\|_{L_t^{\infty} L_x^{\infty}[0,T_0]}} \| u_0^{(h)} \|_2.
\end{align*}

\underline{Estimate of $\|u^{(l)}-u\|_{\infty}$, $\| \omega-\omega^{(l)}\|_{\infty}$}:

From the equation
\begin{align*}
 \partial_t (\omega-\omega^{(l)}) + (u-u^{(l)}) \cdot \nabla \omega + u^{(l)} \cdot \nabla (\omega -\omega^{(l)})=0,
\end{align*}
we get (by using \eqref{infty_mp_e1})
\begin{align*}
 \max_{0\le t \le T_0} \| \omega -\omega^{(l)} \|_{\infty}
 \le \| \omega_0^{(h)} \|_{\infty} + ( \max_{0\le t \le T_0} \|u-u^{(l)}\|_{\infty}  )
 \cdot \max_{0\le t \le T_0} \|D \omega \|_{\infty}\cdot T_0.
\end{align*}

By interpolation,
\begin{align*}
 \| u - u^{(l)} \|_{\infty} \lesssim \| u -u^{(l)} \|_2^{\frac 12} \| \omega -\omega^{(l)} \|_{\infty}^{\frac 12}.
\end{align*}
Therefore
\begin{align*}
 &\max_{0\le t\le T_0} \| \omega -\omega^{(l)}\|_{\infty} \\
 \lesssim & \| \omega_0^{(h)} \|_{\infty}
 + \bigl( \max_{0\le t\le T_0} \|D\omega\|_{\infty} \cdot T_0 \bigr)^2
 \cdot e^{T_0 \|Du^{(l)}\|_{L_t^{\infty}L_x^{\infty}([0,T_0])} } \| u_0^{(h)}\|_2.
\end{align*}

Also
\begin{align*}
  & \max_{0\le t \le T_0} \| u-u^{(l)} \|_{\infty} \notag \\
\lesssim & e^{\frac 12 T_0 \| Du^{(l)} \|_{L_t^{\infty}L_x^{\infty}([0,T_0])} } \|u_0^{(h)}\|_2^{\frac 12} \|\omega_0^{(h)}\|_{\infty}^{\frac12}
\notag \\
& \quad + e^{T_0 \| Du^{(l)} \|_{L_t^{\infty}L_x^{\infty}([0,T_0])}} T_0 \cdot \max_{0\le t\le T_0}\|D\omega\|_{\infty}
\cdot \| u_0^{(h)}\|_2.
\end{align*}

\underline{Estimate of $\|D(\omega-\tilde \omega)\|_{\infty}$}:

Set $\eta=\omega-\tilde \omega$. Then
\begin{align*}
 \partial_t \eta + (u-u^{(l)}) \cdot \nabla \omega + u^{(l)} \cdot \nabla \eta =0.
\end{align*}
Therefore
\begin{align*}
 \partial_t D\eta + D(u-u^{(l)}) \cdot \nabla \omega + (u-u^{(l)}) \cdot \nabla (D\omega)
 + Du^{(l)} \cdot \nabla \eta + u^{(l)} \cdot \nabla (D\eta)=0.
\end{align*}
To bound $\|D\eta\|_{\infty}$, by using \eqref{infty_mp_e1}, we only need to treat the second, third and fourth term above. Note that
the forth term gives the integrating factor bounded by $e^{\|Du^{(l)}\|_{L_t^{\infty}L_x^{\infty}([0,T_0])} T_0}$. Therefore
we only need to treat the second and the third.

Note that
\begin{align*}
 \| D(u-u^{(l)}) \cdot \nabla \omega \|_{\infty} & \lesssim \| D \omega \|_{\infty} \| D(u-u^{(l)}) \|_{\infty} \notag \\
 & \lesssim \| D \omega \|_{\infty} \| \omega -\omega^{(l)} \|_{\infty} \log ( \|u_0^{(l)}\|_{H^3} + \| u_0^{(h)} \|_{H^3}+3)
\end{align*}
Also
\begin{align*}
 \| (u-u^{(l)})\cdot \nabla (D\omega )\|_{\infty} \lesssim \max_{0\le t \le T_0} \| u-u^{(l)} \|_{\infty}
 \cdot (\| D^2 \omega_0\|_{\infty}+3).
\end{align*}
Here we used the estimate
\begin{align} \label{tmp_d2omega_e1}
\|D^2 \omega (t)\|_{\infty} \lesssim \| D^2 \omega_0 \|_{\infty} +3.
\end{align}
which can be easily checked
 using the
usual log-Gronwall inequality
(together with the fact $\|u_0\|_2+ \|\omega_0\|_{\infty} \lesssim 1$ and $0\le t\lesssim 1$). We postpone
the proof of \eqref{tmp_d2omega_e1} to the end of this proof.

Thus
\begin{align}
  & \| D(\omega -\tilde \omega ) \|_{\infty} \notag \\
\lesssim & \log(\|u_0^{(l)}\|_{H^3}+\|u_0^{(h)}\|_{H^3}+3 )
 e^{ \|Du^{(l)}\|_{L_t^{\infty}L_x^{\infty}([0,T_0])} T_0}
 (T_0  \max_{0\le t\le T_0} \|D\omega \|_{\infty} )
  \| \omega_0^{(h)}\|_{\infty} \notag \\
 &+  \log(\|u_0^{(l)}\|_{H^3}+\|u_0^{(h)}\|_{H^3}+3 )
 e^{ 2\|Du^{(l)}\|_{L_t^{\infty}L_x^{\infty}([0,T_0])} T_0}
 (T_0  \max_{0\le t\le T_0} \|D\omega \|_{\infty} )^3
  \| u_0^{(h)}\|_{2} \notag \\
&+(\|D^2\omega_0 \|_{\infty}+3)
 e^{ 2\|Du^{(l)}\|_{L_t^{\infty}L_x^{\infty}([0,T_0])} T_0}
 (T_0  \max_{0\le t\le T_0} \|D\omega \|_{\infty} )
  \| u_0^{(h)}\|_{2} \notag \\
 &+(\|D^2\omega_0 \|_{\infty}+3)
 e^{ \frac 32\|Du^{(l)}\|_{L_t^{\infty}L_x^{\infty}([0,T_0])} T_0}
 T_0
  \| u_0^{(h)}\|_{2}^{\frac 12} \|\omega_0^{(h)}\|_2^{\frac 12}. \label{tmp_domega_tmp1}
\end{align}

Similar to \eqref{tmp_d2omega_e1}, one can also easily check that
\begin{align*}
\max_{0\le t \le T_0} \| D\omega(t)\|_{\infty} \lesssim \|D\omega_0\|_{\infty} +1.
\end{align*}

Plugging the above into the RHS of \eqref{tmp_domega_tmp1} then yields a bound for $\|D(\omega-\tilde \omega)\|_{\infty}$
expressed in terms of initial data only.

Finally to bound $\|D^2(u-\tilde u)\|_{\infty}$, we just note that $D^2(u-\tilde u) =
\Delta^{-1} \nabla^{\perp}D^2 (\omega -\tilde \omega)$.
The usual log-interpolation inequality
\begin{align*}
\| D^2 (u-\tilde u )\|_{\infty} \lesssim \| D (\omega-\tilde \omega)\|_{\infty} \log (
\| u\|_{H^4} +\|\tilde u\|_{H^4} )
\end{align*}
together with \eqref{tmp_lemgd1_cond_Hk_e1} then yields the desired bound for $\|D^2(u-\tilde u)\|_{\infty}$.

\underline{Proof of the estimate \eqref{tmp_d2omega_e1}}:

Denote $\partial^2$ as any one of $\partial^2_{x_i x_j}$, $i=1,2$. Then
\begin{align*}
&\partial_t \partial^2 \omega + \Delta^{-1} \nabla^{\perp}\partial^2 \omega \cdot \nabla \omega
+ \Delta^{-1} \nabla^{\perp}\partial \omega \cdot \nabla \partial \omega  \notag \\
& \qquad+ (\Delta^{-1} \nabla^{\perp} \omega
\cdot \nabla) (\partial^2 \omega) =0.
\end{align*}
 Then clearly by \eqref{infty_mp_e1}, we have
\begin{align*}
&\partial_t (\| \partial^2 \omega(t)\|_{\infty}) \notag \\
\lesssim  &\| \Delta^{-1} \nabla^{\perp} \partial^2 \omega\|_{\infty}
\| D\omega \|_{\infty} +\| \Delta^{-1} \nabla^{\perp} \partial \omega\|_{\infty} \| D^2 \omega \|_{\infty}.
\end{align*}
By interpolation, easy to check
\begin{align*}
\|\Delta^{-1}\nabla^{\perp} \partial^2 \omega \|_{\infty}
\lesssim \|\omega\|_{\infty}^{\frac 12} \| D^2 \omega\|_{\infty}^{\frac 12}.
\end{align*}
By using the log-interpolation, we have
\begin{align*}
&\|\Delta^{-1} \nabla^{\perp} \partial \omega \|_{\infty} \notag \\
\lesssim &\;\| \omega \|_{\infty} \log(3+ \|\omega\|_2 +
\|D^2 \omega\|_{\infty}) \notag \\
\lesssim  &\|\omega \|_{\infty} \log(3+\|D^2\omega\|_{\infty}).
\end{align*}
One can then arrive at
\begin{align*}
\partial_t (\|D^2 \omega\|_{\infty}) \lesssim  \|\omega\|_{\infty} \|D^2 \omega\|_{\infty}
\log( 3+\|D^2 \omega \|_{\infty}).
\end{align*}
Integrating in time then gives the desired estimate \eqref{tmp_d2omega_e1}.

\end{proof}

\section{Local norm inflation for the $2D$ $C^m$, $m\ge 2$ case} \label{sec_local_2Dcm}
The bulk of this section is on the norm inflation for 2D $C^2$ case. At the end we
sketch how to do the general $C^m$, $m\ge 2$ case.

\subsection{The case for $C^2$.}

\begin{align*}
\;
\end{align*}
We choose initial stream function in the form
\begin{align}
\psi_0(x) = \psi_0^{(l)}(x) + \psi_0^{(h)}(x), \notag
\end{align}
where $\psi_0^{(l)}$ will "generate" the desired Lagrangian deformation and $\psi_0^{(h)}$ has the expansion
\begin{align}
\psi_0^{(h)}(x) = \sum_{M\le j \le M+\sqrt M} 2^{-3j} a_0(2^{j} x). \notag
\end{align}
The corresponding velocity and vorticity have the form:
\begin{align}
&u_0^{(h)}(x) = \sum_{M\le j \le M+\sqrt M} 2^{-2j} (\nabla^{\perp} a_0)(2^j x), \notag \\
&\omega_0^{(h)}(x) = \sum_{M\le j \le M+\sqrt M} 2^{-j} (\Delta a_0)(2^j x). \notag
\end{align}
We shall choose $M\gg 1$. More detailed assumptions on $\psi_0^{(l)}$ and $a_0$ will become clear later.
For the moment we assume $\operatorname{supp}(a_0) \subset \{x:\, \rho_0<|x|<\rho_1\}$ for some positive
numbers $0<\rho_0<\rho_1<2\rho_0$ so that the functions $a_0(2^jx)$ have non-overlapping supports.
Denote $u_0^{(l)}= \nabla^{\perp} \psi_0^{(l)}$ and $\omega_0^{(l)}=\Delta \psi_0^{(l)}$.
We assume $\operatorname{supp}(\omega_0^{(l)}) \subset \{x: 0<\rho_2<|x|<\rho_3<\infty\}$
for some $\rho_2>0$, $\rho_3>0$.  This is just to ensure
that it is compactly supported away from the origin.

Let $u(t)$, $\omega(t)$ be the velocity and vorticity corresponding to the 2D Euler flow with
initial data $u_0=u_0^{(l)}+u_0^{(h)}$, $\omega_0=\omega_0^{(l)} + \omega_0^{(h)}$.

To get norm inflation, it suffices to examine any one of the entries $D^2u(t)$. In particular we will consider
$\partial_{11} u_2 = \partial_{11} (\Delta^{-1} \partial_1 \omega) = \Delta^{-1} \partial_{1}^3 \omega$.

Define
\begin{align*}
F_0(x) = \frac{x_1^3-3x_1 x_2^2} {|x|^6}.
\end{align*}

Let
\begin{align*}
\begin{cases}
\partial_t \omega^{(l)} + u^{(l)} \cdot \nabla \omega^{(l)} =0, \\
u^{(l)}=\Delta^{-1} \nabla^{\perp} \omega^{(l)},\\
\omega^{(l)}\Bigr|_{t=0} =\omega_0^{(l)}.
\end{cases}
\end{align*}
and let $W$ solve the \emph{linear} system:
\begin{align*}
\begin{cases}
\partial_t W + u^{(l)} \cdot \nabla W =0, \\
W\Bigr|_{t=0} =\omega_0^{(h)}.
\end{cases}
\end{align*}

Define the flow map
\begin{align*}
\begin{cases}
\partial_t \phi^{(l)}(t,x) = u^{(l)}(t,\phi^{(l)}(t,x)), \\
\phi^{(l)}(0,x)=x.
\end{cases}
\end{align*}

By Lemma \ref{lem_gd1}
 and Remark \ref{rem_gd1},  the main part of $\pi(\partial_{11} u_2)(t,\phi^{(l)}(t,0) )$ is given
by the quantity
\begin{align*}
 & \int_{\mathbb R^2} F_0(x-\phi^{(l)}(t,0))  W(t,x) dx \notag \\
 = & \int_{\mathbb R^2} F_0(\phi^{(l)}(t,x)-\phi^{(l)}(t,0) ) \omega_0^{(h)}(x) dx \notag \\
 = & \sum_{M\le j\le M+\sqrt M} \int_{\mathbb R^2} F_0(2^{j} \tilde \phi^{(l)}(t,2^{-j}x)) (\Delta a_0)(x) dx,
\end{align*}
where $\tilde \phi^{(l)} (t,x) = \phi^{(l)}(t,x)-\phi^{(l)}(t,0)$.

Now note that $\phi^{(l)}$ is completely fixed and \emph{independent of $M$}. It is not difficult to
check that as $\lambda \to \infty$,
\begin{align*}
 &\int_{\mathbb R^2} F_0(\lambda \tilde \phi^{(l)}(t,\lambda^{-1} x)) (\Delta a_0)(x) dx \notag \\
\to & \int_{\mathbb R^2} F_0( A(t) x) (\Delta a_0)(x) dx,
\end{align*}
where
\begin{align*}
A(t) = (D\phi^{(l)})(t,0).
\end{align*}

Clearly now we only need to choose $\psi_0^{(l)}$ such that for some small $0<t\lesssim 1$,
\begin{align*}
\int_{\mathbb R^2} F_0(A(t) x) (\Delta a_0)(x) dx >0.
\end{align*}
To do this one can just choose\footnote{For example one can just
take $\psi_0^{(l)}$ to be a suitable odd function of $x_1$ and
$x_2$.} $\psi_0^{(l)}$ such that for small $t$,
\begin{align*}
A(t) = \begin{pmatrix}
 {r(t)} \quad 0 \\
 0\quad \frac 1 {r(t)}
\end{pmatrix},
\end{align*}
with $r(0)=1$, $r(t)=1+$ for $t=0+$. Consider for $r>1$,
\begin{align*}
 & \int_{\mathbb R^2} F_0(rx_1,\frac 1r x_2) \Delta a_0(x) dx \notag \\
 = & \int_{\mathbb R^2} \Delta( F_0(rx_1,\frac 1 r x_2) ) a_0(x)dx.
 \end{align*}
It is not difficult to check that (as it should be) $\Delta F_0=0$ and,
\begin{align*}
 &\partial_r \Bigl(  \Delta(F_0(rx_1, \frac 1 r x_2) \Bigr) \Bigr|_{r=1} \notag \\
=&\; 48 (x_1^5 - 10 x_1^3 x_2^2 + 5 x_1 x_2^4)/(x_1^2 + x_2^2)^5 =:H_0(x).
\end{align*}
Obviously we only need to choose $a_0\in C_c^{\infty}$ to be sufficiently localized near the
point $(x_1,x_2)=(1,0)$ so that
\begin{align*}
\int_{\mathbb R^2} H_0(x) a_0(x) dx >0.
\end{align*}
For sufficiently small $t$, we then have
\begin{align*}
 & \int_{\mathbb R^2} F_0(A(t) x) (\Delta a_0)(x) dx \notag \\
 > & \int_{\mathbb R^2} H_0(x) a_0(x) dx (r(t)-1).
 \end{align*}

Clearly then for $M$ sufficiently large,
\begin{align*}
 & \pi |(\partial_{11} u_2)(t,0)| > \sqrt M \cdot \operatorname{const}.
 \end{align*}
 This gives the desired local $C^2$ norm inflation.

\subsection{The case for $C^m$, $m\ge 2$}
$$\;$$

As was already mentioned, the general case $C^m$ is a simple change of numerology. Therefore
we shall be rather brief and only sketch the needed modifications.

The initial stream function is sought in the form
\begin{align}
\psi_0(x) = \psi_0^{(l)}(x) + \psi_0^{(h)}(x), \notag
\end{align}
where again $\psi_0^{(l)}$ will produce the  Lagrangian deformation and
\begin{align}
\psi_0^{(h)}(x) = \sum_{M\le j \le M+\sqrt M} 2^{-(m+1) j} a_0(2^{j} x). \notag
\end{align}
The corresponding velocity and vorticity then have the form:
\begin{align}
&u_0^{(h)}(x) = \sum_{M\le j \le M+\sqrt M} 2^{-mj} (\nabla^{\perp} a_0)(2^j x), \notag \\
&\omega_0^{(h)}(x) = \sum_{M\le j \le M+\sqrt M} 2^{-(m-1)j} (\Delta a_0)(2^j x). \notag
\end{align}

We will consider
$\partial_{1}^m u_2 = \partial_{1}^{m} (\Delta^{-1} \partial_1 \omega) =
\Delta^{-1} \partial_{1}^{m+1} \omega$
and examine the quantity\footnote{Without loss of generality we consider the case $\phi^{(l)}(t,0)\equiv 0$.
Otherwise we can just shift to the point $\phi^{(l)}(t,0)$ as in the $C^2$ case.}
\begin{align*}
(\partial_{1}^m u_2)(t,0) =  \int_{\mathbb R^2} F_m(x) \omega(t,x)dx,
\end{align*}
where
\begin{align*}
F_m(x) &= \op{const} \cdot \partial_1^{m+1} ( \log |x| ) \notag \\
&= \op{const} \cdot |x|^{-(2m+2)} ( x_1^{m+1} + x_2 F_{\op{low}}(x) ),
\end{align*}
where $F_{\op{low}}(x)$ collects terms of the form $x_1^l x_2^{m-l}$ with $l\le m$.

Now similar to the $C^2$ case, we need to consider the quantity
\begin{align*}
H=\partial_r \Delta\biggl( F_m(rx_1, \frac 1 r x_2) \biggr) \Bigr|_{r=1}.
\end{align*}

Since
\begin{align*}
&\Delta \biggl( F_m(rx_1, \frac 1 r x_2) \biggr) \notag \\
= & r^2 (\partial_1^2 F_m)(rx_1,\frac 1 r x_2) + \frac 1 {r^2} (\partial_2^2 F_m)(rx_1,\frac 1r x_2),
\end{align*}
then easy to check that
\begin{align*}
H= \Bigl( 2(\partial_1^2-\partial_2^2) +x_1 \partial_1 \Delta-x_2 \partial_2 \Delta\Bigr)F_m(x_1,x_2).
\end{align*}
Note that by definition, $\Delta F_m =0$ for $|x|\ne 0$ and $a_0(x)$ will be chosen to be localized near
$(x_1,x_2)=(1,0)$. Therefore on the support of $a_0$, we have
\begin{align*}
H&= 4 \partial_1^2 F_m (x_1,x_2) \notag \\
& = \op{const} \cdot |x|^{-(2m+6)} ( x_1^{m+3} + x_2 \tilde F_{\op{low}}(x) ),
\end{align*}
where $\tilde F_{\op{low}}$ is a polynomial homogeneous of degree $m+2$. Obviously $H(x_1=1,x_2=0)\ne 0$ and
we can choose $a_0$ sufficiently localized near $(x_1,x_2)=(1,0)$ such that
\begin{align*}
\int H(x_1,x_2) a_0(x) dx \ne 0.
\end{align*}
After this modification, the rest of the argument proceeds in a very much similar way as in the $C^2$ case. We omit further
details.

\section{patching for 2D $C^m$, $m\ge 1$ case} \label{sec_2D_patch}
In this section we will establish several lemmas needed for the patching of local constructions in
previous sections. The bulk of this section will be occupied with the 2D $C^1$ case. At the very
end we sketch the patching argument for general $C^m$, $m\ge 2$ case. Note that the latter
case is considerably simpler in view of flow decoupling.

We first state a simple control of flow map lemma.
Let $U^{(o)} (t,x)$ be a smooth velocity field and consider
\begin{align*}
\begin{cases}
\partial_t \Phi^{(o)}(t,x) = U^{(o)} (t, \Phi^{(o)}(t,x)),\\
\Phi^{(o)} (0,x)=x.
\end{cases}
\end{align*}
Assume on some time interval $[0,T_1]$, $T_1>0$,
\begin{align}
\max_{0\le t\le T_1} \| D U^{(o)} (t,\cdot) \|_{\infty} =A_1 >0. \notag
\end{align}
Note that here we \emph{do not} assume $U^{(o)} (t,0)=0$.

Denote
\begin{align}
& \Phi(t,x)= \Phi^{(o)}(t,x) - \Phi^{(o)} (t,0), \notag\\
& U(t,x) = U^{(o)} (t,x+\Phi^{(o)} (t,0)) - U^{(o)} (t, \Phi^{(o)} (t,0)). \notag
\end{align}

Then
\begin{lem}\label{lem_hc0}
For any $\lambda\ge 3$, $|x| \lesssim 1$, and any $0\le t\le T_1$,
\begin{align}
 & \Bigl| \lambda  \Phi(t,\lambda^{-1} x) - x - \lambda t  U(0,\lambda^{-1} x) -\int_0^t (t-\tau) \lambda
 (\partial_{\tau}  U)(\tau,\lambda^{-1} x) d\tau \Bigr| \notag\\
 \lesssim & \; A_1^2 t^2 e^{tA_1}, \label{lem_hc0_st_e1}
 \end{align}
 and
 \begin{align}
  & \Bigl|\int_0^t (t-\tau) \lambda (\partial_{\tau}  U)(\tau, \lambda^{-1} x) d\tau \Bigr| \notag \\
\lesssim & \; (A_1t+A_1^2 t^2) e^{tA_1}. \label{lem_hc0_st_e2}
\end{align}
\end{lem}
\begin{rem}
The weak estimate \eqref{lem_hc0_st_e2} will be used in controlling error terms (in the proof of Lemma
\ref{lem_ha0_1}) when
making the Taylor expansion later.
\end{rem}

\begin{proof}[Proof of Lemma \ref{lem_hc0}]
By definition, it is easy to check that
\begin{align}\notag
\begin{cases}
\partial_t \Phi(t,x) = U(t, \Phi(t,x)), \\
\Phi(0,x) =x.
\end{cases}
\end{align}
and $U(t,0)\equiv 0$, $\Phi(t,0)\equiv 0$. Also $\max_{0\le t\le T_1} \|DU(t,\cdot)\|_{\infty}=A_1$.

Now for any $y\in \mathbb R^d$, we have
\begin{align}
\Phi(t,y)-y  
& = \int_0^t U(\tau, y) d\tau + \int_0^t \Bigl( U(\tau,\Phi(\tau,y)) -U(\tau, y) \Bigr) d\tau \notag\\
& = U(0,y) t \notag \\
& \quad + \int_0^t (t-\tau) (\partial_{\tau} U)(\tau, y) d\tau \notag \\
& \quad + \int_0^t \Bigl( U(\tau,\Phi(\tau,y)) -U(\tau, y) \Bigr) d\tau. \notag
\end{align}
It remains to estimate the quantity
\begin{align*}
Q_1=\int_0^t \sup_{0\le \tau\le t} \sup_{\;\lambda\ge 1, \, |x|\lesssim 1} \lambda |U(\tau,\Phi(\tau, \lambda^{-1} x))
- U(\tau, \lambda^{-1} x)| d\tau.
\end{align*}
Clearly
\begin{align*}
Q_1 &\le t \|DU\|_{\infty} \sup_{0\le \tau \le t}
\sup_{\lambda\ge 1, |x|\lesssim 1} |\lambda \Phi(t,\lambda^{-1} x) - x| \notag \\
& \le t A_1 \sup_{0\le \tau \le t} \sup_{\lambda\ge 1, |x|\lesssim 1} |\lambda \Phi(\tau,\lambda^{-1} x) - x|. \notag
\end{align*}

Now consider $F(x) = \lambda \Phi(\tau,\lambda^{-1} x) - x$. Note that $F(0)=0$ (here we used $\Phi(\tau,0)=0$). Therefore
\begin{align*}
\| F(x)\|_{L_x^{\infty}(|x|\lesssim 1)} &=\|F(x)-F(0)\|_{L_x^{\infty}(|x|\lesssim 1)}
\lesssim \| DF(x) \|_{L_x^{\infty}(|x|\lesssim 1)} \notag \\
& \lesssim \| D\Phi(\tau,z)-\op{Id}\|_{L_z^{\infty}(|z| \lesssim 1)}.
\end{align*}

Since $\partial_\tau D\Phi= (DU)(\tau,\Phi(\tau,z)) D\Phi$, it is easy to show that $\|D\Phi(\tau,\cdot)
\|_{\infty} \le e^{\tau A_1}$ and
$\| \partial_\tau D \Phi\|_{\infty} \le A_1 e^{\tau A_1}$. Since $D\Phi(0,z)=\op{Id}$, we get
\begin{align*}
\| D \Phi(t,z) -\op{Id} \|_{L_z^{\infty}} & \lesssim t \sup_{0\le \tau \le t} \| \partial_\tau D\Phi(\tau,
z)\|_{L_z^{\infty}} \notag \\
 &\lesssim A_1 t e^{tA_1}.
\end{align*}

Hence
\begin{align}
&\sup_{0\le \tau \le t} \sup_{\lambda\ge 1, |x|\lesssim 1} |\lambda \Phi(\tau,\lambda^{-1} x) - x| \notag \\
\lesssim & A_1 t e^{tA_1} \label{lem_hc0_e5}
\end{align}
and
\begin{align*}
Q_1 \le A_1^2 t^2 e^{tA_1}.
\end{align*}
This settles the estimate \eqref{lem_hc0_st_e1}.

Since $U(0,0)=0$, we have
\begin{align} \label{lem_hc0_e7}
|\lambda U(0,\lambda^{-1}x) |_{L_x^{\infty}(|x|\lesssim 1)} \lesssim \| DU \|_{\infty} \lesssim A_1.
\end{align}
The estimate \eqref{lem_hc0_st_e2} then follows from \eqref{lem_hc0_st_e1}, \eqref{lem_hc0_e5} and
\eqref{lem_hc0_e7}.

\end{proof}

Let $U^{\op{ext}}=U^{\op{ext}}(t,x)$ be a given smooth velocity field. Consider 2D Euler in vorticity
form
\begin{align*}
\begin{cases}
\partial_t \omega + (\Delta^{-1} \nabla^{\perp} \omega + U^{\op{ext}})\cdot \nabla \omega=0, \\
\omega\Bigr|_{t=0}=\omega_0,
\end{cases}
\end{align*}
and let $u_0 = \Delta^{-1} \nabla^{\perp} \omega_0$. Define the characteristic line
\begin{align*}
\begin{cases}
\partial_t \phi(t,x) = (\Delta^{-1} \nabla^{\perp} \omega + U^{\op{ext}})(t, \phi(t,x)), \\
\phi(0,x)=x.
\end{cases}
\end{align*}
Assume on some time interval $[0,T_0]$, $T_0\le 1$, we have
\begin{itemize}
\item $\max_{0\le t\le T_0} (1+\| (D \Delta^{-1} \nabla^{\perp} \omega)(t,\cdot)\|_{\infty}) =A \ge 1$;
\item $\max_{0\le t\le T_0} \bigl( \| U^{\op{ext}}(t,\cdot)\|_{H^{10}}+ \| D \partial_t U^{\op{ext}} (t,\cdot)\|_{\infty}
\bigr) \lesssim 1$.
\item $\max_{0\le t\le T_0} \| \Delta^{-1} \nabla^{\perp} \omega(t,\cdot) \|_{\infty} \lesssim 1$.
\end{itemize}

Denote
\begin{align*}
U_1^{\op{ext}}(t,x)= U^{\op{ext}}(t, x+\phi(t,0)) - U^{\op{ext}}(t,\phi(t,0)).
\end{align*}

Then

\begin{lem} \label{lem_ha0}
For any $\lambda \ge 3$, $0\le t\le T_0$, $|x| \lesssim 1$, we have
\begin{align}
  \Bigl|  \int_0^t (t-\tau) \lambda \partial_{\tau} U_1^{\op{ext}}(\tau, \lambda^{-1} x) d\tau \Bigr|
 \lesssim  t^2. \label{lem_ha0_s0}
\end{align}

\end{lem}

\begin{proof}[Proof of Lemma \ref{lem_ha0}]
Observe $U_1^{\op{ext}}(t,0)\equiv 0$. Therefore for $|x|\lesssim 1$ and $\lambda \ge 3$,
\begin{align*}
  &\lambda | (\partial_{\tau} U_1^{\op{ext}})(\tau, \lambda^{-1} x) | \notag \\
= & \lambda | (\partial_{\tau} U_1^{\op{ext}})(\tau, \lambda^{-1} x) - (\partial_{\tau} U_1^{\op{ext}})(\tau,0)
| \notag \\
\lesssim & \| D \partial_{\tau} U_1^{\op{ext}} (\tau, z) \|_{L_z^{\infty}(|z|\lesssim 1)}.\notag
\end{align*}

By using the definition of $U_1^{\op{ext}}$, we have
\begin{align*}
\partial_{\tau} U_1^{\op{ext}}(\tau, x)& = (\partial_{\tau} U^{\op{ext}})(\tau, x+\phi(\tau,0))
- (\partial_{\tau} U^{\op{ext}} )(\tau, \phi(\tau,0)) \notag \\
& \quad + \partial_{\tau} \phi(\tau,0) \cdot (\nabla U^{\op{ext}})(\tau, x+\phi(\tau,0)) \notag \\
& \quad -\partial_{\tau} \phi(\tau,0) \cdot (\nabla U^{\op{ext}})(\tau, \phi(\tau,0)). \notag
\end{align*}

By using the assumption on $\omega$ and $U^{\op{ext}}$, we have
\begin{align*}
|\partial_{\tau} \phi(\tau, 0)|
&\lesssim \| \Delta^{-1} \nabla^{\perp} \omega\|_{\infty} + \| U^{\op{ext}} \|_{\infty} \notag \\
& \lesssim 1.
\end{align*}

Therefore
\begin{align*}
\| D \partial_{\tau} U_1^{\op{ext}} \|_{\infty} &\lesssim \| \partial_{\tau} D U^{\op{ext}} \|_{\infty}
+ \| D^2 U^{\op{ext}} \|_{\infty} \notag \\
&\lesssim 1.
\end{align*}

\end{proof}

Now under the same assumptions as in Lemma \ref{lem_ha0}, we further assume the initial $\omega_0$ is taken
in terms of the stream function $\psi_0$ as follows:
\begin{align*}
& \psi_0(x)=\sum_{ \sqrt M \le j \le M}2^{-2j} a(2^jx), \\
& \omega_0(x)=\sum_{\sqrt M \le j \le M} (\Delta a)(2^j x),
\end{align*}
where the function $a(x)$ is the same as in Section \ref{sec:2DC1} (see Lemma \ref{lem_gb30}).

Then
\begin{lem}\label{lem_ha0_1}
Assume for some $\rho>0$,
\begin{align}
U^{\op{ext}}(0, x)=0, \qquad \text{for $|x|<\rho$}. \label{lem_ha0_1_e00}
\end{align}

For any $0<t \le T_0$,  we have the estimate
\begin{align*}
|(\Delta^{-1} \partial_{12} \omega)(t,\phi(t,0) )| \gtrsim Mt - M t^2 A^2 -M(A^2 t^2 +A^4 t^4) e^{2tA}.
\end{align*}

\end{lem}

\begin{proof}[Proof of Lemma \ref{lem_ha0_1}]
This is similar to the argument in Section \ref{sec:2DC1}. Therefore we only sketch the details. Recall
$F(x)=x_1x_2/|x|^4$. By Lemma \ref{lem_hc0}, we have (note that the term involving $U^{\op{ext}}(0,\lambda^{-1}x)$
drops out due to the assumption \eqref{lem_ha0_1_e00} and by taking $M$ sufficiently large),
\begin{align}
\pi (\Delta^{-1} & \partial_{12} \omega)(t,\phi(t,0) ) \notag \\
&= \sum_{\sqrt M \le j\le M}
\int F(2^j (\phi(t,2^{-j}x)-\phi(t,0) )) \Delta a(x) dx \notag \\
& \ge \sum_{\sqrt M\le j\le M} t \int \nabla F(x) \cdot \nabla^{\perp} a(x)  \Delta a(x) dx  \notag \\
&  \; + \sum_{\sqrt M \le j \le M} \int \nabla F(x) \cdot
\int_0^t (t-\tau) 2^j (\partial_{\tau} U_2)(\tau, 2^{-j}
 x) d\tau \Delta a (x) dx \label{lem_ha0_e10a} \\
&  \; + \sum_{\sqrt M \le j \le M} \int \nabla F(x) \cdot
\int_0^t (t-\tau) 2^j (\partial_{\tau} U_1)(\tau, 2^{-j}
 x) d\tau \Delta a (x) dx \label{lem_ha0_e10b} \\
 & \;+ \op{error}, \notag
\end{align}
where
\begin{align*}
U_1(\tau,x) &= (\Delta^{-1} \nabla^{\perp} \omega)(\tau,x+\phi(\tau,0)) -
(\Delta^{-1} \nabla^{\perp} \omega) (\tau, \phi(\tau,0)), \notag \\
U_2(\tau,x) & = U^{\op{ext}}(\tau, x+\phi(\tau,0)) - U^{\op{ext}}(\tau, \phi(\tau,0));
\end{align*}
and
\begin{align*}
\| \op{error} \|_{\infty} \lesssim M(A^2 t^2 +A^4 t^4) e^{2tA}.
\end{align*}
By Lemma \ref{lem_ha0}, we have
\begin{align*}
| \eqref{lem_ha0_e10a} | \lesssim M t^2.
\end{align*}
Therefore we only need to deal with \eqref{lem_ha0_e10b}.  For this denote
$u(t) = \Delta^{-1}\nabla^{\perp}  \omega(t)$.
Then since $\nabla^{\perp} \cdot u(t) = \omega(t)$, we get the equation
\begin{align*}
\nabla^{\perp}\cdot \Bigl( \partial_t u + ( u+ U^{\op{ext}})\cdot \nabla u \Bigr) = - \sum_{j=1}^2 (\nabla^{\perp}
U^{\op{ext}}_j)\cdot (\partial_j u).
\end{align*}
Here we denote $U^{\op{ext}}=(U^{\op{ext}}_1, U^{\op{ext}}_2)$, and we have used the fact that in $2D$,
if $\nabla \cdot u=0$, then
\begin{align*}
\nabla^{\perp} \cdot \Bigl(  (u \cdot \nabla) u \Bigr) = (u\cdot \nabla )( \nabla^{\perp} \cdot u).
\end{align*}
Thus we get
\begin{align*}
\partial_t u + (u+ U^{\op{ext}} ) \cdot \nabla u = -\nabla p  -\Delta^{-1} \nabla^{\perp}
\Bigl( \sum_{j=1}^2 ( \nabla^{\perp} U^{\op{ext}}_j) \cdot (\partial_j u) \Bigr).
\end{align*}

In terms of $u$, we have
\begin{align}
(\partial_{\tau} U_1)(\tau, x) &= (\partial_{\tau} u)(\tau, x+\phi(\tau,0) ) +
\partial_{\tau} \phi(\tau,0) \cdot  (\nabla u)(\tau, x+\phi(\tau,0)) \notag \\
&\quad  -(\partial_{\tau} u)(\tau, \phi(\tau,0) ) -
\partial_{\tau} \phi(\tau,0) \cdot  (\nabla u)(\tau, \phi(\tau,0)). \notag\\
& = \bigl( (u+U^{\op{ext}})(\tau, \phi(\tau,0)) - (u+U^{\op{ext}})(\tau, x+ \phi(\tau,0))
\bigr)\cdot (\nabla u)(\tau,x+\phi(\tau,0)) \label{lem_ha0_e20a} \\
& \quad - (\nabla p)(\tau, x+\phi(\tau,0)) \label{lem_ha0_e20aa}\\
&\quad +(\nabla p)(\tau, \phi(\tau,0))  \label{lem_ha0_e20b} \\
& \quad -\Biggl( -\Delta^{-1} \nabla^{\perp}
\Bigl( \sum_{j=1}^2 ( \nabla^{\perp} U^{\op{ext}}_j) \cdot (\partial_j u) \Bigr)
\Biggr)(\tau, x+\phi(\tau,0)) \label{lem_ha0_e20c}\\
& \quad +\Biggl( -\Delta^{-1} \nabla^{\perp}
\Bigl( \sum_{j=1}^2 ( \nabla^{\perp} U^{\op{ext}}_j) \cdot (\partial_j u) \Bigr)
\Biggr)(\tau, \phi(\tau,0)). \label{lem_ha0_e20d}
\end{align}

Observe that for $\lambda =2^j$,
\begin{align*}
&\lambda \| (u+ U^{\op{ext}})(\tau, \lambda^{-1} x +\phi(\tau,0))-
(u+ U^{\op{ext}})(\tau, \phi(\tau,0)) \|_{L_x^{\infty}(|x| \lesssim 1)} \notag\\
\lesssim & \|Du \|_{\infty} + \| D U^{\op{ext}} \|_{\infty}  \lesssim A+1 \lesssim A.
\end{align*}
Therefore the contribution of the term \eqref{lem_ha0_e20a} in \eqref{lem_ha0_e10b} is
\begin{align*}
\lesssim M t^2 A^2
\end{align*}
which is OK.

Now for the pressure term \eqref{lem_ha0_e20aa}, we may proceed the
same way as in the proof of Lemma \ref{lem_gb10} (note that
translation by $\phi(\tau,0)$ does not affect the argument), and its
contribution is
\begin{align*}
\lesssim M t^2 A^2.
\end{align*}

The contribution of \eqref{lem_ha0_e20b} to \eqref{lem_ha0_e10a} is
zero since for any $l=1,2$,
\begin{align*}
\int_{\mathbb R^2} \partial_{x_l} F(x) \Delta a (x) dx = \int_{\mathbb R^2} \partial_{x_l} \Delta F(x) a(x) dx =0.
\end{align*}
(Recall that $F$ corresponds to $\Delta^{-1} \partial_1 \partial_2$ and $a$ is supported away from zero.)

By the same reason, the contribution of \eqref{lem_ha0_e20d} to \eqref{lem_ha0_e10a} is also zero.

Finally for the term \eqref{lem_ha0_e20c}, denote $x_0=\phi(\tau,0)$ and note that
\begin{align*}
  &\lambda \Bigl(\Delta^{-1} \nabla^{\perp} ( \nabla^{\perp} U_j^{\op{ext}} \cdot \partial_j u )
  \Bigr)(t, \lambda^{-1} x+x_0) \notag \\
= & \Bigl((\Delta^{-1} \nabla^{\perp}) \Bigl( \nabla^{\perp} U_j^{\op{ext}} (\lambda^{-1} y)
\cdot \partial_j u (\lambda^{-1} y)   \Bigr) \Bigr)(x+ \lambda x_0) \notag \\
= & \Bigl((\Delta^{-1} \nabla^{\perp}) \Bigl( \nabla^{\perp} U_j^{\op{ext}} (\lambda^{-1} y -x_0)
\cdot \partial_j u (\lambda^{-1} y-x_0)   \Bigr) \Bigr)(x), \notag
\end{align*}
where in  the above we used the notation $(\partial_j u)(\lambda^{-1} y)$ to stress that the scaling is
done before applying the non-local operator $\Delta^{-1} \nabla^{\perp}$.

After the above switch, one can proceed in a similar fashion as in the $\nabla p$ case. Note that we
need to check the condition $\Delta^{-1} \nabla^{\perp} \cdot ( \nabla F \Delta a ) \in L_x^1$. The only
difference is that the conditions \eqref{lem_gb10_e4} and \eqref{lem_gb10_e5} become\footnote{One just
need to replace $g(y)$ therein by $g(y)^{\perp}=(-g_2(y),g_1(y))$.}
\begin{align*}
&\int  y_1 g_2(y) dy= \int y_2 g_1(y) dy=0,\\
& \int (-y_1 g_1(y) + y_2 g_2(y) ) dy =0.
\end{align*}
This only introduces the stronger condition \eqref{gb_40d_p} in Lemma \ref{lem_gb30} which is obviously
satisfied.

Hence the term \eqref{lem_ha0_e20d} also gives the contribution
\begin{align*}
\lesssim M t^2 A^2.
\end{align*}

\end{proof}

\begin{lem}[Patching lemma for 2D $C^1$ case] \label{lem_ha1}
Suppose $u^A \in C_c^{\infty}(\mathbb R^2)$ is a given velocity field on $\mathbb R^2$ such that
 for some $R_0>0$,
\begin{align*}
\operatorname{supp}{(u^A)} \subset\{x=(x_1,x_2):\,  x_2 <-2R_0 \};
\end{align*}

Then for any $0<\epsilon<\frac{R_0}{100}$, one can find $0<\delta_0=\delta_0(u^A, \epsilon,R_0)<R_0$,
$0<t_0=t_0(u^A,\epsilon,R_0)<\epsilon$, and $\psi^B \in C_c^{\infty}(B(0,\epsilon))$ ($\psi^B$
depends only on $(u^A, \epsilon, R_0)$) with the property
$$\|\psi^B \|_{\infty}+ \|D^2 \psi^B \|_{\infty} <\epsilon$$

such that for any velocity field $u^C \in C_c^{\infty}(\mathbb R^2)$ with the properties:
\begin{itemize}
\item $\operatorname{supp}(u^C) \subset \{ x=(x_1,x_2):\, x_2 >R_0\}$;
\item $\|u^C \|_{2} + \|D u^C \|_{\infty} <\delta_0$;
\end{itemize}
the following hold true:

Consider the 2D Euler equation (in velocity form) for $u=(u_1,u_2)$:
\begin{align*}
\begin{cases}
\partial_t u + (u\cdot \nabla) u + \nabla p =0, \\
\nabla \cdot u=0,\\
u\Bigr|_{t=0} = u^A +u^B + u^C,
\end{cases}
\end{align*}
where  $u^B= \nabla^{\perp} \psi^B$.
Consider also the corresponding vorticity equation
\begin{align*}
\begin{cases}
\partial_t \omega + (u\cdot \nabla) \omega =0, \\
u= \Delta^{-1} \nabla^{\perp} \omega,\\
\omega \Bigr|_{t=0} = \omega_0^A + \omega_0^B + \omega_0^C,
\end{cases}
\end{align*}
where $\omega_0^A = \nabla^{\perp}\cdot u^A$, $\omega_0^B=\Delta \psi^B$, $\omega_0^C = \nabla^{\perp}
\cdot u^C$. Then the smooth
solution $u=u(t,x)$ with vorticity $\omega =\omega(t,x)$ satisfies the following properties:

\begin{enumerate}
\item for any $0\le t\le t_0$, we have the decomposition
\begin{align*}
\omega(t,x) = \omega^A(t,x) + \omega^B (t,x) + \omega^C(t,x),
\end{align*}
where
\begin{align}
& \operatorname{supp}(\omega^A(t)) \subset B( \operatorname{supp}(u^A), \frac 1 8 R_0); \notag\\
& \operatorname{supp}(\omega^B(t)) \subset B(0, \epsilon+ \frac 18 R_0); \notag\\
& \operatorname{supp}(\omega^C(t)) \subset B(\operatorname{supp}(u^C), \frac 18 R_0). \label{lem_ha1_e3}
\end{align}

\item there exist $0<t_1<t_2<t_0$, such that for any $t\in [t_1,t_2]$, we have
\begin{align} \label{lem_ha1_e4_a9}
\| (D\Delta^{-1} \nabla^{\perp}\omega^B)(t,x) \|_{L_x^{\infty}(|x| \le \epsilon + \frac 18 R_0)} > \frac 2 {\epsilon},
\end{align}
and
\begin{align} \label{lem_ha1_e4_a10}
&\| (D\Delta^{-1}\nabla^{\perp} \omega^A)(t,x) \|_{L_x^{\infty}(|x| \le \epsilon + \frac 18 R_0)}
+\| (D\Delta^{-1}\nabla^{\perp} \omega^B)(t,x) \|_{L_x^{\infty}(|x| > \epsilon + \frac 18 R_0)} \notag \\
&\quad +\| (D\Delta^{-1}\nabla^{\perp} \omega^C)(t,x) \|_{L_x^{\infty}(|x| \le \epsilon + \frac 18 R_0)}
\lesssim_{R_0,u_0^A} 1.
\end{align}
Consequently,
\begin{align} \label{lem_ha1_e4}
\| (Du)(t,x) \|_{L_x^{\infty}(|x| \le \epsilon + \frac 18 R_0)} > \frac 1 {\epsilon}.
\end{align}

\end{enumerate}

\end{lem}

\begin{proof}[Proof of Lemma \ref{lem_ha1}]
We take $\psi_0^B$ in the form
\begin{align*}
\psi_0^B(x) = \sum_{ \sqrt M \le j \le M} 2^{-2j} a(2^jx),
\end{align*}
where $a(x)$ is the same as in Lemma \ref{lem_ha0_1} and $M$ will be taken sufficiently large.

Introduce $\omega^A(t)$ such that
\begin{align*}
\begin{cases}
\partial_t \omega^A(t,x) + (u(t,x) \cdot \nabla) \omega^A(t,x)=0, \\
\omega^A \Bigr|_{t=0} = \omega_0^A.
\end{cases}
\end{align*}
Define $\omega^B(t)$, $\omega^C(t)$ similarly (with initial values $\omega_0^B$ and $\omega_0^C$ respectively).
By using finite speed transport and choosing $t_0$ sufficiently small, one can easily make for any
$0\le t\le t_0$,
\begin{align*}
& \operatorname{supp}(\omega^A(t)) \subset B( \operatorname{supp}(u^A), \frac 1 {16} R_0); \notag\\
& \operatorname{supp}(\omega^B(t)) \subset B(0, \epsilon+ \frac 1{16} R_0); \notag\\
& \operatorname{supp}(\omega^C(t)) \subset B(\operatorname{supp}(u^C), \frac 1{16} R_0). \notag
\end{align*}
Note that here we take $\frac 1{16} R_0$ (instead of $\frac 18 R_0$ in \eqref{lem_ha1_e3}) to have a little bit
more room to play.

Now $\omega=\omega^A+\omega^B + \omega^C$, we have
\begin{align*}
Du(t) = D\Delta^{-1} \nabla^{\perp} \omega^A(t) + D \Delta^{-1} \nabla^{\perp} \omega^B(t) +
D\Delta^{-1} \nabla^{\perp} \omega^C(t).
\end{align*}
Since the support of $\omega^A(t)$ remains at a distance at least $R_0$
from the support of $\omega^B(t)$, easy to check that for any $0\le t\le t_0$,
\begin{align*}
\| D \Delta^{-1} \nabla^{\perp} \omega^A(t) \|_{L_x^{\infty}(|x|<\epsilon+\frac 18 R_0)}
\lesssim_{u^A, R_0} 1.
\end{align*}
Similarly
\begin{align*}
\| D \Delta^{-1} \nabla^{\perp} \omega^C(t) \|_{L_x^{\infty}(|x|<\epsilon+\frac 18 R_0)}
\lesssim_{R_0} 1.
\end{align*}
Also since $\op{supp}(\omega^B(t)) \subset B(0,\epsilon+\frac 1 {16} R_0)$, easy to check that
\begin{align*}
\| D \Delta^{-1} \nabla^{\perp} \omega^B (t) \|_{L_x^{\infty}(|x|\ge \epsilon+\frac 18 R_0)} \lesssim_{R_0} 1.
\end{align*}
Therefore the main "mass" of $\| Du(t) \|_{L_x^{\infty} (|x| < \epsilon+ \frac 18 R_0)}$ is
given by $\| D \Delta^{-1} \nabla^{\perp}\omega^B(t)\|_{\infty}$ which is concentrated in the region $B(0,\epsilon+\frac 18R_0)$.

To get \eqref{lem_ha1_e4}, we only need to examine the equation for $\omega^B(t)$. Clearly it has the form
\begin{align*}
\begin{cases}
\partial_t \omega^B(t,x) + ( \Delta^{-1} \nabla^{\perp} \omega^B(t) + U^{\op{ext}}(t,x)) \cdot \nabla
\omega^B(t,x) =0, \\
\omega^B \Bigr|_{t=0} = \omega_0^B,
\end{cases}
\end{align*}
where
\begin{align*}
U^{\op{ext}}(t) = \chi_{ < \epsilon+\frac 1{16} R_0} \cdot \Bigl((\Delta^{-1}\nabla^{\perp} \omega^A) (t,x) +
(\Delta^{-1} \nabla^{\perp}\omega^C)(t,x))
\Bigr).
\end{align*}
Here $\chi_{<\epsilon+\frac 1{16}R_0} \in C_c^{\infty}(\mathbb R^2)$ is such that
$\chi_{<\epsilon+\frac 1{16}R_0}(x)=1$ for $|x|<\epsilon+\frac 1 {16}R_0$ and $\chi_{<\epsilon+\frac 1{16}R_0}
(x)=0$ for $|x|\ge \epsilon+\frac 1 8R_0$. Easy to verify that
\begin{align*}
\| U^{\op{ext} }(t,x) \|_{H^{10}}+ \| D\partial_t  U^{\op{ext}}(t) \|_{\infty} \lesssim_{R_0} 1.
\end{align*}
We can then apply Lemma \ref{lem_ha0_1} to get the result. We omit further routine details.

\end{proof}

We now sketch the patching argument for general $C^m$, $m\ge 2$. To avoid unnecessary numerology we will just
explain the $C^2$ case. We just need to modify the inflation argument from Section \ref{sec_local_2Dcm} and
still produce $C^2$ norm inflation for the system
\begin{align*}
\begin{cases}
\partial_t \omega + (U^{\op{ext}}(t,x) + u \cdot \nabla) \omega =0, \\
u=\Delta^{-1} \nabla^{\perp} \omega,\\
\omega \Bigr|_{t=0}=\omega_0,
\end{cases}
\end{align*}
where $U^{\op{ext}}$ is smooth (it represents contribution of velocity from other patches)
 and $U^{\op{ext}}(0,x)=0$ for $x$ in a small neighborhood of the origin.

Take $\omega_0= \Delta \psi_0$, $\psi_0=\psi_0^{(l)}+\psi_0^{(h)}$, where
\begin{align}
\psi_0^{(h)}(x) = \sum_{M\le j \le M+\sqrt M} 2^{-3j} a_0(2^{j} x). \notag
\end{align}

Let $\omega^{(l)}$ and $\omega^{(h)}$ solve respectively
\begin{align*}
\begin{cases}
\partial_t \omega^{(l)} + (U^{\op{ext}} + \Delta^{-1} \nabla^{\perp} \omega^{(l)}\cdot
\nabla) \omega^{(l)} =0,\\
\omega^{(l)} \Bigr|_{t=0} =\omega_0^{(l)} =\Delta \psi_0^{(l)};
\end{cases}\\
\begin{cases}
\partial_t \omega^{(h)} + (U^{\op{ext}} + \Delta^{-1} \nabla^{\perp} \omega^{(l)}\cdot
\nabla) \omega^{(h)} =0,\\
\omega^{(h)} \Bigr|_{t=0} =\omega_0^{(h)} =\Delta \psi_0^{(h)};
\end{cases}
\end{align*}

By using a similar argument to that in Lemma \ref{lem_gd1}, it is not difficult to check
that we can also get flow decoupling (in the presence of smooth $U^{\op{ext}}$) such that
\begin{align*}
 \|D^2 u(t)-D^2 \tilde u(t)\|_{\infty} \ll 1,
 \end{align*}
 where $\tilde u= \Delta^{-1} \nabla^{\perp} \omega^{(h)}$.

Similar to Section \ref{sec_local_2Dcm}, we then only need to examine the quantity
\begin{align*}
Q:=\int F_0(A(t) x) ) \Delta a_0(x) dx,
\end{align*}
where $A(t) = (D \phi^{(l)})(t,0)$, and
\begin{align*}
\begin{cases}
\partial_t \phi^{(l)}(t,x) = (U^{\op{ext}} + \Delta^{-1} \nabla^{\perp} \omega^{(l)})(t,\phi^{(l)}(t,x)),\\
\phi^{(l)}(0,x)=x.
\end{cases}
\end{align*}

Since
\begin{align*}
\partial_t D \phi^{(l)}(t,x) = (DU^{\op{ext}} + D \Delta^{-1} \nabla^{\perp} \omega^{(l)} )(t,\phi^{(l)}(t,x)) D\phi^{(l)}(t,x),
\end{align*}
we get
\begin{align*}
(\partial_t D\phi^{(l)})(t,0) = (D\nabla^{\perp} \psi_0^{(l)})(x=0) + O(t).
\end{align*}
Note here we used the assumption $U^{\op{ext}}(0,x)=0$ for $x$ in a neighborhood of the origin. Hence
\begin{align*}
D \phi^{(l)}(t,0) = \op{Id} + t(D\nabla^{\perp} \psi_0^{(l)})(0) +O(t^2).
\end{align*}

Obviously then
\begin{align*}
Q = t \int \nabla F_0(x) \cdot (D \nabla^{\perp} \psi_0^{(l)}(0) x) \Delta a_0(x) dx +O(t^2).
\end{align*}
Now a little thinking shows that
clearly we are back to the situation in Section \ref{sec_local_2Dcm} where $U^{\op{ext}}=0$. The same
choice of $\psi_0^{(l)}$ and $a_0$ would work. Therefore we are done.

\section{proof of Theorem \ref{thm1} for $d=2$ and proof of Theorem \ref{thm2D_2Ck}}

\begin{lem} \label{s7_lem1}
Let $f^i\in C_c^{\infty}(\mathbb R^2)$ with $\op{supp}(f^i) \subset B(0,1)$, $i=1,2$.
Let $W$ and $\omega^i$, $i=1,2$ be smooth solutions to the following systems:
\begin{align}
&\begin{cases}
\partial_t W + (\Delta^{-1} \nabla^{\perp} W \cdot \nabla) W =0, \\
W \Bigr|_{t=0} =W_0(x) = f^1(x) + f^2(x-x_W);
\end{cases}
\notag\\
&\begin{cases}
\partial_t \omega^i + (\Delta^{-1} \nabla^{\perp} \omega^i\cdot\nabla) \omega^i=0,\\
\omega^i \Bigr|_{t=0} =f^i(x),
\end{cases}
\notag
\end{align}
where $i=1,2$, $x_W \in \mathbb R^2$ is a vector parameter to control the distance between
supports of $f^1$ and (translated) $f^2$.

For any $\epsilon>0$ and integer $k_0\ge 3$, there exists
$R_{\epsilon}=R_{\epsilon}(\epsilon,k_0, \|f^1\|_{H^{k_0+3}(\mathbb R^2)},
\|f^2\|_{H^{k_0+3}(\mathbb R^2)})>0$ sufficiently large, such that if $|x_W| \ge R_{\epsilon}$, then the following
hold:

\begin{enumerate}
\item For any $0\le t\le 1$, $W(t,x)$ admits the decomposition
\begin{align} \label{s7_lem1_e1}
W(t,x) = W^1(t,x) + W^2(t,x-x_W),
\end{align}
where
\begin{align*}
& \op{supp}(W^1) \subset B(\op{supp}(f^1), R_0)=:B_1,\\
& \op{supp}(W^2) \subset B(\spp (f^2), R_0)=: B_2,
\end{align*}
and
\begin{align*}
R_0 = A_1\cdot (\|f^1\|_1 +\|f^1\|_{\infty} + \|f^2\|_1 +\|f^2\|_{\infty}).
\end{align*}
Here $A_1>0$ is an absolute constant.\footnote{It is roughly the same constant occurring in the
inequality $\|\Delta^{-1} \nabla^{\perp} f \|_{L^{\infty}(\mathbb R^2)} \le A_1
(\|f\|_{L^1(\mathbb R^2)} + \|f\|_{L^{\infty}(\mathbb R^2)})$.}

Note that $\spp (\omega^i(t,\cdot) ) \subset B_i$, $i=1,2$ for all $0\le t\le 1$ as well.

\item All $H^k$, $k\ge 2$ norms of $W^i$ can be bounded (almost) purely in terms of  initial data in
the $i^{th}$ patch:
for $i=1,2$, and any $k\ge 2$,
\begin{align} \label{s7_lem1_e2}
\max_{0\le t\le 1} \| W^i (t,\cdot)\|_{H^k} \le C(k,\|f^1\|_{\infty}, \|f^2\|_{\infty}) \|f^i\|_{H^k}.
\end{align}

\item $\omega^i$ and $W^i$ are close:
\begin{align} \label{s7_lem1_e3}
\max_{0\le t\le 1} \| \omega^i(t) -W^i(t) \|_{H^{k_0} } <\epsilon.
\end{align}

\end{enumerate}

\end{lem}

\begin{proof}[Proof of Lemma \ref{s7_lem1}]
This is more or less straightforward. Let $W^1$ be the solution to the \emph{linear} system:
\begin{align*}
\begin{cases}
\partial_t W^1 + \Delta^{-1} \nabla^{\perp} W\cdot \nabla W^1 =0, \quad 0<t\le 1,\\
W^1 \Bigr|_{t=0}=f^1.
\end{cases}
\end{align*}
Also let $W^2$ solve the system
\begin{align*}
\begin{cases}
\partial_t W^2 + (\Delta^{-1} \nabla^{\perp} W)(t,x+x_W) \cdot \nabla W^2 =0, \quad 0<t\le 1,\\
W^2 \Bigr|_{t=0}=f^2.
\end{cases}
\end{align*}
By using finite speed propagation, the inequality $\| \Delta^{-1}
\nabla^{\perp} u\|_{L^\infty(\mathbb R^2)} \lesssim \|u
\|_{L^1(\mathbb R^2)} + \|u\|_{L^{\infty}(\mathbb R^2)}$ (and the
conservation of $L^1$ and $L^{\infty}$ norms of vorticity), we
easily obtain \eqref{s7_lem1_e1}.

The estimate \eqref{s7_lem1_e2} follows from simple energy estimates. For example, consider the equation
for $W^1$:
\begin{align*}
\partial_t W^1 + (\Delta^{-1} \nabla^{\perp} W \cdot \nabla) W^1=0.
\end{align*}
The drift term $\Delta^{-1} \nabla^{\perp} W$ naturally splits as
\begin{align*}
\Delta^{-1} \nabla^{\perp} W = \Delta^{-1} \nabla^{\perp} W^1 + \Delta^{-1} \nabla^{\perp} ( W^2(\cdot-x_W)).
\end{align*}
We only need to treat the second term above and estimate its contribution (in the energy estimate).
Since $\spp (W^i) \subset B_i$, we can insert smooth cut-offs and write:
\begin{align}
 V^{\op{ext}} &= \chi_{|x|<R_0+100} \Delta^{-1} \nabla^{\perp} ( W^2(\cdot -x_W) ) \notag\\
= & \chi_{|x|<R_0+100} \int_{\mathbb R^2} K(x-y) \chi_{|y-x_W|<R_0+100} W^2(y-x_W) dy \notag\\
= & \chi_{|x|<R_0+100} \int_{\mathbb R^2} K(x-y)\chi_{|x-y|>\frac 12 |x_W|} W^2(y-x_W) dy, \label{s7_lem1_p1_e3}
\end{align}
where $K$ is the kernel corresponding to $\Delta^{-1} \nabla^{\perp}$ and $|x_W|$ is taken sufficiently large.
Note that $K(z) \sim |z|^{-1}$ for $|z|\gtrsim 1$. Obviously then for any integer $k\ge 1$,
\begin{align*}
\| D^k V^{\op{ext}} \|_{\infty} &\lesssim_k \|W^2 (t,\cdot) \|_1 \notag \\
& \lesssim \|W^2(t=0,\cdot)\|_1 \lesssim \|f^2 \|_{\infty}.
\end{align*}

 The estimate of \eqref{s7_lem1_e2} for $W^2$ is similar and therefore omitted.

We now turn to \eqref{s7_lem1_e3}. We only need to treat the case $i=1$. Set $\eta=\omega^1-W^1$. Then
\begin{align*}
\partial_t \eta + \Delta^{-1} \nabla^{\perp} \omega^1 \cdot \nabla \eta +\Delta^{-1}
\nabla^{\perp} \eta \cdot \nabla W^1 -V^{\op{ext}} \cdot \nabla W^1=0,
\end{align*}
Clearly
\begin{align*}
\partial_t (\| \eta\|_2 ) &\lesssim  \| \Delta^{-1} \nabla^{\perp} \eta \cdot \nabla W^1\|_2
+\| V^{\op{ext}} \cdot \nabla W^1\|_2 \notag \\
& \lesssim \| \Delta^{-1} \nabla^{\perp} \eta \|_{2+} \| \nabla{W^1}\|_{\infty-}
 + \| V^{\op{ext}}\|_{\infty } \|\nabla W^1\|_{2} \notag \\
& \lesssim_{R_0} \| \eta\|_2  \| W^1 \|_{H^2(\mathbb R^2)}  + \|V^{\op{ext}}\|_{\infty} \| \nabla W^1\|_{2}.
\end{align*}
In the last inequality above we have used the fact that $\| \Delta^{-1} \nabla^{\perp} \eta\|_{2+} \lesssim
\| \eta\|_{1+} \lesssim \|\eta\|_2$ since $\eta$ is compactly supported.
By the analysis in \eqref{s7_lem1_p1_e3}, easy to see that (below $R=|x_W|$)
\begin{align*}
\| V^{\op{ext}} \|_{\infty} \lesssim_{R_0} R^{-1} \|f^2\|_1
\end{align*}
Since $\eta(t=0)=0$, we get
\begin{align*}
\max_{0\le t\le 1} \| \eta(t,\cdot)\|_2 \lesssim_{R_0} \frac 1 R C(\|f^1\|_{H^3},\|f^2\|_{H^3}).
\end{align*}
Interpolating with \eqref{s7_lem1_e2} (note that $\omega^i$ obeys the similar Sobolev bounds as $W^i$) and
taking $R=|x_W|$ sufficiently large then
gives the result.

\end{proof}

\begin{prop}[Local to global, gluing of almost non-interacting patches] \label{s7_prop1}
  Let $\{ f^j \}_{j=1}^{\infty}$ be a sequence of functions in $C_c^{\infty} (B(0,1))$ and
  satisfy the following condition:
  \begin{align} \label{s7_pr1_1}
    \sum_{j=1}^{\infty} \| f^j \|_{L^1} +
    \sup_{j} \| f^j \|_{L^{\infty}}
   \le C_1^{\prime} <\infty.
  \end{align}

Denote $C_1=C_1^{\prime}+1$. Let $k_0\ge 4$ be a fixed integer.
Then there exist centers $x_j \in \mathbb R^2$ whose mutual distances are sufficiently large (i.e. $|x_j-x_k|\gg 1$ if $j\ne k$) such
that the following hold:

\begin{enumerate}
 \item Take the initial data
 \begin{align*}
  \omega_0(x) = \sum_{j=1}^{\infty} f^j(x-x_j),
 \end{align*}
then $\omega_0 \in L^1 \cap L^{\infty} \cap C^{\infty}$. Furthermore for any $j\ne k$
\begin{align} \label{s7_pr1_2}
 B(x_j, 100A_1C_1) \cap B(x_k, 100 A_1C_1) = \varnothing.
\end{align}
Here $A_1>0$ is an absolute constant.

\item With $\omega_0$ as initial data, there exists a unique solution $\omega$ to the Euler equation
\begin{align*}
 \partial_t \omega + \Delta^{-1} \nabla^{\perp} \omega \cdot \nabla \omega =0
\end{align*}
on the time interval $[0,1]$ satisfying $\omega \in L^1\cap L^{\infty} \cap C^{\infty}$, $u=\Delta^{-1} \nabla^{\perp} \omega \in C^{\infty}$.
Moreover for any $0\le t \le1$,
\begin{align}
 \operatorname{supp} ( \omega(t,\cdot) ) \subset \bigcup_{j=1}^{\infty} B(x_j, 3A_1C_1), \label{s7_pr1_3}
\end{align}
and $\omega(t,x)$ can be decomposed accordingly as:
\begin{align*}
\omega(t,x) = \sum_{j=1}^{\infty} \omega^j(t,x-x_j),
\end{align*}
where $\omega^j (t,\cdot) \in C_c^{\infty}( B(0,3A_1 C_1))$ for all $0\le t\le 1$. Furthermore
for any $k\ge 3$,
\begin{align}
\max_{0\le t\le 1} \| \omega^j(t,\cdot)\|_{H^k} \le C_2\cdot \| f^j\|_{H^k}, \label{s7_pr1_3a}
\end{align}
where $C_2>0$ is a constant depending only on $k$ and $C_1^{\prime}$.

\item For any $j\ge 1$,
\begin{align}
 \max_{0\le t \le 1} \| \omega^j(t,\cdot) - \tilde \omega^j(t,\cdot)\|_{H^{k_0+100} } <2^{-j}. \label{s7_pr1_4}
\end{align}
Here $\tilde \omega^j$ is the solution solving the equation
\begin{align*}
 \begin{cases}
  \partial_t \tilde \omega^j + \Delta^{-1} \nabla^{\perp} \tilde \omega^j \cdot \nabla \tilde \omega^j =0,
  \quad 0<t\le 1,\, x \in \mathbb R^2; \\
  \tilde \omega^j(t=0,x)= f^j(x), \quad x \in \mathbb R^2.
 \end{cases}
\end{align*}

\item For any integer $1\le k\le k_0$, there is a constant $C_k>0$ such that
\begin{align} \label{s7_pr1_5}
\sup_{0\le t\le 1} \| (D^k u)(t,x) - \sum_{j=1}^{\infty} (D^k \Delta^{-1} \nabla^{\perp} \tilde \omega^j)(t,x-x_j)
\|_{L_x^\infty(\mathbb R^2)} <C_k.
\end{align}

\item For any $j\ge 1$, there exists $r_j>0$  such that
\begin{align} \label{s7_pr1_6}
\sup_{\substack{0\le t\le 1\\ 1\le k\le k_0\\ |x|>r_j}} | (D^k \Delta^{-1} \nabla^{\perp} \tilde \omega^j)(t,x)|<
\frac 1 {2^j}.
\end{align}
Therefore
for any $1\le k \le k_0$,
 \begin{align*}
  & \sum_{j=1}^{\infty} (D^k \Delta^{-1} \nabla^{\perp} \tilde \omega^j)(t,x-x_j) \notag \\
= & \tilde \eta_k (x)+
 \sum_{j=1}^{\infty} \chi_{<1}(\frac {x-x_j} {r_j}) \cdot (D^k \Delta^{-1} \nabla^{\perp} \tilde \omega^j)(t,x-x_j),
 \end{align*}
where $\|\tilde \eta_k\|_{\infty} <2$ and $\chi_{<1}$ is a smooth cut-off function such that
$\chi_{<1}(x)=1$ for $|x|<1$ and $\chi_{<1}(x)=0$ for $|x|\ge 2$.

Consequently by choosing the centers $x_j$ sufficiently far away from each other, we can have
\begin{align}
&\spp \Bigl( \chi_{<1}(\frac {x-x_j} {r_j}) \cdot (D^k \Delta^{-1} \nabla^{\perp} \tilde \omega^j)(t,x-x_j) \Bigr)
\notag \\
&\bigcap \spp
\Bigl( \chi_{<1}(\frac {x-x_l} {r_l}) \cdot (D^k \Delta^{-1} \nabla^{\perp} \tilde \omega^l)(t,x-x_l) \Bigr)
= \varnothing, \label{s7_pr1_7}
\end{align}
for any $j\ne l$.

\end{enumerate}

\end{prop}

\begin{proof}[Proof of Proposition \ref{s7_prop1}]
Note that \eqref{s7_pr1_2} is fairly easy to achieve. The properties \eqref{s7_pr1_3}, \eqref{s7_pr1_3a} and \eqref{s7_pr1_4} follow
from recursively applying Lemma \ref{s7_lem1} (using \eqref{s7_pr1_1}), and a simple contraction argument.

The inequality \eqref{s7_pr1_5} follow from \eqref{s7_pr1_4}, Sobolev embedding and the triangle inequality:
\begin{align*}
   & \| D^k \Delta^{-1} \nabla^{\perp} \omega -\sum_{j=1}^{\infty} (D^k \Delta^{-1} \nabla^{\perp} \tilde \omega^j)
   (t,x-x_j) \|_{\infty} \notag \\
   \le & \; \sum_{j=1}^{\infty} \| D^k \Delta^{-1} \nabla^{\perp} (\omega^j -\tilde \omega^j ) \|_{\infty} \notag \\
   \lesssim_k & \sum_{j=1}^{\infty} \| \omega^j -\tilde \omega^j \|_{H^{k_0+10}} \le C_k.
   \end{align*}

The property \eqref{s7_pr1_6} obviously follows from the fact that $\tilde \omega^j$ is a smooth and
decaying function. The property \eqref{s7_pr1_7} justifies that the infinite sum in
\eqref{s7_pr1_5} is actually a locally finite summation.
\end{proof}

We are now ready to complete the
\begin{proof}[Proof of Theorem \ref{thm1} for 2D]
Without loss of generality we can assume $u^{(g)}\equiv 0$. The argument for nonzero $u^{(g)}$ is a small
modification.

For each $j\ge 1$, we can use the local constructions (for $C^1$ see Section 3 and for $C^m$, $m\ge 2$ see Section 5)
to find $g^j \in C_c^{\infty} (B(0,1/100))$ such that
\begin{itemize}
\item $\| g^j \|_1 + \|g^j \|_{H^m} + \|g^j \|_{C^m} <2^{-j-1} \epsilon$.

\item Let $\tilde u^j$ solve
\begin{align*}
\begin{cases}
\partial_t \tilde u^j + (\tilde u^j \cdot \nabla) \tilde u^j = -\nabla \tilde p^j,\\
\nabla \cdot \tilde u^j =0,\\
\tilde u^j \Bigr|_{t=0} = g^j.
\end{cases}
\end{align*}
Then for some $0<t_j^1<t_j^2<2^{-j}$, we have
\begin{align*}
\|(D^m \tilde u^j)(t,x) \|_{\infty} >j, \quad\forall\, t \in[t_j^1,t_j^2].
\end{align*}
Furthermore,
\begin{align*}
\sup_{0\le t\le 1} \| (D^m \tilde u^j)(t,x) \|_{L_x^{\infty}(|x|>1/2)} \le 1.
\end{align*}

\item Let $\tilde \omega^j = \nabla^{\perp} \cdot \tilde u^j$. Then
\begin{align*}
\max_{0\le t\le 1} \| D^{m-1} \tilde \omega^j(t,\cdot)\|_{\infty} \le 1.
\end{align*}

\end{itemize}

Define $f^j = \nabla^{\perp} \cdot g^j$. We then apply Proposition
\ref{s7_prop1} to conclude the proof.

\end{proof}

For the proof of Theorem \ref{thm2D_2Ck}, we need a simple perturbation lemma.

\begin{lem} \label{s7_lem2}
Let $f^i\in C_c^{\infty}(B(0,100))$, $i=1,2$.
Let $\omega^a$ and $\omega$ be smooth solutions to the
2D Euler equations in vorticity form:
\begin{align} \label{s7_lem2_t1}
\begin{cases}
\partial_t {\omega^a} + (u^a \cdot \nabla) {\omega^a}
 =0, \quad 0<t\le 1, \, x\in\mathbb R^2, \\
u^a=\Delta^{-1} \nabla^{\perp} \omega^a, \\
\omega^a \Bigr|_{t=0} =f^1.
\end{cases}
\end{align}
\begin{align} \label{s7_lem2_t2}
\begin{cases}
\partial_t \omega + (u \cdot \nabla)\omega
 =0,\quad 0<t\le 1,\, x\in \mathbb R^2, \\
u=\Delta^{-1} \nabla^{\perp} \omega, \\
\omega \Bigr|_{t=0} =f^1+f^2.
\end{cases}
\end{align}

For any $\epsilon>0$, there exists $\delta=\delta(\epsilon,f^1)>0$
sufficiently small such that if
\begin{align*}
\| f^2\|_{\infty}<\delta
\end{align*}
then
\begin{align} \label{s7_lem2_t3}
\max_{0\le t\le 1}\| \omega^a(t,\cdot) -\omega(t,\cdot)\|_{\infty}
<\epsilon.
\end{align}

\end{lem}

\begin{proof}[Proof of Lemma \ref{s7_lem2}]
Let $\eta= \omega^a-\omega$. Then $\eta(t=0)=f^2$, and
\begin{align*}
\partial_t \eta + (\Delta^{-1} \nabla^{\perp} \eta) \cdot \nabla \omega^a + (u\cdot \nabla) \eta =0.
\end{align*}
By finite speed propagation (and choosing $\|f^2\|_{\infty}$ small if necessary), easy to show that for any
$0\le t\le 1$, we have $\spp (\eta(t,\cdot)) \subset B(0,R_0)$ for some $R_0>0$ depending only on $f^1$. Then
easy to check that
\begin{align*}
\| \Delta^{-1} \nabla^{\perp} \eta (t,\cdot) \|_{\infty} \lesssim_{R_0} \|\eta(t,\cdot) \|_{\infty}.
\end{align*}
By a simple energy estimate, we also have
\begin{align*}
\max_{0\le t\le 1} \| \nabla \omega^a(t,\cdot) \|_{\infty} \lesssim_{f^1} 1.
\end{align*}
Using the equation for $\eta$, we then have
\begin{align*}
\| \eta(t,\cdot) \|_{\infty} \lesssim_{ f^1} \| f^2\|_{\infty} + \int_0^t \| \eta(s,\cdot) \|_{\infty} ds.
\end{align*}

The desired conclusion now easily follows from Gronwall and choosing $\|f^2\|_{\infty}$ small.

\end{proof}

We are now ready to prove

\begin{proof}[Proof of Theorem \ref{thm2D_2Ck}]
We shall only sketch the proof for $C^1$ ($m=1$) case. For $m\ge 2$ one can just use the material
from Section 5 and proceed in a similar fashion as in Section 6. Note that the patching argument for
$C^m$, $m\ge 2$ is actually easier in view of the flow decoupling.

Without loss of generality we may assume $u^{(g)}=0$.

For $j\ge 1$, define $z_j=(0, 1+\frac 1j)$. The point $z_j$ will be the center of $j^{\op{th}}$ patch.
Define $x_*= \lim_{j\to \infty} z_j = (0,1)$.

By recursively applying Lemma \ref{lem_ha1}, we can find
stream functions $\psi^j_0 \in C_c^{\infty} (B(z_j,2^{-j-100}))$ such that
\begin{align*}
\| \psi^j_0 \|_{\infty} + \| D^2 \psi^j_0\|_{\infty} <2^{-j}
\end{align*}
and the corresponding $j^{\op{th}}$-patch develops $C^1$-norm inflation in a time interval $<2^{-j}$ (we omit
the laborious details here since it is essentially a re-statement of Lemma
\ref{lem_ha1} with more explicit constants).

We then take initial stream function in the form
\begin{align*}
\psi_0 = \sum_{j=1}^{\infty} \psi^j_0,
\end{align*}
and $u^0= \nabla^{\perp} \psi_0$, $\omega_0 = \Delta \psi_0$. By using Lemma \ref{lem_ha1} together with
Lemma \ref{s7_lem2}  it is not difficult
to extract the needed regularity properties and prove the inflation statement. We omit further routine
details.

\end{proof}

\section{$3D$ $C^m$, $m\ge 2$ case: flow decoupling for axisymmetric flows without swirl}
We begin by reviewing a little bit the theory of axisymmetric flows on $\mathbb R^3$.
We call a scalar function $f=f(x_1,x_2,z):\, \mathbb R^3 \to \mathbb R$ axisymmetric if
$f=f(r,z)$, $r=\sqrt{x_1^2+x_2^2}$. An axisymmetric vector field $u$ on $\mathbb R^3$ has the
form
\begin{align*}
u(x_1,x_2,z)= u^r (r,z) e_r + u^{\theta}(r,z) e_{\theta} + u^z (r,z) e_z,
\end{align*}
where
\begin{align*}
e_r = \frac 1 r (x_1,x_2,z), \; e_{\theta}=\frac 1r (-x_2,x_1,0),\;
e_z=(0,0,1).
\end{align*}
If $u^{\theta} \equiv 0$, we then call $u$ an axisymmetric flow without swirl. In this case,
the vorticity $\omega=\nabla \times u$ becomes parallel to $e_{\theta}$:
\begin{align*}
\omega(x_1,x_2,z) = \omega^{\theta}(r,z) e_{\theta} = (\partial_{z} u^r - \partial_r u^z) e_{\theta},
\end{align*}
and the vorticity stretching term simplifies as well:
\begin{align*}
(\omega \cdot \nabla) u & = \frac 1 r u^r \omega^{\theta} e_{\theta} = \frac 1 r u^r \omega \notag \\
& = \frac 1 r (u\cdot e_r) \omega.
\end{align*}
From a technical point of view, the change $\nabla u \to \frac 1 r u^r$ brings a lot of simplification
in the perturbation theory as one can (for example) freely move the metric factor $\frac 1r$ to $\omega$ whenever
needed.

The vorticity equation then takes the form
\begin{align}
\partial_t \omega + (u\cdot \nabla) \omega &= \frac 1 r u^r \omega \notag \\
&= \frac 1r (u\cdot e_r ) \omega,
\label{ge_e1}
\end{align}
or more compactly,
\begin{align} \label{ge_e2}
\partial_t \bigl( \frac {\omega} r \bigr) + (u\cdot \nabla) \bigl( \frac {\omega} r \bigr)=0.
\end{align}
We shall frequently switch between the expressions \eqref{ge_e1} and \eqref{ge_e2} below without explicit
mentioning. The form \eqref{ge_e1} has the advantage that it reacts well with the usual Cartesian derivatives.
On the other hand the form \eqref{ge_e2} can be used to deduce (easily) the conservation laws.

We recall the following two lemmas which we will often use without explicit mentioning.

\begin{lem}[$L^{p,q}$-preservation] \label{lem_Lpq_1}
Let $1\le p,q\le\infty$. Suppose $u$ is a given smooth divergence-free vector field on $\mathbb R^d$, $d\ge 2$.
Let $h$ be the smooth solution to the transport equation
\begin{align*}
\begin{cases}
\partial_t h + (u \cdot \nabla) h =f, \\
h\Bigr|_{t=0}=h_0.
\end{cases}
\end{align*}
Then for any $t>0$, we have
\begin{align*}
\| h(t) \|_{L^{p,q}(\mathbb R^d)} \le \|h_0 \|_{L^{p,q}(\mathbb R^d)} +
 \int_0^t \| f(\tau)\|_{L^{p,q}(\mathbb R^d)}
d \tau.
\end{align*}
If $f\equiv 0$, then
\begin{align*}
\| h(t) \|_{L^{p,q}(\mathbb R^d)} = \|h_0\|_{L^{p,q}(\mathbb R^d)}.
\end{align*}
\end{lem}

\begin{proof}[Proof of Lemma \ref{lem_Lpq_1}]
See for example Proposition 2 on p484 of Danchin \cite{Danchin07} or Prop 2.2 of Abidi-Hmidi-Keraani \cite{AHK10}.
\end{proof}

\begin{lem}[Axisymmetric Biot-Savart law: estimate of $u^r$] \label{lem_ur_r}
There exists an absolute constant $C>0$ such that
\begin{align*}
\left\| \frac{u^r} r \right\|_{L^{\infty}(\mathbb R^3)} \le C \left\| \frac{\omega^{\theta}}r
\right\|_{L^{3,1}(\mathbb R^3)},
\end{align*}
where $u=u^r e_r+u^z e_z$, $\omega=\nabla \times u= \omega^{\theta} e_{\theta}$.
\end{lem}
\begin{proof}[Proof of Lemma \ref{lem_ur_r}]
See Proposition 3.1 of \cite{AHK10}. The key idea is to use the kernel
estimate (see Lemma 1 from \cite{SY94}) of the form
\begin{align*}
|u^r (x) | \lesssim \int_{|y-x|\le r } \frac{|\omega(y)|}{|x-y|^2} dy
+ r \int_{|y-x| \ge r } \frac{|\omega(y)|}{|x-y|^3} dy.
\end{align*}
\end{proof}

Analogous to the $2D$ case, we shall prove a perturbation lemma for decoupling the flow map. In order not
to obscure the main ideas, we shall just
state and prove the case for $C^2$. The general case $m\ge 2$ is a simple modification of numerology which
we leave it to interested readers.

Consider the following systems:
\begin{align*}
 \begin{cases}
  \partial_t \bigl( \frac {\omega^{(l)}} r \bigr) + (u^{(l)} \cdot \nabla) \bigl(
  \frac {\omega^{(l)}} r \bigr) =0, \\
  u^{(l)} = -\Delta^{-1} \nabla \times \omega^{(l)}, \\
  \omega^{(l)}\Bigr|_{t=0}=\omega_0^{(l)}.
 \end{cases}
\end{align*}

\begin{align*}
 \begin{cases}
  \partial_t \bigl(\frac{\omega}r\bigr) + (u\cdot \nabla) \bigl( \frac{\omega} r \bigr)=0, \\
  u= -\Delta^{-1} \nabla \times  \omega,\\
  \omega \Bigr|_{t=0}= \omega_0^{(l)} + \omega_0^{(h)};
 \end{cases}
\end{align*}

Let $\tilde \omega$ solve the \emph{linear} system
\begin{align*}
 \begin{cases}
  \partial_t \bigl( \frac{\tilde \omega} r\bigr) + (u^{(l)} \cdot \nabla) \bigl( \frac{\tilde \omega}r \bigr) =0, \\
  \tilde \omega \Bigr|_{t=0}= \omega_0=\omega_0^{(l)} + \omega_0^{(h)}.
 \end{cases}
\end{align*}

Let $\tilde u= -\Delta^{-1} \nabla \times \tilde \omega $ be the velocity corresponding to $\tilde \omega$. Let
$u_0^{(l)}$, $u_0^{(h)}$ be the velocities corresponding to $\omega_0^{(l)}$, $\omega_0^{(h)}$ respectively.
We assume $0<T_0 \lesssim 1$, and for some $r_1>0$,
\begin{align*}
 & \operatorname{supp}( \omega_0^{(l)}) \subset\{(r,z): \; r_1<r \lesssim 1, \, |z| \lesssim 1 \}, \\
& \operatorname{supp}( \omega_0^{(h)}) \subset\{(r,z): \; r_1<r  \lesssim 1, \, |z| \lesssim 1 \}, \\
& \| u_0^{(l)}\|_2 + \| u_0^{(h)} \|_2 + \| \frac{\omega_0^{(l)}} r \|_{\infty} + \|
 \frac{\omega_0^{(h)}} r \|_{\infty} \lesssim 1.
\end{align*}
These conditions guarantee that on the time interval $[0,T_0]$,
\begin{align*}
 &\| u(t) \|_{H^k} \lesssim C_k \| u(0) \|_{H^k}, \qquad k\ge 3,\\
& \|D\omega(t) \|_{\infty} \lesssim 1+ \| D\omega_0 \|_{\infty},\\
& \| D^2 \omega (t) \|_{\infty} \lesssim 1+\| D^2 \omega_0 \|_{\infty},
\end{align*}
where $C_k$ is some constant depending on $k$.

\begin{lem} \label{lem_ge1}
\begin{align}
 & \max_{0\le t\le T_0} \| D^2 u(t,\cdot) - D^2 \tilde u(t,\cdot)\|_{\infty} \notag \\
 \lesssim & \Bigl(\log(3+ \| u_0^{(h)}\|_{H^{10}} + \| u_0^{(l)} \|_{H^{10}})\Bigr) \cdot
 e^{10 T_0 \| D u^{(l)} \|_{L_t^{\infty} L_x^{\infty} ([0,T_0])} } \notag \\
 & \quad \cdot\Bigl( (1+\|D\omega_0\|_{\infty} )
 \| u_0^{(h)} \|_2^{\frac 2 7}  ( \max_{0\le t \le T_0}
 \| D^2 u(t) \|_{\infty} + \max_{0\le t \le T_0} \| D^2 u^{(l)} (t) \|_{\infty} )^{\frac 57} \notag \\
& \quad +(1+\|D^2\omega_0\|_{\infty} )
 \| u_0^{(h)} \|_2^{\frac 4 7}  ( \max_{0\le t \le T_0}
 \| D^2 u(t) \|_{\infty} + \max_{0\le t \le T_0} \| D^2 u^{(l)} (t) \|_{\infty} )^{\frac 37}\Bigr). \notag
 \end{align}
\end{lem}

\begin{proof}[Proof of Lemma \ref{lem_ge1}]
We start with
\begin{align*}
\begin{cases}
\partial_t \omega + u \cdot \nabla \omega = \frac 1 r (u\cdot e_r) \omega, \\
\partial_t \tilde {\omega} + u^{(l)} \cdot \nabla \tilde \omega
=\frac 1 r (u^{(l)} \cdot e_r) \tilde \omega, \\
\omega \Bigr|_{t=0} = \tilde \omega \Bigr|_{t=0} = \omega_0^{(l)} + \omega_0^{(h)}.
\end{cases}
\end{align*}

Set $\eta= \omega-\tilde \omega$. Then
\begin{align*}
&\partial_t \eta + (u-u^{(l)}) \cdot \nabla \omega + u^{(l)} \cdot \nabla \eta \notag \\
&\quad = \frac 1 r ( (u-u^{(l)})\cdot e_r) \omega + \frac 1 r (u^{(l)} \cdot e_r) \eta.
\end{align*}
Taking the derivative (here $\partial$ denotes any one of the derivatives $\partial_{x_1}$,
$\partial_{x_2}$ or $\partial_z$) gives
\begin{align*}
&\partial_t \partial \eta + \partial( u-u^{(l)}) \cdot \nabla \omega
+ (u-u^{(l)}) \cdot \nabla \partial \omega + \partial u^{(l)} \cdot \nabla \eta
+ (u^{(l)} \cdot \nabla) (\partial \eta) \notag \\
& \quad = \partial( \frac 1 r (u-u^{(l)}) \cdot e_r ) \omega +
\frac 1 r ((u-u^{(l)}) \cdot e_r) \partial \omega \notag \\
& \qquad + \partial( \frac 1 r (u^{(l)}\cdot e_r) )  \eta +
\frac 1 r ( u^{(l)} \cdot e_r) \partial \eta.
\end{align*}

Note that for any axisymmetric function $f$ with
 $f(0,z)\equiv 0$, we have
\begin{align*}
\| \frac f r \|_{\infty} = \| \frac { f(r,z) -f(0,z)} r \|_{\infty}
 \le \| \partial_r f \|_{\infty} \le \| D f \|_{\infty}.
\end{align*}

Then\footnote{Note that if $v(x_1,x_2,z)=v^r(r,z)e_r+v^z(r,z)e_z$
and $v$ is smooth, then $v^r(0,z)\equiv 0$ since otherwise $v$ will
not be smooth at $(0,0,z)$. Similarly if
$\omega(x_1,x_2,z)=\omega^{\theta}(r,z)e_{\theta}$ is smooth, then
$\omega(0,0,z)\equiv 0$. These facts were used in the derivations.}
clearly,
\begin{align*}
&\max_{0\le t \le T_0} \| D \eta \|_{\infty} \notag \\
\lesssim & e^{8 T_0 \| D u^{(l)} \|_{L_t^{\infty}L_x^{\infty} ([0,T_0])} }
\max_{0\le t\le T_0}
\bigl( \| D (u-u^{(l)}) \|_{\infty} \|D \omega\|_{\infty}
+ \| u - u^{(l)} \|_{\infty} \| D^2 \omega \|_{\infty} \bigr).
\end{align*}

Now since
\begin{align*}
\partial_t (u-u^{(l)}) + (u-u^{(l)}) \cdot \nabla u^{(l)} + u \cdot \nabla (u-u^{(l)})
=-\nabla (p-p^{(l)}),
\end{align*}
we get
\begin{align*}
\max_{0\le t \le T_0} \| u-u^{(l)} \|_2 \lesssim e^{ T_0\| D u^{(l)}\|_{L_t^{\infty}L_x^{\infty}([0,T_0])}}
\| u_0^{(h)} \|_2.
\end{align*}

By using the interpolation inequalities (applied to $u-u^{(l)}$)
\begin{align*}
&\| D f\|_{L_x^{\infty}(\mathbb R^3)} \lesssim \| f \|_{L_x^2(\mathbb R^3)}^{\frac 27}
  \| D^2 f\|_{L_x^{\infty}(\mathbb R^3)}^{\frac 57}, \\
& \| f \|_{L_x^{\infty}(\mathbb R^3)}
 \lesssim \| f \|_{L_x^2(\mathbb R^3)}^{\frac 47}  \| D^2 f \|_{L_x^{\infty}}^{\frac 37},
 \end{align*}
we get
\begin{align*}
  & \max_{0\le t\le T_0} \| D( u - u^{(l)} ) \|_{\infty} \notag \\
  \lesssim & \max_{0\le t\le T_0}
  \Bigl(  \| u- u^{(l)} \|_2^{\frac 27}
   ( \| D^2 u \|_{\infty} + \| D^2 u^{(l)} \|_{\infty} )^{\frac 57} \Bigr) \notag \\
    \lesssim & e^{\frac 2 7 T_0 \| Du^{(l)} \|_{L_t^{\infty}L_x^{\infty}([0,T_0])}}
    \| u_0^{(h)}\|_2^{\frac 27}
    ( \max_{0\le t\le T_0} \|D^2 u(t) \|_{\infty}
    + \max_{0\le t\le T_0} \| D^2 u^{(l)} (t) \|_{\infty} )^{\frac 57};
    \end{align*}
and
\begin{align*}
 & \max_{0\le t \le T_0} \| u - u^{(l)} \|_{\infty} \notag \\
 \lesssim & e^{\frac 4 7 T_0 \| D u^{(l)} \|_{L_t^{\infty} L_x^{\infty} ([0,T_0])}}
 \| u_0^{(h)} \|_2^{\frac 47}
 (\max_{0\le t\le T_0} \| D^2 u(t)\|_{\infty}
 + \max_{0\le t\le T_0} \| D^2 u^{(l)}(t) \|_{\infty} )^{\frac 37}.
 \end{align*}

Hence the result follows from the usual log interpolation inequality (to bound $\|D^2 (u-\tilde u)\|_{\infty}$ in terms
of the product of $\|D\eta\|_{\infty}$ and a log-term).

\end{proof}

\section{Local $C^m$, $m\ge 2$ norm inflation: argument for the $3D$ case}
To suppress numerology, we shall just give the details for the case $m=2$. The general $m\ge 2$ is
 a simple modification (similar to what was done at the end of Section \ref{sec_local_2Dcm}).

We start with a general derivation.  To produce $C^2$-norm inflation, we will examine the quantity
$(\partial_{zz} u^{z})(t,0)$ which belongs to
one of the entries of $D^2 u(t)$. By using $\omega^{\theta}= \partial_r u^z - \partial_z u^r$ and
the incompressibility condition
\begin{align*}
\frac 1 r \partial_r (r u^r) + \partial_z u^z =0,
\end{align*}
it is easy to derive
\begin{align}
u^z = \Delta^{-1} \bigl( \frac 1 r \partial_r (r \omega^{\theta}) \bigr), \label{ge_e50a}
\end{align}
where $\Delta^{-1}=-\operatorname{const}|x|^{-1}*$ ($*$ denotes the usual convolution).

Similar to the 2D case we shall choose initial data $\omega_0$ with support away from the $r=0$ axis.
By differentiating \eqref{ge_e50a} twice, we get (below $C_1>0$ is
an absolute constant)
\begin{align}
(\partial_{zz} u^z)(t,0)&= C_1 \int \frac {r^2 -2z^2} {(r^2+z^2)^{\frac 52}} \frac 1r
\partial_r ( r \omega^{\theta} ) \cdot r dr dz \notag \\
& = C_1 \int \frac {3 r (r^2 - 4 z^2)} {(r^2 + z^2)^{\frac 72}}  \omega^{\theta}(t,r,z) r dr dz. \label{ge_e50b}
\end{align}

Note in the above derivation we only used the Biot-Savart law in axisymmetric form
(to express $u$ in terms of $\omega^{\theta}$) and no dynamics is used yet.  Now assume that $\omega^{\theta}$ obeys
the equation:
\begin{align*}
\begin{cases}
\partial_t \bigl( \frac {\omega^{\theta}} r \bigr) + (V \cdot \nabla)\bigl( \frac {\omega^{\theta}} r \bigr)=0, \\
\omega^{\theta} \Bigr|_{t=0} =\omega_0^{\theta}.
\end{cases}
\end{align*}
Here $V=V^r(t,r,z) e_r + V^z(t,r,z) e_z$ is a \emph{given} velocity field. Define the axisymmetric
characteristics lines  $\phi(t,r,z)=(\phi^r(t,r,z), \phi^z(t,r,z))$, such that
\begin{align*}
\begin{cases}
\partial_t \phi (t,r,z) = (V^r(t, \phi(t,r,z)), V^z(t,\phi(t,r,z)) ), \\
\phi (0,r,z) =(r,z).
\end{cases}
\end{align*}
Then
\begin{align*}
 \frac { \omega^{\theta}(t, \phi(t,r,z))} {\phi^r(t,r,z)} = \frac{\omega_0^{\theta}(r,z)} r, \quad \forall\, t\ge 0, \, r>0, \, z\in \mathbb R.
\end{align*}
Using the above relation, we then make a change of variable\footnote{Note that the map $(r,z) \to \phi(t,r,z)$ preserves
the measure $rdrdz$.} $(r,z) \to \phi(t,r,z)$ in \eqref{ge_e50b} to get,
\begin{align}
 (\partial_{zz} u^z)(t,0) &= C_1 \int F_1( \phi(t,r,z) ) \frac{\omega_0^{\theta}(r,z)} r  rdr dz, \notag \\
 &= C_1 \int F_1( \phi(t,r,z) ) {\omega_0^{\theta}(r,z)}  dr dz, \label{ge_e50c}
\end{align}
where
\begin{align}
 F_1(r,z) =  \frac {3 r^2 (r^2 - 4 z^2)} {(r^2 + z^2)^{\frac 72}}.  \notag
\end{align}
We shall use the above formula in the computation below.

We now specify the form of initial data for producing $C^2$-norm inflation in the 3D Euler equation.
Take initial (axisymmetric) stream function\footnote{Note that the velocity-stream relation is
different from  2D. Due to the incompressibility condition $\partial_r( r u^r) + \partial_z (r u^z)=0$,
the velocity-stream relation
takes the form $u^r = \frac 1 r (-\partial_z) \psi$, $u^z= \frac 1 r \partial_r \psi$, i.e. there is a new metric
factor $\frac 1r$ in front of the differentiation.} of the form
\begin{align*}
\psi_0(r,z)= \psi_0^{(l)}(r,z) + \psi_0^{(h)}(r,z),
\end{align*}
where $\psi_0^{(l)}$ will produce the Lagrangian deformation and $\psi_0^{(h)}$ has the
expansion\footnote{Because of the new velocity-stream relation and the metric factor $\frac 1r$, we have
here $2^{-4i}$ instead of $2^{-3i}$.}
\begin{align*}
\psi_0^{(h)}(r,z) = \sum_{M\le j\le M+\sqrt M} 2^{-4j} a_0(2^j(r,z)).
\end{align*}
More assumptions on $\psi_0^{(l)}$ and $a_0$ will be clearly
specified later. For the moment we just assume $a_0$ is smooth and
compactly supported on $\{(r,z): \, \rho_1<r<\rho_2,
\rho_3<r+|z|<\rho_4 \}$ for some $0<\rho_1<\rho_2<2\rho_1$,
$0<\rho_3<\rho_4<2\rho_3$.  This way the replicates $a_0(2^j(r,z))$
will have non-overlapping supports. We also assume
$\omega_0^{\theta,(l)}$ is compactly supported away from the $r=0$
axis.

Then recalling  $u^r =\frac 1 r (-\partial_z)\psi$, $u^z = \frac 1 r
\partial_r \psi$, and $\omega^{\theta}= \partial_z u^r-\partial_r
u^z$, we get
\begin{align*}
&u_0^{r,(h)}(r,z) = \sum_{M\le j\le M+\sqrt M} 2^{-3j} \frac 1 r (-\partial_z a_0)(2^j(r,z)),\\
&u_0^{z,(h)}(r,z) = \sum_{M\le j\le M+\sqrt M} 2^{-3j} \frac 1 r (\partial_r a_0)(2^j (r,z)),\\
&\omega_0^{\theta,(h)} (r,z) = \sum_{M\le j\le M+\sqrt M} 2^{-j} b_0(2^j(r,z)),
\end{align*}
where
\begin{align*}
b_0(r,z)=\frac 1 r \bigl(- \partial_{zz} a_0 -\partial_{rr} a_0 +
\frac 1 r \partial_r a_0 \bigr).
\end{align*}

Let $u_0=u_0^{(l)}+u_0^{(h)}$ and denote $\tilde u$ as in the setting of Lemma \ref{lem_ge1}. By Lemma
\ref{lem_ge1}, it is not difficult to check that we the main part of $(\partial_{zz} u^z)(t,0)$ is
given by the quantity (see formula \eqref{ge_e50c})
\begin{align}
&C_1 \sum_{M\le j\le M+\sqrt M} \int F_1 (\phi^{(l)}(t,r,z) ) 2^{-j} b_0(2^j(r,z)) dr dz \notag \\
= & C_1 \sum_{M \le j\le M+\sqrt M} \int F_1( 2^j \phi^{(l)}(t,2^{-j} r, 2^{-j}z) ) b_0(r,z) dr dz. \notag
\end{align}

As $\lambda \to \infty$, we have\footnote{Similar to the 2D case, we
may (for simplicity of presentation) assume $\phi^{(l)}(t,0,0)\equiv 0$. Otherwise, one can just
shift to the point $\phi^{(l)}(t,0,0)$ and consider $(\partial_{zz}
u^z)(t,\phi^{(l)}(t,0,0))$.}
\begin{align}
 &\int F_1 ( \lambda \phi^{(l)} (t , \lambda^{-1} r, \lambda^{-1} z ) )  b_0(r,z) dr dz \notag \\
\to & \int F_1 ( A(t) \begin{pmatrix} r \\ z \end{pmatrix} ) b_0(r,z) dr dz, \label{ge_e60a}
\end{align}
where $A(t) = D \phi^{(l)}(t,0)$ (here $D \phi^{(l)}$ denotes the Jacobian matrix in $\partial_r$ and $\partial_z$).
Now we only need to specify $\psi^{(l)}_0$, $a_0$ and $t$ such that \eqref{ge_e60a} is nonzero.
The simplest choice of $\psi^{(l)}_0$ is such that the dynamics near $(r=0,z=0)$ is hyperbolic, namely
\begin{align*}
\partial_t (D \phi^{(l)})(t,0) =
\begin{pmatrix}
(\partial_r u^{r,(l)} )(t,0,0) \quad 0\\
0 \quad (\partial_z u^{z,(l)} )(t,0,0)
\end{pmatrix} D\phi^{(l)}(t,0).
\end{align*}
From the incompressibility condition $\frac 1 r\partial_r (r u^{r,(l)}  )+ \partial_z u^{z,(l)} =0$, we have
$$2\partial_r u^{r,(l)} (t,0,0) + \partial_z u^{z,(l)}(t,0,0)=0.$$  Therefore one can choose $\psi^{(l)}_0$ such that
\begin{align*}
A(t) \begin{pmatrix} r \\z
\end{pmatrix}
= \begin{pmatrix}
\lambda(t) r \\
\frac 1 {\lambda(t)^2} z
\end{pmatrix},
\end{align*}
where $\lambda(t)>1$, $\lambda(t)=1+$ for $t=0+$. Consider now for $\lambda>1$, the integral
\begin{align*}
 &\int F_1(\lambda r, \lambda^{-2} z) b_0(r,z) dr dz \notag \\
= & \int F_1 (\lambda r, \lambda^{-2} z) \frac 1 r \bigl(
-\partial_{zz} a_0 - \partial_{rr} a_0 +\frac 1 r \partial_r a_0
\bigr) dr dz.
\end{align*}
It is not difficult to check that
\begin{align*}
 & \int \partial_{\lambda} \Bigl( F_1 (\lambda r, \lambda^{-2} z) \Bigr)\Bigr|_{\lambda=1}
\frac 1 r \bigl( -\partial_{zz} a_0 - \partial_{rr} a_0 +\frac 1 r \partial_r a_0 \bigr) dr dz\\
= & \int F_2(r,z) a_0(r,z)dr dz,
\end{align*}
where
\begin{align*}
F_2(r,z)= -270 \frac { r (r^4 - 12 r^2 z^2 + 8 z^4)} {(r^2 +
z^2)^{\frac {11}2}}.
\end{align*}
Now clearly we only need to choose $a_0(r,z)$ to be supported in a sufficiently small neighborhood
around the point $(r=1,z=0)$. With such a choice, we have for $M$ sufficiently large,
\begin{align*}
|\partial_{zz} u^z(t,0)| > \operatorname{const} \sqrt M,
\end{align*}
producing the desired (local) $C^2$-norm inflation.

\section{Local $C^1$-norm inflation for $3D$ case} \label{sec_3DC1}
We shall again work with axisymmetric flows without swirl. We will retain similar notations as in the previous
section and only sketch the construction since it strongly parallels with the 2D case.
We shall examine
the quantity $\partial_z u^z(t,0)$. Recall that $u^z=\Delta^{-1}(\frac 1 r \partial_r (r\omega^{\theta}) )$.
Easy to check then
\begin{align*}
(\partial_z u^z)(t,0)
&= C_2 \int \frac{z}{ (r^2+z^2)^{\frac 32}} \frac 1 r \partial_r (r\omega^{\theta}) r dr dz \notag \\
&= 3C_2 \int \frac {r^2 z}{(r^2+z^2)^{\frac 52}} \frac {\omega^{\theta}(t,r,z)} r
 r dr dz,
 \end{align*}
 where $C_2>0$ is an absolute constant.
  By using the conservation of $\omega^{\theta} /r $ on characteristic lines, we then have
 \begin{align}
 (\partial_z u^z)(t,0) = 3 C_2 \int F_3( \phi(t,r,z) ){\omega_0(r,z)} dr dz, \label{ge_65a}
 \end{align}
where
\begin{align*}
F_3(r,z)=\frac {r^2 z}{(r^2+z^2)^{\frac 52}}.
\end{align*}

Now take initial axisymmetric stream function in the form
\begin{align} \label{ge_66aa}
\psi_0(r,z) = \sum_{100\le j\le M} 2^{-3j} a_3(2^j(r,z)),
\end{align}
where $a_3$ is supported on $r\sim 1$, $r+|z|\sim 1$ such that $a_3(2^j(r,z))$ have non-overlapping supports.
More conditions on $a_3$
 will be specified later. In what follows, to simplify notation we will often write $\sum_{100\le j\le M}$
simply as $\sum_{j\le M}$. Now
\begin{align*}
&u_0^r(r,z) = \sum_{j\le M} 2^{-2j} \frac 1 r ( -\partial_z a_3)(2^j(r,z)), \\
&u_0^z(r,z) = \sum_{j\le M} 2^{-2j} \frac 1 r (\partial_r a_3)(2^j(r,z)).
\end{align*}
Consequently
\begin{align*}
\omega_0^{\theta}(r,z) = -\sum_{j\le M} b_3(2^j(r,z)),
\end{align*}
where
\begin{align*}
b_3(r,z)=\frac 1 r \bigl( \partial_{zz} a_3 + \partial_{rr} a_3- \frac 1 r \partial_r a_3 \bigr).
\end{align*}

The formula \eqref{ge_65a} then becomes
\begin{align}
- \frac 1 {3C_2} (\partial_z u^z)(t,0)
& = \sum_{j\le M} \int F_3(\phi(t,r,z)) b_3(2^j(r,z)) dr dz \notag \\
& = \sum_{j\le M} \int F_3( 2^j \phi(t, 2^{-j}r, 2^{-j} z)) b_3(r,z) dr dz. \label{ge_66a}
\end{align}

Now to avoid problems of estimating the pressure in axisymmetric coordinates, we switch to Euclidean Characteristic
lines. Define
\begin{align*}
\begin{cases}
\partial_t \Phi(t,x_1,x_2,z)= u(t, \Phi(t,x_1,x_2,z)),\\
\Phi(0,x)=x.
\end{cases}
\end{align*}

For $x=(x_1,x_2,z)$, we denote $x^{\prime}=(x_1,x_2,0)$. Then
\begin{align}
\text{RHS of \eqref{ge_66a}} = \sum_{j\le M} \int F_3(2^j \Phi(t,2^{-j} x) ) \frac{b_3(x)}{|x^{\prime}|} dx.
\label{ge_66b}
\end{align}
Here and below, we will slightly abuse the notation and denote any axisymmetric function $f=f(r,z)$ simultaneously
by $f=f(x)=f(x_1,x_2,z)$.

For simplicity of presentation, assume that\footnote{This can be easily achieved by choosing $\psi_0(r,z)$
to be an odd function
of $z$. Of course this kind of property is no longer available
 in the patching argument later. There one need to shift to the point $\Phi(t,0)$ similar to what was done  in the
2D case.}
 $u(t,x=0)\equiv 0$  and denote $$\|D u(t)\|_{L_t^{\infty}L_x^{\infty}([0,1])}=A \ge 1.$$
  Then by essentially repeating the
proof of Lemma \ref{lem_ga10} and Lemma \ref{lem_ga10a},
we get for any $\lambda\ge 3$, $0\le t \le \frac 1 A$, $|x| \lesssim 1$,
\begin{align*}
  &  |\lambda \Phi(t,\frac 1 {\lambda} x) - x - t \lambda u_0(\frac 1 {\lambda } x ) | \notag \\
  \lesssim & \; A^2 t^2 e^{tA}  \log(3+ \|u_0\|_{H^6}).
 \end{align*}
Also in place of Lemma \ref{lem_ga20}, we have
\begin{align*}
  &  \Bigl|\lambda \Phi(t,\frac 1 {\lambda} x) - x - t \lambda u_0(\frac 1 {\lambda } x )-\lambda \int_0^t (t-\tau)
  (\nabla p)(\tau, \frac {x} \lambda) d\tau \Bigr| \notag \\
  \lesssim &\; t^3 A^3 \log(3+\|u_0\|_{H^6}).
 \end{align*}
For $\lambda =2^j$, $j\le M$, and $x \in \operatorname{supp}(a_3)$, easy to check that
\begin{align*}
\lambda u_0 ( \frac 1 {\lambda} x) =
- \frac 1 r (\partial_z a_3)(r,z) e_r + \frac 1 r (\partial_r a_3)(r,z) e_z =:H(x).
\end{align*}

Now similar to the 2D case, we can Taylor expand the integrand $F_3(\cdot)$ in \eqref{ge_66b}
around the point $x$. This will introduce the main term (for producing norm inflation)
\begin{align}
\sum_{j\le M} t \int (\nabla F_3)(x) \cdot H(x) \frac{b_3(x)}{|x^{\prime}|} dx, \label{ge_68a}
\end{align}
and the pressure error term
\begin{align}
\sum_{j \le M} \int \nabla F_3(x) \cdot \int_0^t (t-\tau) 2^j(\nabla p)(\tau, 2^{-j}x) d\tau
\frac{b_3(x)}{|x^{\prime}|} dx, \label{ge_68b}
\end{align}
and negligible error terms (provided we choose $t=N/M$, with $1\ll N\ll M$).

Now similar to the 2D case, we can bound \eqref{ge_68b} by $\operatorname{const}\cdot M A^2 t^2$, provided
we satisfy the condition
\begin{align*}
\Delta^{-1} \nabla \cdot ( \nabla F_3(x) \frac{b_3(x)} {|x^{\prime}|} ) \in L_x^1(\mathbb R^3).
\end{align*}
Note that for any vector function $g=(g_1,g_2,g_3) \in C_c^{\infty}(\mathbb R^3)$,
$\Delta^{-1} \nabla \cdot g \in L_x^1(\mathbb R^3)$ if
\begin{align}
& \int_{\mathbb R^3} g(y) dy =0, \label{ge_67a} \\
& \int_{\mathbb R^3} (y\cdot g(y)) dy =0,  \\
& \int_{\mathbb R^3} (x\cdot y) (x\cdot g(y)) dy =0, \quad \forall\, x \in \mathbb R^3.
\label{ge_67c}
\end{align}
The conditions \eqref{ge_67a}--\eqref{ge_67c} are equivalent to
\begin{align}
 & \int_{\mathbb R^3} g(y) dy =0,  \notag \\
 & \int_{\mathbb R^3} y_k g_k(y) dy=0, \quad \forall\, 1\le k\le 3, \notag \\
 & \int_{\mathbb R^3} (y_k g_l(y)+y_l g_k(y)) dy =0, \quad \forall\, k \ne l. \notag
\end{align}
For $g(x)=\nabla F_3 \frac{b_3(x)}{|x^{\prime}|}$, recalling that $F_3(x)=r^2 z/ (r^2+z^2)^{5/2}$,
$b_3=\frac 1 r (\partial_{zz} a_3 + \partial_{rr} a_3 - \frac 1 r \partial_r a_3)$ and $a_3$ is an odd function of $z$,
we then obtain the equivalent condition (after a tedious computation)
\begin{align}
& \int \frac{ z (3 r^3 - 4 r z^2)}  {(r^2 + z^2)^{\frac 92}} a_3(r,z)  dr dz=0. \label{ge_69_8.9}
\end{align}
We need to find $a_3$ such that the above integral vanishes.
We also need the integral in \eqref{ge_68a} to be nonzero. Easy to check that
\begin{align*}
 \nabla F_3 \cdot H = \frac 1 r ( - \partial_r F_3 \partial_z a_3 + \partial_r a_3 \partial_z F_3).
\end{align*}
Then
\begin{align}
 & \int \nabla F_3 (x) \cdot H(x) \frac{b_3(x)} {|x^{\prime}|} dx \notag \\
 = & \; \int \frac 1 {r^2} (- \partial_r F_3 \partial_z a_3 + \partial_r a_3 \partial_z F_3)  \cdot
 (\partial_{zz}a_3 + \partial_{rr} a_3 - \frac 1 r \partial_r a_3) dr dz. \notag
\end{align}
After several integration by parts, we obtain
\begin{align}
 & \int \nabla F_3 (x) \cdot H(x) \frac{b_3(x)} {|x^{\prime}|} dx \notag \\
= & \int (\partial_r a_3)^2\cdot  \bigl( - \frac 1 {r^3} \partial_z F_3 - \frac 12 \partial_r ( \frac 1 {r^2} \partial_z F_3)
-\frac 12 \partial_z ( \frac 1 {r^2} \partial_r F_3) \bigr) dr dz \notag \\
& \quad + \int (\partial_z a_3)^2 \cdot \bigl( \frac 12 \partial_z(\frac 1 {r^2} \partial_r F_3) + \frac 12 \partial_r (\frac 1{r^2}
\partial_z F_3) \bigr) dr dz \notag \\
& \quad + \int \partial_z a_3 \partial_r a_3 \bigl( \frac 1 {r^3}
\partial_r F_3 - \frac 1 {r^2} \partial_{zz} F_3 + \partial_r ( \frac 1 {r^2} \partial_r F_3) \bigr) dr dz.
\notag \\
= & \int (\partial_r a_3)^2 \cdot \frac {3 r^4 - 24 r^2 z^2 + 8 z^4} { r (r^2 + z^2)^{\frac 92} } dr dz \notag \\
&\quad + \int (\partial_z a_3)^2 \cdot \frac{ -4 r^4 + 27 r^2 z^2 - 4 z^4 }{ r (r^2 + z^2)^{\frac 9 2}}        drdz        \notag \\
& \quad + \int \partial_z a_3 \partial_r a_3
\frac{10 z(3r^2-4z^2)}  { (r^2 + z^2)^{\frac 92} } dr dz.
\label{ge_69_8.10}
\end{align}

Now Let $\varphi_0 \in C_c^{\infty}(\mathbb R)$ be such that $0\le \varphi_0(z) \le 1$,
$\varphi_0(z)=1$ for $|z| \le 1$ and $\varphi_0(z)=0$ for $|z|>2$.
For $0\le s \le 1$, consider the function
\begin{align*}
h(s,r,z) =(1-s)  \underbrace{z\varphi_0( \frac{r -\frac 2 {\sqrt 3}} {\delta}) \varphi_0( \frac{|z|-(1-\epsilon)} {\epsilon^2} )}_{h_0(r,z)}
+s \underbrace{z\varphi_0( \frac{r -\frac 2 {\sqrt 3}} {\delta}) \varphi_0( \frac{|z|-(1+\epsilon)} {\epsilon^2} )}_{h_1(r,z)}.
\end{align*}

We shall choose $0<\delta\ll \epsilon$.  Here the special value $r_0=\frac 2 {\sqrt 3}$ is such that $3r_0^3=4r_0$ in the
integrand of \eqref{ge_69_8.9}. Easy to check that for $\delta\ll \epsilon \ll 1$,
\begin{align*}
& \int \frac{ z (3 r^3 - 4 r z^2)}  {(r^2 + z^2)^{\frac 92}} h_0(r,z)  dr dz>0, \\
& \int \frac{ z (3 r^3 - 4 r z^2)}  {(r^2 + z^2)^{\frac 92}} h_1(r,z)  dr dz<0. \\
\end{align*}
Hence by the  Intermediate Value Theorem, there exists $s_0 \in (0,1)$ such that
\begin{align*}
 & \int \frac{ z (3 r^3 - 4 r z^2)}  {(r^2 + z^2)^{\frac 92}} h(s_0, r,z)  dr dz=0,
\end{align*}
i.e. \eqref{ge_69_8.9} is fulfilled by choosing $a_3(r,z) = h(s_0,r,z)$.

Now return to  \eqref{ge_69_8.10}. Easy to check that uniformly in $0\le s\le 1$, we have
\begin{align*}
&\Bigl| \int |(\partial_r h(s,r,z))|^2 \frac {3 r^4 - 24 r^2 z^2 + 8 z^4} { r (r^2 + z^2)^{\frac 92} } dr dz \Bigr|\notag\\
\ge & \operatorname{const} \cdot \delta^{-1} \epsilon^2.
\end{align*}
The other two terms in \eqref{ge_69_8.10} are $O(1)$ (in terms of $\delta$). Therefore by choosing $\delta$ sufficiently small we
can make the integral \eqref{ge_69_8.10} non-zero.

Collecting all the estimates, we conclude that for $t= N/M$,  $1\ll N\ll M$, we have the bound
\begin{align*}
 |\partial_z u^z (t,0)| \ge \operatorname{const} \cdot t M = \op{const} \cdot N \gg 1.
\end{align*}
This produces the desired $C^1$-norm inflation for the $3D$ case.

\section{patching for 3D $C^1$ case}
Let $U^{\op{ext}}=U^{\op{ext}}(t,x_1,x_2,z)$ be a given smooth velocity field on $\mathbb R^3$
which is axisymmetric without swirl, i.e.
\begin{align*}
U^{\op{ext}}(t,x_1,x_2,z)= U^{\op{ext},r}(t,r,z) e_r + U^{\op{ext},z}(t,r,z) e_z.
\end{align*}
Consider 3D Euler in vorticity form
\begin{align} \label{hc0_e1}
\begin{cases}
\partial_t \omega + \bigl( (u+U^{\op{ext}})\cdot \nabla \bigr) \omega
= (\omega \cdot \nabla)(U^{\op{ext}}+u), \\
u=-\Delta^{-1} \nabla \times \omega,\\
\omega\Bigr|_{t=0} =\omega_0.
\end{cases}
\end{align}
Let $u_0 =-\Delta^{-1} \nabla \times \omega_0$ and we assume $u_0$ is axisymmetric without swirl.
Note that the system \eqref{hc0_e1} can be written more compactly as
\begin{align*}
\partial_t \bigl( \frac{\omega} r \bigr) + (u+U^{\op{ext}})\cdot \nabla \bigl( \frac{\omega} r \bigr)=0,
\end{align*}
which highlights the axisymmetry of the system.

Define the characteristic line
\begin{align} \label{lem_hc0_e30}
\begin{cases}
\partial_t \phi(t,x) = ( u+ U^{\op{ext}})(t,\phi(t,x)),\\
\phi(0,x)=x.
\end{cases}
\end{align}
Assume on some time interval $[0,T_0]$, $T_0\le 1$, we have
\begin{itemize}
\item $\max_{0\le t \le T_0} (1+\| Du(t,\cdot)\|_{\infty}) =  A \ge 1$;
\item $\max_{0\le t\le T_0} ( \| U^{\op{ext}}(t,\cdot)\|_{H^{10}} + \| D \partial_t U^{\op{ext}}(t,\cdot) \|_{\infty}
)\lesssim 1$;
\end{itemize}

Now we specify initial data $\omega_0$ for \eqref{hc0_e1}. We shall
borrow the notation from Section \ref{sec_3DC1}, and choose initial
axisymmetric stream function in the form (see the derivation therein
after formula \eqref{ge_65a})
\begin{align*}
& \psi_0(r,z) = \sum_{\sqrt M \le j \le M} 2^{-3j} a_3(2^j(r,z)),\\
& \omega_0 (r,z) = \omega_0^{\theta} e_{\theta},\\
& \omega_0^{\theta}(r,z) = -\sum_{\sqrt M \le j \le M} b_3 (2^j(r,z)),
\end{align*}
where
\begin{align*}
b_3(r,z) = \frac 1 r ( \partial_{zz} a_3 + \partial_{rr} a_3 - \frac 1 r \partial_r a_3).
\end{align*}
Note that
\begin{align*}
u(t,r,z) = u^r(t,r,z) e_r + u^z(t,r,z) e_z.
\end{align*}
Then
\begin{lem} \label{lem_hc2}
Assume for some $\rho>0$,
\begin{align*}
U^{\op{ext}}(0,x) =0, \quad \text{for $|x|<\rho$}.
\end{align*}
Then for any $0\le t\le T_0$, we have
\begin{align*}
|(\partial_z u^z)(t,\phi(t,0) )| \gtrsim Mt -M t^2 A^2 - M (A^2t^2 +A^4 t^4)e^{2tA}.
\end{align*}
\end{lem}

\begin{proof}[Proof of Lemma \ref{lem_hc2}]
By a similar derivation as in \eqref{ge_66b}, we only need to work with the quantity
\begin{align*}
Q_1=\sum_{\sqrt M \le j\le M} \int F_3(2^j (\phi(t,2^{-j} x)-\phi(t,0)) ) \frac {b_3(x)}{|x^{\prime}|} dx,
\end{align*}
where $\phi(t,x)$ is given in \eqref{lem_hc0_e30}, and we recall that
$F_3(x_1,x_2,z)=F_3(r,z)=r^2 z/(r^2+z^2)^{5/2}$, $x^{\prime}=(x_1,x_2,0)$.

By Lemma \ref{lem_hc0} and Taylor expanding $F_3(\cdot)$ around the point $x$, we get
\begin{align}
Q_1 & \ge \sum_{\sqrt M\le j \le M}
t \int (\nabla F_3)(x) \cdot 2^j u_0(2^{-j} x) \frac{b_3(x)}{|x^{\prime}|} dx \label{lem_hc2_e3a}\\
&\; + \sum_{\sqrt M \le j\le M}
\int \nabla F_3(x) \cdot \int_0^t
(t-\tau) 2^j (\partial_{\tau} U_1)(\tau, 2^{-j} x) d\tau \frac{b_3(x)} {|x^{\prime}|} dx
\label{lem_hc2_e3b}\\
&\; + \sum_{\sqrt M \le j\le M}
\int \nabla F_3(x) \cdot \int_0^t
(t-\tau) 2^j (\partial_{\tau} U_2)(\tau, 2^{-j} x) d\tau \frac{b_3(x)} {|x^{\prime}|} dx
\label{lem_hc2_e3b_00}\\
 & \;+ \op{error}, \notag
\end{align}
where
\begin{align*}
U_1(\tau,x) &= u(\tau,x+\phi(\tau,0)) -
u (\tau, \phi(\tau,0)), \notag \\
U_2(\tau,x) & = U^{\op{ext}}(\tau, x+\phi(\tau,0)) - U^{\op{ext}}(\tau, \phi(\tau,0));
\end{align*}
and

\begin{align*}
\| \op{error}\|_{\infty} \lesssim M \cdot (A^2 t^2 +A^4 t^4) e^{2tA}.
\end{align*}

Note that \eqref{lem_hc2_e3a} can be handled in the same way as \eqref{ge_68a}. It gives the main
order $Mt$.

By Lemma \ref{lem_ha0} (note that the argument therein is independent of the dimension), we have
\begin{align*}
| \eqref{lem_hc2_e3b_00}| \lesssim M t^2
\end{align*}
which is acceptable.

For \eqref{lem_hc2_e3b}, we note that by using \eqref{hc0_e1}, $\partial_t u$ satisfies the equation
\begin{align*}
\partial_t u + \bigl((u+ U^{\op{ext}})\cdot \nabla \big) u = - \nabla p -\Delta^{-1} \nabla \times
( \partial U^{\op{ext}} \partial u),
\end{align*}
where we have used the symbolic notation $\partial U^{\op{ext}} \partial u$ to denote generic terms
of the type $\partial_j U^{\op{ext}}_k \partial_l u_m$. The actual form does not matter in the estimates.

Now note
\begin{align}
(\partial_{\tau} U_1)(\tau, x) &= (\partial_{\tau} u)(\tau, x+\phi(\tau,0) ) +
\partial_{\tau} \phi(\tau,0) \cdot  (\nabla u)(\tau, x+\phi(\tau,0)) \notag \\
&\quad  -(\partial_{\tau} u)(\tau, \phi(\tau,0) ) -
\partial_{\tau} \phi(\tau,0) \cdot  (\nabla u)(\tau, \phi(\tau,0)). \notag\\
& = \bigl( (u+U^{\op{ext}})(\tau, \phi(\tau,0)) - (u+U^{\op{ext}})(\tau, x+ \phi(\tau,0))
\bigr)\cdot (\nabla u)(\tau,x+\phi(\tau,0)) \label{lem_hc2_e20a} \\
& \quad - (\nabla p)(\tau, x+\phi(\tau,0)) \label{lem_hc2_e20aa}\\
&\quad +(\nabla p)(\tau, \phi(\tau,0))  \label{lem_hc2_e20b} \\
& \quad +\Bigl(-\Delta^{-1} \nabla \times
( \partial U^{\op{ext}} \partial u)
\Bigr)(\tau, x+\phi(\tau,0)) \label{lem_hc2_e20c}\\
& \quad -\Bigl( -\Delta^{-1} \nabla \times
( \partial U^{\op{ext}} \partial u)
\Bigr)(\tau, \phi(\tau,0)). \label{lem_hc2_e20d}
\end{align}

Now similar to the situation in the 2D Lemma \ref{lem_ha0_1}, the contribution of the terms
\eqref{lem_hc2_e20a} and \eqref{lem_hc2_e20aa} to \eqref{lem_hc2_e3b} are
\begin{align*}
\lesssim M t^2 A^2,
\end{align*}
which is acceptable.

The contribution of \eqref{lem_hc2_e20b} and \eqref{lem_hc2_e20d} to \eqref{lem_hc2_e3b} are zero.
This follows from the fact that
\begin{align}
\int \partial_j F_3(x) \frac{b_3(x)} {|x^{\prime}|} dx =0, \quad \partial_j = \partial_{x_1},
\partial_{x_2} \text{ or } \partial_z. \label{lem_hc2_e20b_fact_1}
\end{align}
To prove \eqref{lem_hc2_e20b_fact_1}, one just recall that
$F_3(r,z)=r^2 z/ (r^2+z^2)^{\frac 52}$ and $b_3(r,z)$ is odd in $z$. If $j=1,2$, then since
$\nabla_{j} F_3(r,z)= \frac {x_j} r \partial_r F_3(r,z)$, the integral then obviously vanishes (by using
oddness in $x_j$).
If $\partial_j =\partial_z$, then $\partial_z F_3$ is even in $z$ and the integrand is odd in $z$
which also vanishes.

We only need to focus on the term \eqref{lem_hc2_e20c}. It will have the same bound $Mt^2 A^2$ provided we verify the
condition
\begin{align*}
\Delta^{-1} \nabla \times (\nabla F_3(x) \frac{b_3(x)} {|x^{\prime}|} ) \in L_x^1(\mathbb R^3).
\end{align*}
Now for a vector function $g=(g_1,g_2,g_3) \in C_c^{\infty}(\mathbb R^3)$, easy
to check that $\Delta^{-1} \nabla \times g \in L_x^1(\mathbb R^3)$ if
\begin{align*}
& \int g(y) dy =0, \\
& \int y_j g_k(y) dy =0, \quad \forall\, 1\le j,k\le 3.
\end{align*}
After a tedious computation, it is not difficult to check that it is equivalent to
the condition \eqref{ge_69_8.9} which is already satisfied by the same choice of $a_3$ as in Section
\ref{sec_3DC1}. No additional work is needed.
\end{proof}

\begin{lem}\label{lem_hd1}
Let $f \in C_c^{\infty}(B(0,100))$, $g\in C_c^{\infty}(B(0,100))$ be axisymmetric functions on $\mathbb R^3$
which take the form:
\begin{align*}
f(x) = f^{\theta}(r,z) e_{\theta}, \;
g(x)= g^{\theta}(r,z) e_{\theta},\;
x=(x_1,x_2,z), \; r=\sqrt{x_1^2+x_2^2},
\end{align*}
where $f^{\theta}$ and $g^{\theta}$ are scalar-valued and vanish near $r=0$, i.e. for
some $r_0>0$,
\begin{align*}
&\op{supp}(f^{\theta}) \subset \{(r,z):\, r>r_0\},\\
&\op{supp}(g^{\theta}) \subset \{(r,z):\, r>r_0\}.
\end{align*}

Let $\omega^{a}$ and $\omega$ be smooth solutions to the following axisymmetric without swirl Euler
equations:
\begin{align} \notag
\begin{cases}
\partial_t( \frac{\omega^a}r ) + (u^a \cdot \nabla)( \frac{\omega^a} r) =0, \\
u^a = -\Delta^{-1} \nabla \times \omega^a,\\
\omega^a \Bigr|_{t=0}=f.
\end{cases}
\end{align}

\begin{align} \notag
\begin{cases}
\partial_t (\frac{\omega} r) + (u\cdot \nabla) ( \frac{\omega} r) =0,\\
u = - \Delta^{-1} \nabla \times \omega,\\
\omega\Bigr|_{t=0} =f +g.
\end{cases}
\end{align}

There exist a finite positive constant $C_f>0$ depending only on the data $f$, and an absolute constant
$C_1>0$, such that for any $2\le p \le \infty$,
\begin{align*}
\max_{0\le t\le 1} \| \omega(t) -\omega^a(t)\|_p  \le \; C_f
\cdot e^{C_1 \| \frac g r \|_{L^{3,1}(\mathbb R^3)}} ( \| g \|_p+ \| |\nabla|^{-1} g \|_2).
\end{align*}

Let
\begin{align*}
\theta(p)= \log\log (p+100).
\end{align*}

For any $\epsilon>0$, there exists $\delta=\delta(\epsilon,f)>0$ sufficiently small such that
if
\begin{align}
\sup_{2\le p<\infty}
\Bigl( \frac 1 {\theta(p)}
\cdot e^{C_1 \| \frac g r \|_{L^{3,1}(\mathbb R^3)}} ( \| g \|_p+ \| |\nabla|^{-1} g \|_2) \Bigr) <
\delta (\epsilon,f), \label{lem_hd1_e1a}
\end{align}
then
\begin{align}
\sup_{0\le t\le 1}\;
\sup_{2\le p<\infty} \frac{\| \omega(t) -\omega^a(t) \|_p} {\theta(p)} <\epsilon. \label{lem_hd1_e2}
\end{align}

\end{lem}

\begin{rem}
Our choice of $\theta(p) \sim \log\log p$ is certainly an
``overkill" here. We chose such $\theta(p)$ just to be on the safe
side. Alternatively one can choose $\theta(p)=\log p$ which
corresponds to the requirement in Yudovich's uniqueness
theorem\footnote{As was already pointed out by Vishik in
\cite{Vishik99} (see the footnote on P770 of \cite{Vishik99}
 therein), ``Not much stronger than linear" on p.28 of
\cite{Yu95} is a misprint and $\log p$ growth of $\|\omega\|_p$ is what is actually covered by \cite{Yu95}.}
\cite{Yu95}. In our application later, we shall take the function $g(x) = g^{\theta}(r,z) e_{\theta}$
approximately of the form
(neglecting some additional prefactors)
\begin{align*}
g^{\theta}(x) = \sum_{ M-N \le j \le M} \phi(2^j x),
\end{align*}
where $\phi(x)=\phi(x_1,x_2,z)=\phi(r,z)$
is a smooth axisymmetric function supported on $r\sim 1$, $|z| \sim 1$. Easy to check that
\begin{align*}
& \| \frac g r \|_{L^{3,1}(\mathbb R^3)} \lesssim N, \notag \\
& \| g \|_{p} \lesssim N \cdot 2^{-\frac 3p (M-N)}, \notag \\
& \| |\nabla|^{-1} g \|_2 \lesssim 2^{-\frac 52 (M-N)}.
\end{align*}
Then for $1\ll N\ll M$, we have (below $C_2>0$ is an absolute constant)
\begin{align*}
& \frac 1 {\theta(p)} e^{C_1 \|\frac g r \|_{L^{3,1}}} (\|g \|_p +\| |\nabla|^{-1} g\|_2) \notag \\
\le & \frac 1 {\log\log(p+100)} \cdot e^{C_2\cdot N - \frac 1 p M}.
\end{align*}
If $p< \frac M {2 C_2 N}$, then the above quantity is less than $e^{-C_2 N}$ which can be made
arbitrarily small. If $p>\frac M {2 C_2 N}$, then the above quantity is less than
\begin{align*}
\frac 1 {\log\log ( \frac M{2C_2 N} +100)} \cdot e^{C_2 \cdot N},
\end{align*}
which can also be made arbitrarily small by choosing $M$ sufficiently large.
\end{rem}

\begin{proof}[Proof of Lemma \ref{lem_hd1}]
In this proof we shall denote by $C_f$ any finite positive constant which depends only on $f$. The value
of $C_f$ can change from line to line. For example by standard wellposedness theory for axisymmetric
without swirl flows, we have\footnote{Here the axisymmetric without swirl assumption is used to guarantee that
the solution is global in time. In particular it exists on $[0,1]$.}
\begin{align*}
\max_{0\le t\le 1} \| u^a(t,\cdot)\|_{H^{10}(\mathbb R^3)} \le C_f.
\end{align*}
For any two quantities $X$ and $Y$, we shall use the usual notation $X \lesssim Y$ if $X\le C Y$ where $C>0$
is some harmless absolute constant.

We proceed in two steps.

\texttt{Step 1}. Estimate of $\|u^a-u\|_2$. Since $u^a$ and $u$ satisfies the equations:
\begin{align*}
&\partial_t u + (u\cdot \nabla) u = - \nabla p, \\
& \partial_t u^a + (u^a \cdot \nabla) u^a = - \nabla p^a,
\end{align*}
we get
\begin{align*}
\partial_t ( u-u^a) +  (u\cdot \nabla) (u-u^a) +\bigl((u-u^a) \cdot \nabla \bigr) u^a = - \nabla (p-p^a).
\end{align*}
Then clearly
\begin{align*}
\max_{0\le t\le 1} ( \| (u-u^a)(t,\cdot) \|_2) &\le \| (u-u^a)(0,\cdot) \|_2 \cdot e^{ \max_{0\le t\le 1}\|
\nabla u^a(t,\cdot) \|_{\infty} } \notag \\
&\lesssim \| |\nabla|^{-1} g \|_2 \cdot C_f.
\end{align*}

\texttt{Step 2}. Estimate of $\| \omega - \omega^a \|_p$.  Set $\eta= \omega -\omega^a$. Then since
\begin{align*}
&\partial_t \omega + (u\cdot \nabla) \omega = \frac{u^r} r \omega,\\
&\partial_t \omega^a + (u^a \cdot \nabla) \omega^a = \frac{(u^a)^r} r \omega^a,
\end{align*}
we get
\begin{align*}
&\partial_t \eta + (u\cdot \nabla ) \eta + \bigl( (u-u^a) \cdot \nabla \bigr) \omega^a \notag \\
& \qquad \quad = \frac {u^r} r \eta + (u-u^a)^r \frac{\omega^a}r.
\end{align*}
Therefore for any $1<p<\infty$, we have
\begin{align}
\frac 1 p \partial_t ( \| \eta \|_p^p ) &\le \| u-u^a\|_p \| \eta\|_p^{p-1} \| \nabla \omega^a \|_{\infty}
+\| \frac{u^r } r \|_{\infty} \| \eta\|_p^p \notag \\
& \qquad  + \|u-u^a \|_p \| \eta\|_p^{p-1} \| \frac{\omega^a} r \|_{\infty} \notag \\
& \le\; C_f \|u-u^a \|_p \|\eta\|_p^{p-1} + \| \frac{u^r} r\|_{\infty} \| \eta\|_p^p. \label{lem_hd1_e5}
\end{align}

Now by conservation of $L^{3,1}$-norm of $\omega/r$, we have
\begin{align*}
\| \frac {u^r(t)} r \|_{\infty} &\lesssim \| \frac{\omega(t)} r \|_{L^{3,1}} \notag \\
& \lesssim \| \frac{\omega(0)} r \|_{L^{3,1}} \notag \\
& \lesssim C_f + \| \frac g r \|_{L^{3,1}}.
\end{align*}

On the other hand, for $p\ge 2$, we have
\begin{align*}
\| u-u^a \|_p &\lesssim \| u-u^a \|_2 + \| \eta \|_p \notag\\
& \lesssim C_f \| |\nabla|^{-1} g \|_2 + \| \eta\|_p.
\end{align*}

Plugging the above estimates into \eqref{lem_hd1_e5}, we get for any $2\le p<\infty$,
\begin{align*}
\partial_t \| \eta\|_p \lesssim C_f \| |\nabla|^{-1} g \|_2 + (C_f + \| \frac g r \|_{L^{3,1}}) \| \eta\|_p.
\end{align*}
Integrating in time then gives (for $2\le p<\infty$)
\begin{align*}
\max_{0\le t\le 1} \| \eta(t,\cdot)\|_p \le C_f e^{C_1 \| \frac g r \|_{L^{3,1}}} (\|g\|_p + \||\nabla|^{-1} g\|_2),
\end{align*}
where $C_1>0$ is an absolute constant. Note that by taking $p\to \infty$, the above inequality also
holds for $p=\infty$.

Finally the estimate \eqref{lem_hd1_e2} is a simple consequence of the above inequality.
\end{proof}

We are now ready to state a proposition which gives the existence and uniqueness of solutions to the
3D axisymmetric without swirl Euler equation for a special class of initial (vorticity) data. Roughly speaking
the constructed solution $\omega=\omega(t)$ have the property that $\|\omega \|_p \lesssim
\log\log p$ for $p$ large.

\begin{prop} \label{prop_hd10}
 Suppose $\{g_i\}_{i=1}^{\infty}$ is a sequence of axisymmetric
 functions on $\mathbb R^3$ satisfying the following conditions:
 \begin{itemize}
  \item For each $i\ge 1$, $g_i \in C_c^{\infty}(B(0,100))$ and has the form $g_i(x)=g_i^{\theta}(r,z) e_{\theta}$, where
  $g_i^{\theta}$ is scalar-valued and vanishes near $r=0$:
  \begin{align*}
   \operatorname{supp}(g_i^{\theta}) \subset \{(r,z): \, r>r_i\},\quad \text{for some $r_i>0$}.
  \end{align*}
 \item For each $i\ge 2$, denote $f_i =\sum_{j=1}^{i-1} g_j$, then (recall $\theta(p)=\log\log(p+100)$)
 \begin{align*}
 \sup_{2\le p<\infty}
 \Bigl( \frac 1 {\theta(p)} e^{C_1 \| \frac {g_i} r \|_{L^{3,1}}} (\|g_i\|_p
 + \||\nabla|^{-1} g_i \|_2 ) \Bigr) <\delta_{i},
 \end{align*}
where $\delta_i=\delta(2^{-i}, f_i)$ as defined in \eqref{lem_hd1_e1a}.

 \end{itemize}
Let
\begin{align*}
 g= \sum_{i=1}^{\infty} g_i
\end{align*}
and consider the system
\begin{align} \label{prop_hd10_e1}
 \begin{cases}
  \partial_t \omega + (u\cdot \nabla)\omega=(\omega\cdot \nabla)u, \quad 0<t\le 1; \\
  u=-\Delta^{-1} \nabla \times \omega, \\
  \omega \Bigr|_{t=0}=g.
 \end{cases}
\end{align}

Then there exists a unique solution $\omega$ to \eqref{prop_hd10_e1} with the following properties:
\begin{enumerate}
 \item $\omega$ is compactly supported:
 \begin{align*}
  \operatorname{supp}(\omega(t,\cdot)) \subset B(0,R_0), \quad \forall\, 0<t\le 1.
 \end{align*}
Here $R_0>0$ is an absolute constant.

\item $\omega$ obeys the bound:
\begin{align*}
\sup_{0\le t\le 1} \; \sup_{2\le p<\infty} \frac{\| \omega(t) \|_p} {\theta(p)} <\infty.
\end{align*}

\item For any $2\le p<\infty$, $\omega \in C_t^0 L_x^p([0,1] \times {B(0,R_0)})$; Also
$u \in C_t^0 L_x^2 \cap L_t^{\infty} L_x^{\infty} ([0,1]\times \mathbb R^3)$.
 In fact $u\in C_t^0 C_x^{\alpha}([0,1]\times \mathbb R^3)$
for any $0<\alpha<1$. Also $\omega \in C([0,1],X_{\theta} )$, where $X_{\theta}$ is the Banach space endowed
with the norm
\begin{align} \label{prop_hd1_e3a}
\| \omega \|_{X_{\theta}} := \sup_{2\le p<\infty} \frac{\| \omega \|_p} {\theta(p)}.
\end{align}

\end{enumerate}

\end{prop}

\begin{proof}[Proof of Proposition \ref{prop_hd10}]
For each $m\ge 2$, let $\omega^m$ be the solution to the system
\begin{align*}
\begin{cases}
\partial_t ( \frac{\omega^m} r ) + (u^m \cdot \nabla) ( \frac{\omega^m} r ) =0, \quad 0<t\le 1,\\
u^m = -\Delta^{-1} \nabla \times \omega^m,\\
\omega^m \Bigr|_{t=0} = \sum_{i=1}^m g_i.
\end{cases}
\end{align*}

By Lemma \ref{lem_hd1} and our assumptions on $g_i$, we get
\begin{align*}
\sup_{0\le t\le 1} \; \sup_{2\le p <\infty}
\frac{\| \omega^{m+1}(t) -\omega^m(t) \|_p} {\theta (p)} < 2^{-m}.
\end{align*}
Clearly we can then extract a limiting solution $\omega$ in the Banach space $C([0,1], X_{\theta})$ (see \eqref{prop_hd1_e3a}).

By using energy conservation, we have $\| u^m(t) \|_2 = \|u^m(0)\|_2 \lesssim 1$.
Since
\begin{align*}
\sup_{m\ge 2} \sup_{0\le t\le 1} \|\omega^m(t) \|_{L^4(\mathbb R^3)} \lesssim 1,
\end{align*}
it follows
that
\begin{align*}
\sup_{m\ge 2} \max_{0\le t\le 1} \| u^m(t) \|_{\infty} &\lesssim  \sup_{m\ge 2} \max_{0\le t\le 1} (\|\omega^m(t)\|_4
+ \|u^m(t)\|_2) \notag \\
& \lesssim 1.
\end{align*}
Then for some constant $R_0>0$, we have $\op{supp}(\omega^m(t) ) \subset B(0,R_0/2)$ for all $m$ and $0\le t\le 1$.
This implies that the limiting
solution $\omega$ is compactly supported in $B(0,R_0)$. The other regularity properties of $\omega$ (and $u$) can
be easily checked. We omit the details.
\end{proof}

The next proposition is the key to our patching of solutions for the
3D $C^1$ case. The overall statement of the proposition is a bit
long due to some technical complications pertaining to the 3D
situation. In short summary the main body of the proposition should
read as ``\,Let $u_{-1} \in C_c^{\infty}(B(0,100))$ ... Then for any
$0<\epsilon<\epsilon_0$, we can find a smooth ... with the
properties ... and $\delta_0$...,   such that for any $\omega_j$
with the properties ...,  the following hold true: ....".

\begin{prop} \label{prop_hd20}
Let $u_{-1} \in C_c^{\infty}(B(0,100) )$ be a given axisymmetric without swirl velocity field on $\mathbb R^3$ such that
\begin{align*}
u_{-1} = u_{-1}^r e_r + u_{-1}^z e_z.
\end{align*}
Denote the corresponding vorticity
 $\omega_{-1}=\omega_{-1}^{\theta} e_{\theta}$,  where $\omega_{-1}^{\theta}=\partial_z u_{-1}^r -\partial_r u_{-1}^z$.
Assume for some $r_{-1}>0$, $0<R_0<\frac 1 {100}$,
\begin{align*}
\operatorname{supp}(\omega_{-1}^{\theta}) \subset\{(r,z):\; r>r_{-1},\, z\le -4R_0\}.
\end{align*}
Denote
\begin{align*}
u_{-1}^*=\|D u_{-1} \|_{\infty}.
\end{align*}

 Then for any $0<\epsilon<\epsilon_0$ with $\epsilon_0=\epsilon_0(u_{-1})\ll R_0$  sufficiently small,
 we can find a smooth axisymmetric without swirl velocity field $u_0=u_0^r e_r +u_0^z e_z$
 (depending only on $(\epsilon,u_{-1})$)
 with the properties:
\begin{itemize}
 \item  $u_0\in C_c^{\infty}(B(0,100))$ and for some $r_0>0$,
\begin{align} \label{p_hd20_e3}
 \operatorname{supp}(u_0) \subset\{(r,z):\, r_0<r<\epsilon, \,  -\epsilon <z<\epsilon \}.
\end{align}
Also
\begin{align} \label{p_hd20_e5}
\| D u_0\|_{\infty} <\epsilon u_{-1}^*<\frac 14 u_{-1}^*;
\end{align}
\item  denote $\omega_0 = \nabla \times u_0$, then (see \eqref{lem_hd1_e1a})
\begin{align} \label{p_hd20_e5a}
\op{sup}_{2\le p<\infty} \Bigl( \frac 1 {\theta(p)}
e^{C_1 \|\frac {\omega_0} r \|_{L^{3,1}}} (\|\omega_0\|_p
+\| |\nabla|^{-1} \omega_0\|_2 ) \Bigr) <\delta(\epsilon^{10}, \omega_{-1});
\end{align}

\end{itemize}

\noindent
and $\delta_0=\delta_0(\epsilon,\omega_{-1},\omega_0)\ll \epsilon $ sufficiently small
 such that for any smooth axisymmetric functions $\omega_j=\omega_j^{\theta}e_{\theta}$,
 $1\le j\le N$ (here $N\ge 1$ is arbitrary but finite)
  with the properties:
\begin{itemize}
\item $\omega_j = \nabla \times u_j$, where $u_j$ is axisymmetric without swirl,
$u_j\in C_c^{\infty}(B(0,100))$ and
\begin{align*}
\|D u_j \|_{\infty} < \frac {\epsilon} {2^{j+1}} u_{-1}^*;
\end{align*}
\item $\operatorname{supp}(u_j) \subset\{(r,z):\, r>r_j, z>2R_0\}$ for some $r_j>0$;
\item for each $j\ge 1$, denote $f_j=\Bigl(
\omega_{-1}^{\theta}+ \omega_0^{\theta}+\sum_{i=1}^{j-1} \omega_i^{\theta} \Bigr) e_{\theta}$,
then
\begin{align*}
\op{sup}_{2\le p<\infty} \Bigl( \frac 1 {\theta(p)}
e^{C_1 \|\frac {\omega_j} r \|_{L^{3,1}}} (\|\omega_j\|_p
+\| |\nabla|^{-1} \omega_j\|_2 ) \Bigr)
<\delta_j,
\end{align*}
where $\delta_j=\delta(2^{-3j}\delta_0, f_j)$ as defined in \eqref{lem_hd1_e1a};
\end{itemize}

 the following hold true:

 Let  $\omega$ be the smooth solution to the axisymmetric system
\begin{align} \notag
 \begin{cases}
  \partial_t \left( \frac {\omega }r \right) + (u \cdot \nabla ) \left( \frac {\omega} r  \right) =0,
  \quad 0<t\le 1, \\
  u=-\Delta^{-1} \nabla \times \omega, \\
  \omega \Bigr|_{t=0} =(\omega_{-1}^{\theta} +\omega_0^{\theta} +\sum_{j=1}^N \omega_j^{\theta}) e_{\theta},
 \end{cases}
 \end{align}
then
\begin{enumerate}

\item $\omega$ is compactly supported: for some absolute constant $R_1>0$,
 \begin{align} \label{p_hd20_t0a}
  \operatorname{supp}(\omega(t,\cdot)) \subset B(0,R_1), \quad \forall\, 0<t\le 1.
 \end{align}

\item $\omega$ obeys the uniform bound: for some constant $C_2>0$ ($C_2$ is independent of
$N$),
\begin{align} \label{p_hd20_t0b}
\sup_{0\le t\le 1} \; \sup_{2\le p<\infty} \frac{\| \omega(t) \|_p} {\theta(p)} <C_2.
\end{align}

\item for any $0\le t \le \epsilon$, we have the decomposition
\begin{align} \label{p_hd20_t1}
\omega(t)=\omega^A(t)+\omega^B(t)+\omega^C(t),
\end{align}
where
\begin{align*}
&\operatorname{supp}(\omega^A(t)) \subset\{(r,z):\, z\le -4R_0+\sqrt{\epsilon} \}; \\
&\operatorname{supp}(\omega^B(t)) \subset\{x:\; |x| \le \frac 12 \sqrt{\epsilon}  \};\\
&\operatorname{supp}(\omega^C(t)) \subset\{(r,z):\, z\ge 2R_0-\sqrt{\epsilon} \};
\end{align*}
and $\omega^A(t=0)=\omega_{-1}$, $\omega^{B}(t=0)=\omega_0^{\theta} e_{\theta}$,
$\omega^C(t=0)=(\sum_{j=1}^N \omega_j^{\theta}) e_{\theta}$.

\item the $C^1$ norm of initial data $u(t=0)=-\Delta^{-1} \nabla \times \omega(t=0)$ has the bound:
\begin{align} \label{p_hd20_t2}
\| D u(t=0)\|_{\infty} \le 2u_{-1}^*.
\end{align}

\item the $C^1$ norm of $u$ is inflated rapidly on the time interval $[0,\epsilon]$ and in the region
$|x| \le \sqrt {\epsilon}$:
there exists $0<t_0^1=t_0^1(\epsilon,u_{-1},u_0)<\epsilon$, $0<t_0^2=t_0^2(\epsilon,u_{-1},
u_0)<\epsilon$, such that
\begin{align}
 &\| D \Delta^{-1} \nabla \times
 \omega^B(t)\|_{L_x^{\infty}(|x|\le \sqrt{\epsilon})}  >\frac 1{\epsilon}, \quad \text{for any $t_0^1\le t \le t_0^2$;}
 \label{p_hd20_t3}
\end{align}
on the other hand, for any $0\le t \le t_0^2$,
\begin{align} \label{p_hd20_t3a}
&\| D \Delta^{-1} \nabla \times \omega^C(t) \|_{L_x^{\infty}(|x|\le R_0/100)}<\sqrt{\epsilon}. \notag \\
&\| D \Delta^{-1} \nabla \times \omega^A(t) \|_{L_x^{\infty}(|x|\le R_0/100)} \lesssim_{R_0,u_{-1}} 1.
\end{align}

and
\begin{align} \label{p_hd20_t3b}
\| D \Delta^{-1} \nabla \times \omega^B(t)\|_{L_x^{\infty}(|x|>\sqrt{\epsilon})} < \epsilon.
\end{align}

\item all $H^k$, $k\ge 2$ norms of $\omega^B$ can be bounded purely in terms of initial
data $\omega_0$ on the time interval $[0,\epsilon]$:
for any $k\ge 2$,
\begin{align} \label{p_hd20_t4}
\max_{0\le t \le \epsilon} \| \omega^B(t) \|_{H^k} \le C(k,R_0,u^*_{-1}) \|\omega_0\|_{H^k}.
\end{align}
Note here the bound of $\|\omega^B\|_{H^k}$ is ``almost local" in
the sense that it depends only on $u^*_{-1}$ but not on other higher
Sobolev norms of $\omega^A$ or $\omega^C$. Similarly we have
\begin{align} \label{p_hd20_t5}
\max_{0\le t \le \epsilon} \| \omega^A(t) \|_{H^k} \le C(k,R_0,u^*_{-1}) \| \omega_{-1} \|_{H^k}, \quad \forall\, k\ge 2.
\end{align}
\end{enumerate}
\end{prop}

\begin{proof}[Proof of Proposition \ref{prop_hd20}]
We shall sketch the details. The main point is to specify the choice of $u_0$ to achieve \eqref{p_hd20_t3}.
Following the notation as in
\eqref{ge_66aa}, we take the axisymmetric stream function in the form
\begin{align*}
\psi_0(r,z)= \frac 1 {\log N_1} \sum_{M_1-N_1 \le j \le M_1} 2^{-3j} a_3(2^j(r,z)),
\end{align*}
where $\log\log\log\log\log\log M_1=N_1$ and we shall take $N_1$ sufficiently large. We then have
\begin{align*}
u_0 = u_0^r e_r +u_0^z e_z,
\end{align*}
where
\begin{align*}
&u_0^r(r,z) = \frac 1 {\log N_1} \sum_{M_1-N_1\le j\le M_1} 2^{-j} a_3^r(2^j(r,z)),\quad
a_3^r(r,z)=\frac 1 r (-\partial_z a_3)(r,z);\\
&u_0^z(r,z) = \frac 1 {\log N_1} \sum_{M_1-N_1\le j\le M_1} 2^{-j} a_3^z(2^j(r,z)),\quad
a_3^z(r,z)=\frac 1 r (\partial_r a_3)(r,z).
\end{align*}
Also $\omega_0= \omega_0^{\theta} e_{\theta}$ with
\begin{align*}
\omega_0^{\theta}(r,z)=- \frac 1 {\log N_1} \sum_{M_1-N_1\le j\le M_1}
b_3(2^j(r,z)).
\end{align*}
Clearly
\begin{align*}
\| Du_0 \|_{\infty} \lesssim \frac 1 {\log N_1}
\end{align*}
and \eqref{p_hd20_e3}--\eqref{p_hd20_e5} can be easily satisfied (by taking $N_1$ large).

For \eqref{p_hd20_e5a}, we note that
\begin{align*}
& \| \frac{\omega_0} r \|_{L^{3,1}} \lesssim \frac 1 {\log N_1} N_1,\\
& \| \omega_0 \|_p \lesssim \frac 1 {\log N_1} 2^{-(M_1-N_1) \frac 3 p} \le 2^{-\frac 1 p M_1},\\
& \| |\nabla|^{-1} \omega_0\|_2 \le 2^{-M_1},
\end{align*}
and
\begin{align}
 &\op{sup}_{2\le p<\infty} \Bigl( \frac 1 {\theta(p)}
e^{C_1 \|\frac {\omega_0} r \|_{L^{3,1}}} (\|\omega_0\|_p
+\| |\nabla|^{-1} \omega_0\|_2 ) \Bigr)  \notag \\
\le & \; \sup_{2\le p<\infty} \Bigl( \frac 1 {\log\log(p+100)} e^{N_1-\frac 1 p M_1} \Bigr) \notag\\
\le & \; \max\{ e^{-N_1}, \; \frac 1 {\log\log(M_1/N_1)} \}. \notag
\end{align}
Obviously by taking $N_1$ large, \eqref{p_hd20_e5a} is satisfied.

The properties \eqref{p_hd20_t0a}--\eqref{p_hd20_t0b} follows from our assumptions on $\omega_j$ and
Proposition \ref{prop_hd10}.

Define $\omega^A$, $\omega^B$, and $\omega^C$ as solutions to the \emph{linear} systems:
\begin{align*}
\begin{cases}
\partial_t ( \frac{\omega^A} r ) + (u\cdot \nabla) (\frac{\omega^A} r)=0, \\
\omega^A \Bigr|_{t=0} =\omega_{-1};
\end{cases}\\
\begin{cases}
\partial_t ( \frac{\omega^B} r ) + (u\cdot \nabla) (\frac{\omega^B} r)=0, \\
\omega^B \Bigr|_{t=0} =\omega_{0};
\end{cases}\\
\begin{cases}
\partial_t ( \frac{\omega^C} r ) + (u\cdot \nabla) (\frac{\omega^C} r)=0, \\
\omega^C \Bigr|_{t=0} =(\sum_{j=1}^N \omega_j^{\theta}) e_{\theta}.
\end{cases}
\end{align*}
The decomposition \eqref{p_hd20_t1} then follows from the definition. Note that the separation of the
supports of $\omega^A$, $\omega^B$ and $\omega^C$ easily follow from the finite speed propagation.
The bound \eqref{p_hd20_t2} is trivial. The estimates \eqref{p_hd20_t4}--\eqref{p_hd20_t5} follow from
local energy estimates using the separation of support. The bound \eqref{p_hd20_t3a} comes from the
fact that $\omega_A$ and $\omega_C$ are supported away from $|x|\le R_0/100$.

It remains for us to check \eqref{p_hd20_t3b} and \eqref{p_hd20_t3}. Note that on the time interval $[0,\epsilon]$,
$\omega^B$ satisfies the equation
\begin{align*}
\begin{cases}
\partial_t ( \frac{\omega^B} r) + (( u^B+U^{\op{ext}}) \cdot  \nabla )( \frac{\omega^B} r) =0,\\
u^B= -\Delta^{-1} \nabla \times \omega^B,\\
\omega^B \Bigr|_{t=0}=\omega_0,
\end{cases}
\end{align*}
where $U^{\op{ext}}(t,x) = \tilde U(t,x) \,\chi_{|x|<R_0/100}$ (here $\chi_{|x|<R_0/100}$ is a smooth
cut-off function localized to $|x|<R_0/100$) and
\begin{align*}
\tilde U = u^A +u^C=-\Delta^{-1} \nabla \times \omega^A - \Delta^{-1} \nabla \times \omega^C.
\end{align*}

By using the fact that $\op{supp}(\omega^B) \subset \{ |x| \le \frac 1 2 \sqrt{\epsilon} \}$
and the decay of the Riesz kernel, we get
\begin{align*}
 & \| D \Delta^{-1} \nabla \times \omega^B(t) \|_{L_x^{\infty}(|x|> \sqrt {\epsilon})} \notag \\
 = & \| \mathcal R_{ij} ( \chi_{|y|<\frac 12 \sqrt{\epsilon}} \omega^B(t,y) ) \|_{L_x^{\infty}
 (|x|>\sqrt{\epsilon})} \notag \\
 \lesssim & \epsilon^{-C} \| \omega^B(t) \|_{1} \notag \\
 \lesssim & \epsilon^{-C} \| \frac{\omega_0} r \|_{1} <\epsilon,
 \end{align*}
 where to achieve the last inequality above, we need to take the parameter $N_1$ in the definition
 of $\omega_0$ sufficiently large. Therefore \eqref{p_hd20_t3b} is proved. Note that by \eqref{p_hd20_t3b},
 to prove \eqref{p_hd20_t3}, we only need to show
 \begin{align*}
 \| D \Delta^{-1} \nabla \times \omega^B(t) \|_{\infty} > \frac 2 {\epsilon}.
 \end{align*}
We are now in a position to apply Lemma \ref{lem_hc2} (with simple changes in numerology). For this we
need to check the condition
\begin{align} \label{p_hd20_cond10}
\| U^{\op{ext}} \|_{H^{10}} + \| D \partial_t U^{\op{ext}} \|_{\infty} \lesssim_{R_0,u_{-1}} 1.
\end{align}
We just need to check the contribution of $\omega^C$ to $\| D \partial_t U^{\op{ext}} \|_{\infty}$.
The other estimates (and the contribution of $\omega^A$) are simpler and therefore omitted. By
definition we just need to show
\begin{align*}
\| D \Delta^{-1} \nabla \times \partial_t \omega^C \|_{L_x^{\infty}(|x|<R_0/100)} \lesssim_{R_0,u_{-1}} 1.
\end{align*}
Since $\op{supp}(\omega^C) \subset \{R_0<|x|<100\}$, and
\begin{align*}
\partial_t \omega^C = - (u\cdot \nabla) \omega^C + u^r \frac{\omega^C}r,
\end{align*}
obviously we have (note that $(u\cdot \nabla)\omega^C = \nabla \cdot ( u \otimes \omega^C)$),
\begin{align*}
 & \| D \Delta^{-1} \nabla \times \partial_t \omega^C \|_{L_x^{\infty}(|x|<R_0/100)}  \notag \\
 \lesssim_{R_0} &\; \| u\|_{\infty} (\|\omega^C \|_{1} +\| \frac{\omega^C} r \|_{1}) \notag \\
 \lesssim_{R_0} & \; \|u\|_{\infty} \| \frac{\omega^C(t=0)} r \|_{1} \lesssim_{R_0,u_{-1} } 1. \\
 \end{align*}
Therefore \eqref{p_hd20_cond10} is proven and \eqref{p_hd20_t3} follows from Lemma \ref{lem_hc2} by
taking the parameter $N_1$ (in the definition of $\omega_0$) sufficiently large.

\end{proof}

\section{proof of Theorem \ref{thm1} for $d=3$, proof of Theorem \ref{thm3D_1} and Theorem \ref{thm3D_k}}
We begin with a simple lemma which is the 3D analogue of Lemma \ref{s7_lem1}. It is effectively
a re-statement of Lemma 7.13 from \cite{BL13} with minor expositional changes.

\begin{lem} \label{s12_lem1}
 Let $\omega^1$, $\omega^2$ be given smooth solutions to the 3D Euler equations in vorticity form:
 \begin{align} \notag
  \begin{cases}
   \partial_t \omega^j+ (u^j \cdot \nabla )\omega^j = (\omega^j \cdot \nabla) u^j, \qquad 0<t\le 1, \\
    u^j =-\Delta^{-1} \nabla \times \omega^j, \\
   \omega^j \Bigr|_{t=0}=f^j \in C_c^{\infty}(\mathbb R^3), \quad j=1,2.
  \end{cases}
 \end{align}
Here the lifespan of each solution $\omega^j$ is assumed to be at least $[0,1]$.

Define
\begin{align} \label{s12_lem1_1}
 u_*= \max_{j=1,2} \max_{0\le t\le 1} \| u^j(t,\cdot)\|_{\infty}.
\end{align}

Consider the problem
\begin{align} \label{s12_lem1_2}
 \begin{cases}
  \partial_t W + (U\cdot\nabla )W = (W\cdot \nabla)U, \\
  U= -\Delta^{-1} \nabla \times W,\\
  W\Bigr|_{t=0} =W_0,
 \end{cases}
\end{align}
where
\begin{align*}
 W_0(x) =f^1(x) + f^2(x-x_W),
\end{align*}
and $x_W\in \mathbb R^3$ is a vector parameter which controls the mutual distances between $f^1$ and (translated) $f^2$.

For any $\epsilon>0$ and any integer $k_0\ge 4$, there exists
$$R_{\epsilon}
=R_{\epsilon}(\epsilon, k_0,\max_{j=1,2}\max_{0\le t \le 1} \|u^j
(t)\|_{H^{k_0+4}})> 100 u_*$$
sufficiently large, such that if
$|x_{W}|\ge R_{\epsilon}$, then the following hold:

\begin{enumerate}
 \item There exists a unique smooth solution $W$ to \eqref{s12_lem1_2} on the time interval $[0,1]$. Furthermore for any $0\le t \le1$ it
 has a smooth decomposition
 \begin{align}
  W(t)=W^{1}(t,x) + W^2 (t,x-x_W), \label{s12_lem1_3}
 \end{align}
where
\begin{align*}
 &\operatorname{supp}(W^{i} (t,\cdot)) \subset B(\spp (f^i), r_0+\epsilon), \quad i=1,2.
\end{align*}

\item The flow $W^i$ is uniformly close to $\omega^i$:
\begin{align} \label{s12_lem1_4}
 & \max_{0\le t\le 1}\| W^i(t,\cdot) -\omega^i(t,\cdot)\|_{H^{k_0}(\mathbb R^3)} <\epsilon, \quad i=1,2.
\end{align}

\item All higher Sobolev norms of $W^i$ can be controlled in terms of $f^i$:  Let
\begin{align*}
 M_*=\max_{i=1,2}\Bigl( \max_{0\le t\le 1} \| \omega^i(t,\cdot) \|_{\infty} +
 \|\Delta^{-1} \nabla \times f^i\|_2\Bigr).
\end{align*}
Then for any $k\ge 4$, $i=1,2$,
\begin{align} \label{s12_lem1_5}
 &\max_{0\le t \le 1 } \| W^i(t,\cdot) \|_{H^k} \le C(k, \| f^i\|_{H^k}, M_*)<\infty.
\end{align}

\end{enumerate}

\end{lem}

\begin{proof}[Proof of Lemma \ref{s12_lem1}]
The proof is an adaptation of the argument for Lemma \ref{s7_lem1}. We omit it here.
For more details one can refer to the proof of Lemma 7.13 in \cite{BL13}.
\end{proof}

\begin{prop} \label{s12_prop1}
Assume $\{\omega^j\}_{j=1}^{\infty}$ is a sequence of smooth functions each of which solves the 3D incompressible Euler equation
(in vorticity form)
\begin{align*}
 \begin{cases}
  \partial_t \omega^j + (u^j \cdot \nabla ) \omega^j = (\omega^j \cdot \nabla )u^j, \quad 0<t\le 1,\\
  u^j=-\Delta^{-1}\nabla \times \omega^j, \\
  \omega^j \Bigr|_{t=0} = f^j \in C_c^{\infty}(\mathbb R^3).
 \end{cases}
\end{align*}
Here we assume the lifespan of each solution $\omega^j$ is at least $[0,1]$.
For each $j\ge 1$,  assume that $\operatorname{supp}(\omega^j(t)) \subset B(0,2^{-10j})$
for any $0\le t\le 1$ and
\begin{align} \label{s12_prop1_1}
 \max_{0\le t\le 1} (\| \omega^j(t) \|_{\infty}
+\|u^j(t)\|_{\infty})\le 2^{-10j}.
\end{align}

Let $k_0\ge 4$ be a fixed integer. Then there exist centers $x_j \in \mathbb R^3$ whose mutual distances
are sufficiently large (i.e. $|x_j-x_k|\gg 1$ if $j\ne k$) such that
the following hold:

\begin{enumerate}
 \item Take the initial data (vorticity)
 \begin{align*}
  W_0(x) = \sum_{j=1}^{\infty} f^j(x-x_j),
 \end{align*}
then $W_0 \in L^1 \cap L^{\infty} \cap
C^{\infty}$. The corresponding initial velocity $U_0 \in L^2 \cap L^{\infty}\cap C^{\infty}$. Furthermore for
any $j\ne l$
\begin{align} \notag 
 B(x_j, 100) \cap B(x_l, 100 ) = \varnothing.
\end{align}

\item With $W_0$ as initial data, there exists a unique smooth solution $W$ to the Euler
equation (in vorticity form)
\begin{align*}
\begin{cases}
 \partial_t W + (U\cdot \nabla)W=(W\cdot \nabla)U,\\
 U=-\Delta^{-1} \nabla \times W,\\
 W\Bigr|_{t=0} =W_0.
 \end{cases}
\end{align*}
on the time interval $[0,1]$ satisfying $W \in L_t^{\infty}L_x^1\cap L_t^{\infty}L_x^{\infty}$,
$U \in C_t^0 L_x^2$.  Moreover for each $0\le t\le 1$,
  $W(t,\cdot) \in C^\infty_x$ and $U(t,\cdot)\in C^{\infty}_x$.

\item For any $0\le t \le1$,
\begin{align}
 \operatorname{supp} (W(t,\cdot) ) \subset \bigcup_{j=1}^{\infty} B(x_j, 1). \label{s12_prop1_3}
\end{align}
$W(t)$ can be decomposed accordingly as
\begin{align*}
W(t,x) =\sum_{j=1}^{\infty} W^j(t,x-x_j),
\end{align*}
where $W^j \in C_c^{\infty} (B(0,1))$. Furthermore for any $k\ge 4$,
\begin{align*}
\max_{0\le t\le 1} \| W^j(t,\cdot) \|_{H^k} \le  C_2 \cdot \| f^j \|_{H^k},
\end{align*}
where $C_2>0$ is a constant depending only on $k$.

\item For any $j\ge 1$,
\begin{align}
 \max_{0\le t \le 1} \| W^j(t,\cdot) - \omega^j(t,\cdot)\|_{H^{k_0+100}} <2^{-j}.
 \label{s12_prop1_4}
\end{align}

\item For any integer $1\le k\le k_0$, there is a constant $C_k>0$ such that
\begin{align} \label{s12_prop1_5}
\sup_{0\le t\le 1} \| (D^k U)(t,x)  -\sum_{j=1}^{\infty} (D^k u^j)(t,x-x_j)
\|_{L_x^\infty(\mathbb R^3)} <C_k.
\end{align}

\item For any $j\ge 1$, there exists $r_j>0$  such that
\begin{align} \label{s12_prop1_6}
\sup_{\substack{0\le t\le 1\\ 1\le k\le k_0\\ |x|>r_j}} | (D^k u^j)(t,x)|<
\frac 1 {2^j}.
\end{align}
Therefore
for any $1\le k \le k_0$,
 \begin{align*}
  & \sum_{j=1}^{\infty} (D^k u^j)(t,x-x_j) \notag \\
= & \tilde \eta_k (x)+
 \sum_{j=1}^{\infty} \chi_{<1}(\frac {x-x_j} {r_j}) \cdot (D^k u^j)(t,x-x_j),
 \end{align*}
where $\|\tilde \eta_k\|_{\infty} <2$ and $\chi_{<1}$ is a smooth cut-off function such that
$\chi_{<1}(x)=1$ for $|x|<1$ and $\chi_{<1}(x)=0$ for $|x|\ge 2$.

Consequently by choosing the centers $x_j$ sufficiently far away from each other, we can have
\begin{align}
&\spp \Bigl( \chi_{<1}(\frac {x-x_j} {r_j}) \cdot (D^k u^j)(t,x-x_j) \Bigr)
\notag \\
&\bigcap \spp
\Bigl( \chi_{<1}(\frac {x-x_l} {r_l}) \cdot (D^k u^l)(t,x-x_l) \Bigr)
= \varnothing, \label{s12_prop1_7}
\end{align}
for any $j\ne l$.

\end{enumerate}

\end{prop}

\begin{proof}[Proof of Proposition \ref{s12_prop1}]
This is similar to the proof of Proposition \ref{s7_prop1}. One just need to apply recursively
Lemma \ref{s12_lem1} (and taking into consideration \eqref{s12_prop1_7}) and choose $x_j$ sufficiently
far apart from each other. We omit the routine details.
\end{proof}

\begin{proof}[Proof of Theorem \ref{thm1} for $d=3$]
This is similar to the 2D case. One just needs to repeat the local construction (for $C^1$ use Section 10 and
for $C^m$, $m\ge 2$ use Section 9) to create a sequence of profiles $u^j$. After that apply Proposition
\ref{s12_prop1}.
\end{proof}

\begin{proof}[Proof of Theorem \ref{thm3D_1} and Theorem \ref{thm3D_k}]
We shall only sketch the proof for $C^1$ ($m=1$) case. For $m\ge 2$ one can just use the material
from Section 5 and proceed in a similar fashion as in Section 6. Note that the patching argument for
$C^m$, $m\ge 2$ is actually easier in view of the flow decoupling.

Without loss of generality we may assume $u^{(g)}=0$.

For $j\ge 1$, define $z_j=(0, 0,1+\frac 1j)$. The point $z_j$ will be the center of $j^{\op{th}}$ patch.
Define $x_*= \lim_{j\to \infty} z_j = (0,0, 1)$.

By recursively applying Proposition \ref{prop_hd20}, we can find
(axisymmetric without swirl) stream functions $\psi^j_0 \in C_c^{\infty} (B(z_j,2^{-j-100}))$
such that
\begin{align*}
\| \psi^j_0 \|_{\infty} + \| D^2 \psi^j_0\|_{\infty} <2^{-j}
\end{align*}
and the corresponding $j^{\op{th}}$-patch develops $C^1$-norm inflation in a time interval $<2^{-j}$ (again we omit
the laborious details here since it is essentially a re-statement of Proposition \ref{prop_hd20}
 with more explicit constants).

We then take initial stream function in the form
\begin{align*}
\psi_0 = \sum_{j=1}^{\infty} \psi^j_0,
\end{align*}
and define the corresponding axisymmetric without swirl velocity $u_0$ and $\omega_0$ respectively.
It is then routine to verify the regularity property and the inflation statement of the corresponding solution.
We omit details.
\end{proof}

\end{document}